\documentclass[11pt,oneside,reqno]{amsart}
\usepackage{amssymb}
\usepackage{amsmath}
\usepackage{amscd}
\usepackage{amsfonts}
\usepackage{mathrsfs}
\usepackage{stmaryrd}
\usepackage{longtable}
\usepackage{cancel}
\usepackage{soul}
\usepackage{xcolor}

\usepackage{tikz}
\usepackage{pgf,tikz,pgfplots}
\pgfplotsset{compat=1.14}
\definecolor{ffffff}{rgb}{1,1,1}
\definecolor{ffff00}{rgb}{1,1,0}

\usetikzlibrary{arrows}
\usepackage[T1]{fontenc}
\usepackage[normalem]{ulem}
\allowdisplaybreaks

\newcommand{\nord}{\mbox{\scriptsize ${\circ\atop\circ}$}}

\usepackage{bbm}

\theoremstyle{definition}
\newtheorem{theorem}{Theorem}[section]

\newtheorem{thm}[theorem]{Theorem}
\newtheorem{prop}[theorem]{Proposition}

\newtheorem{defn}[theorem]{Definition}
\newtheorem{lemma}[theorem]{Lemma}
\newtheorem{cor}[theorem]{Corollary}
\newtheorem{prop-def}{Proposition-Definition}[section]

\newtheorem{rema}[theorem]{Remark}

\newtheorem{nota}[theorem]{Notation}

\newcommand{\R}{{\mathbb R}}
\newcommand{\N}{{\mathbb N}}
\newcommand{\C}{{\mathbb C}}
\newcommand{\Z}{{\mathbb Z}}

\newcommand{\h}{{\mathfrak h}}

\newcommand{\Hom}{\textrm{Hom}}
\newcommand{\End}{\textrm{End}}

\newcommand{\one}{\mathbf{1}}
\renewcommand{\d}{\mathbf{d}}
\newcommand{\wt}{\mbox{\rm wt}\ }

\newcommand{\Res}{\text{Res}}

\usepackage{scalerel,stackengine}
\stackMath
\newcommand\reallywidehat[1]{%
\savestack{\tmpbox}{\stretchto{%
  \scaleto{%
    \scalerel*[\widthof{\ensuremath{#1}}]{\kern-.6pt\bigwedge\kern-.6pt}%
    {\rule[-\textheight/2]{1ex}{\textheight}}
  }{\textheight}%
}{0.5ex}}%
\stackon[1pt]{#1}{\tmpbox}%
}
\begin{document}

\setlength{\oddsidemargin}{0cm} \setlength{\evensidemargin}{0cm}
\baselineskip=18pt

\title[Fermionic construction of the $\frac \Z 2$-graded MOSVA and its $\Z_2$-twisted module, II ]{Fermionic construction of the $\frac \Z 2$-graded meromorphic open-string vertex algebra and its $\Z_2$-twisted module, II}
\author{Fei Qi}

\begin{abstract}
This paper continues with Part I. We define the module for a $\frac{\Z} 2$-graded meromorphic open-string vertex algebra that is twisted by an involution and show that the axioms are sufficient to guarantee the convergence of products and iterates of any number of vertex operators. A module twisted by the parity involution is called a canonically $\Z_2$-twisted module. As an example, we give a fermionic construction of the canonically $\Z_2$-twisted module for the $\frac{\Z} 2$-graded meromorphic open-string vertex algebra constructed in Part I. Similar to the situation in Part I, the example is also built on a universal $\Z$-graded non-anti-commutative Fock space where a creation operator and an annihilation operator satisfy the fermionic anti-commutativity relation, while no relations exist among the creation operators or among the zero modes. The Wick's theorem still holds, though the actual vertex operator needs to be corrected from the na\"ive definition by normal ordering using the $\exp(\Delta(x))$-operator in Part I. 
\end{abstract}

\maketitle

\section{Introduction}

This paper is the continuation of \cite{FQ}. Recall that in \cite{FQ}, we defined the $\frac \Z 2$-graded meromorphic open-string vertex algebra ($\frac \Z 2$-graded MOSVA hereafter) and found that it is an appropriate noncommutative generalization of a vertex operator superalgebra. We also constructed an example $V$ that can be viewed as a noncommutative generalization of the vertex operator superalgebra constructed in \cite{T} and \cite{FFR}. 

In this paper, we define the module $(W, Y_W^\theta)$ for the $\frac \Z 2$-graded MOSVA $V$ that is twisted by an involution $\theta: V \to V$ and study some general properties. When $\theta$ is the parity involution of $V$, we omit the upper script $\theta$ and call $(W, Y_W)$ the canonically $\Z_2$-twisted $V$-module. As an example, we construct the canonically $\Z_2$-twisted $V$-module $W$ for the example of $V$ constructed in \cite{FQ}. The example $W$ can be viewed as a noncommutative generalization of the canonically $\Z_2$-twisted module discussed in \cite{FFR}. In terms of the notations of \cite{FFR}, $V$ is the noncommutative generalization of CM$(\Z+1/2)$, while $W$ is the noncommutative generalization of CM$(\Z)$. 

The motivation for constructing such a twisted module comes from Huang's insight in \cite{H-MOSVA} and \cite{H-MOSVA-Riemann}, as well as the experience learned from \cite{Q-2d-space-form}. $\Z_2$-twisted modules first appeared in \cite{FLM} and played an important role in the construction of the Monstrous Moonshine. The current context is motivated by the study of the quantum two-dimensional nonlinear $\sigma$-model with a nonflat target manifold. It is Huang's idea to first use geometric information to construct meromorphic vertex operators satisfying associativity but not commutativity, then study the modules generated by eigenfunctions over a manifold, and finally study the intertwining operators among these modules. In physics, eigenfunctions correspond to the quantum states of a particle that is a degenerated form of a string. Elements of the MOSVA module generated by an eigenfunction (eigenfunction module hereafter) can be viewed as suitable string-theoretic excitations of the quantum states. 

As the first step, Huang gave a Bosonic construction of MOSVAs and modules in \cite{H-MOSVA} and \cite{H-MOSVA-Riemann}. The author studied the example of such Bosonic MOSVAs and the eigenfunction modules over every two-dimensional space form in \cite{Q-2d-space-form}, i.e., two-dimensional Riemannian manifold with constant sectional curvature. One lesson learned from \cite{Q-2d-space-form} is that the MOSVA itself does not carry much geometric information. Indeed, the MOSVA itself can only distinguish space forms with positive and negative curvatures. But the eigenfunction modules carry enough information to distinguish space forms of different curvatures. For the current work, \cite{FQ} gave a fermionic construction of $\frac{Z} 2$-graded MOSVA $V$ that is parallel to the Bosonic MOSVA in \cite{H-MOSVA}. This paper continues to give a fermionic construction of the canonically $\Z_2$-twisted $V$-module $W$ parallel to the Bosonic construction of the left module in \cite{H-MOSVA}. We expected certain submodules constructed over a manifold to carry geometric and physical information, as in \cite{H-MOSVA-Riemann}. 

We now give a brief description of the construction of $W$. Similar to \cite{FQ}, we build the structure of $W$ over a non-anti-commutative integer-graded Fock space where the creation operators satisfy no relations. The correlation function of the modes is no longer a rational function but an algebraic function involving half-integral powers. We still keep the anti-commutativity relation between creation and annihilation operators that would allow us to define a normal ordering. We similarly prove Wick's theorem, expressing the product of two normal-ordered products in terms of other normal-ordered products. The normal-ordered products lead to a na\"ive definition of vertex operator $\bar{Y}_W$. It has to be corrected by the $\exp(\Delta(x))$-operator discussed in Section 5 of \cite{FQ} to get the actual vertex operator $Y_W$. 

One of the main difficulties is the presence of the zero modes. The zero modes do not satisfy the anti-commutativity relation yet still admit a normal ordering defined by contractions. Contrary to the situation in \cite{FQ}, there does not exist a recurrence relation in terms of the generating functions of the positive and negative modes. To handle the computations, one needs to use the combinatorics of 2-shuffles and 3-shuffles extensively. Another difficulty is that there is no reasonable definition of a $D$ operator that plays the role of $L(-1)$ on the current example $W$, unless the zero modes satisfy additional relations. 

It should be emphasized that the non-anti-commutativity of zero modes is crucial. In the Bosonic construction of eigenfunction modules studied in \cite{Q-2d-space-form}, the zero modes act on an eigenfunction via the covariant derivatives. It was surprising for the author to find in \cite{Q-Cov} that over every space form, the covariant derivative of every parallel tensor acts as a scalar multiple on every eigenfunction. Therefore, the zero modes in \cite{Q-2d-space-form} satisfy a noncommutative covariant derivative condition. We expect a similar situation for the zero modes in the fermionic construction that the zero modes will satisfy a non-anticommutative relation imposed by geometry. 

The paper is organized as follows: Section 2 discusses the definition of a module for a $\frac \Z 2$-twisted module that is twisted by an involution. We prove the axioms are sufficient for the products and iterates of any number of vertex operators to converge. Section 3 starts the discussion of the fermionic construction. We define the non-anti-commutative Fock space $W$, the modes, their generating functions, and the normal ordering. Then we introduce the na\"ive vertex operator using the normal ordering. Section 4 studies the product and iterate formulas of the na\"ive vertex operators. A lemma regarding the summation over 2-shuffles and 3-shuffles will be extensively used. We also prove Wick's theorem. From the formula, it is clear that the product and the iterate of na\"ive vertex operators do not coincide. Section 5 then uses the $\exp(\Delta(x))$-operator on $V$ to correct the na\"ive vertex operator, resulting in the actual vertex operator. Then we check all the axioms and show that the products and the iterates agree. 

\noindent\textbf{Acknowledgements. }I would like to thank Yi-Zhi Huang for his long-term support and the discussions of many aspects of the current work. I would also like to thank Tommy Wu-Xing Cai for the discussion on 2-shuffles and 3-shuffles, Alex Feingold for the discussion regarding the definition of $\Delta(x)$, Shashank Kanade for the discussion of the $D$-commutator formula, and Robert McRae for the discussion on the region of convergence of the iterates.



\section{Twisted modules for $\frac \Z 2$-graded MOSVA}

\begin{defn}\label{Def}
Let $(V, Y_V, \one)$ be a $\frac \Z 2$-graded meromorphic open-string vertex algebra. Let $\theta$ be an involution of $V$, i.e., $\theta: V\to V$ is a grading-preserving linear isomorphism satisfying
$$\theta^2 = 1_V, \theta(Y_V(u, x)v) = Y_V(\theta(u), x) \theta(v), u, v\in V. $$
A \textit{$\theta$-twisted $V$-module} is a $\mathbb{C}$-graded vector space 
$W=\coprod\limits_{n\in\C} W_{(n)}$ (graded by {\it weights}) equipped with a {\it twisted vertex operator map}
\begin{eqnarray*}
   Y_W^\theta:  V\otimes W &\to & W[[x^{1/2},x^{-1/2}]]\\
	v\otimes w &\mapsto& Y_W^\theta(v,x)w = \sum_{n\in \Z/2} v_n w x^{-n-1},
  \end{eqnarray*}
an operator $\d_{W}$ of weight $0$ and 
an operator $D_{W}$ of weight $1$, satisfying the 
following axioms:
\begin{enumerate}

\item Axioms for the grading: 
\begin{enumerate}
\item \textit{Lower bound condition}:  When $\text{Re}{(m)}$ is sufficiently negative,
$W_{(m)}=0$. 
\item  \textit{$\mathbf{d}$-grading condition}: for every $w\in W_{(m)}$, $\d_W w = m w$.
\item  \textit{$\mathbf{d}$-commutator formula}: For $u\in V$, 
$$[\mathbf{d}_{W}, Y_W^\theta(u,x)]= Y_W^\theta(\mathbf{d}_{V}u,x)+x\frac{d}{dx}Y_W^\theta(u,x).$$
\end{enumerate}

\item The \textit{identity property}:
$Y_W^\theta(\one,x)=1_{W}$.

\item The \textit{$D$-derivative property} and the  \textit{$D$-commutator formula}: 
For $u\in V$,
\begin{eqnarray*}
\frac{d}{dx}Y_W^\theta(u, x)
&=&Y_W^\theta(D_{V}u, x) 
\end{eqnarray*}

\item {\it Weak associativity with pole-order condition}: For every homogeneous $v_1$ and every $w\in W$, there exists $P\in \mathbb{N}$ such that for every homogeneous $v_2\in V$, 
\begin{align}
    & (x_0+x_2)^{P + |v_1|/2} x_2^{|v_2|/2} Y_W^\theta(v_1, x_0+x_2)Y_W^\theta(v_2, x_2)w \label{Weak-assoc-LHS}\\
    = \ & (x_0+x_2)^{P+|v_1|/2}x_2^{|v_2|/2} Y_W^\theta(Y_V(v_1, x_0)v_2, x_2)w \label{Weak-assoc-RHS}
\end{align} 
as series in $W((x_1, x_2))$. Here 
$$|v| = \left\{\begin{array}{ll} 1 & \text{ if }v\in \coprod_{n\in \Z+1/2}V_{(n)}\\
0 & \text{ if }v\in \coprod_{n\in \Z}V_{(n)}
\end{array}\right.$$ 
is the parity of a homogeneous element $v\in V$. 
\end{enumerate}  

\end{defn}

\begin{rema}
    Similar to usual modules for a MOSVA, for every homogeneous $v\in V$, the $\d$-commutator formula implies every component $v_n$ of $Y_W^\theta(v, x)$ is a homogeneous map, with 
    $$\wt v_n = \wt v - n - 1. $$
    Thus, for every $v, w\in W$, the lowest power of $x$ appearing in $Y_W^\theta(v, x)w$ is bounded below by a constant depending only on $v$ and $w$. In case both $v$ and $w$ are homogeneous, the constant may be chosen as $N - \wt v - \wt w$, where $N$ is a number satisfying $W_{(m)}=0, \text{Re}(m) < N$.
\end{rema}

\begin{rema}
Recall that for a $\frac \Z 2$-graded MOSVA, 
$$V_0 = \coprod\limits_{n\in \Z} V_{(n)}, V_1 = \coprod\limits_{n\in \Z+\frac 1 2} V_{(n)}$$
are respectively called the \textit{even} part and the \textit{odd} part. We may call $|v|$ the \textit{parity} of a homogeneous $v\in V$. Clearly, the map
$$\theta: V \to V, \theta(v)=(-1)^i v, v\in V_i, i=1, 2$$
is an involution of the MOSVA. The twisted module associated with this involution will be called the canonically $\Z_2$-twisted $V$-module. We will omit the $\theta$ in the notation of the twisted vertex operator and simply use $Y_W$. 
\end{rema}

\begin{rema}
We do not require the existence of an operator $D_W$ to satisfy the $D_W$-commutator formula 
$$[D_W, Y_W^\theta(v, x)] = \frac{d}{dx}Y_W^\theta(v, x). $$
Indeed, in the example that will be constructed in this paper, we will show that such operators do not exist unless the zero modes satisfy more relations. See Remark \ref{D-comm-fail}. 
\end{rema}

\subsection{Complex variable formulation} 
For $z\in \C^\times$, we define
$$l_p(z) = \log|z| + i(\arg z + 2\pi p), \arg z \in [0, 2\pi)$$
as the $p$-th branch of the multi-valued logarithmic function $\text{Log } z$. Then the square root function is the collection of the single-valued functions
$$(z^{1/2})_p = e^{\frac 1 2 l_p(z)} = (-1)^p |z|^{1/2} e^{\frac 1 2 i \arg z}$$
for $p\in \Z$ over $\C$ that are discontinuous on $\R_{\geq 0}$. In other words, 
$$(z^{1/2})_p = e^{\frac{1}{2} l_p(z)} = \left\{\begin{array}{ll}
    |z|^{1/2}e^{i\frac 1 2\arg z} & \text{ if }p\text{ is even,} \\
    -|z|^{1/2}e^{i\frac 1 2\arg z} & \text{ if }p\text{ is odd.}
\end{array}\right.$$

Let 
$$W' = \coprod_{n\in \C}W_{(n)}^*$$
be the restricted dual of $W$. Let 
$$\overline{W} = (W')^* = \prod_{n\in \C} W_{(n)}$$
be the algebraic completion. We understand twisted vertex operator $Y_W^\theta$ as a multi-valued function
$$Y_W^\theta: \C^\times \to \Hom_{\C}(V\otimes W, \overline{W})$$
and use the notation 
$$(Y_W^\theta)^p(v, z) = \sum_{n\in \Z}(Y_W^\theta)_n(v) (z^{-n-1})_p = \sum_{n\in \Z}(Y_W^\theta)_n(v) e^{(-n-1)l_p(z)}$$
for the $p$-th branch of $Y_W^\theta$. 

\begin{prop}
    With the presence of all other axioms, the weak associativity with pole order condition is equivalent to the following associativity axiom: For every homogeneous $v_1, v_2\in V$, every $w\in W, w'\in W' = \coprod\limits_{n\in \C}W_{(n)}^*$, there exists an algebraic function 
    $$f(z_1, z_2) = z_1^{|v_1|/2}z_2^{|v_2|/2}\frac{g(z_1, z_2)}{z_1^{q_1}z_2^{q_2}(z_1-z_2)^{q_{12}}}$$
    where $g(z_1, z_2)$ is a polynomial function, $q_1, q_2, q_{12}\in \N$ that depends respectively only on the pairs $(v_1, w), (v_2, w), (v_1, v_2)$, such that both 
    \begin{align}
        \langle w', (Y_W^\theta)^p(v_1, z_1)(Y_W^\theta)^p(v_2, z_2) w\rangle\label{Product-2-Y_W}  
    \end{align}
    and 
    \begin{align}
        \langle w', (Y_W^\theta)^p(Y_V(v_1, z_1-z_2)v_2, z_2)w\rangle \label{Iterate-2-Y_W}
    \end{align}
    converge absolutely respectively in the regions 
    $$|z_1|>|z_2|>0$$ and 
    $$|z_2|>|z_1-z_2|>0, |\arg z_1 - \arg z_2|<\pi/2$$ to the $p$-th branch 
    \begin{align}
    f^{p,p}(z_1, z_2)=(z_1^{|v_1|/2})_p(z_2^{|v_2|/2})_p \frac{g(z_1,z_2)}{z_1^{q_1}z_2^{q_2}(z_1-z_2)^{q_{12}}}. \label{Correlation-2-Y_W}
    \end{align}
    of the multivalued function $f(z_1, z_2)$. 
\end{prop}

\begin{proof}
    Assuming the weak associativity and fixing homogeneous $v_1, v_2\in V$ with parities $|v_1|, |v_2|\in \{0, 1\}$ and homogeneous $w\in W$, we let $F(x_0, x_2)$ to denote the series in $W((x_0, x_2))$ that is equal to both (\ref{Weak-assoc-LHS}) and (\ref{Weak-assoc-RHS}). Note that the lowest power of $x_0$ is bounded below by a constant that depends only on $(v_1, v_2)$, and the lowest power of $x_2$ is bounded below by a constant that depends only on $(v_2, w)$. Therefore, there exists $P_{12}\in \N$ depending only on $(v_1, v_2)$, $P_2\in \N$ depending only on $(v_2, w)$, such that $x_0^{P_{12}}x_2^{P_2} F(x_0, x_2)\in W[[x_0, x_2]]$. 
    
    We evaluate $x_0, x_2$ by complex numbers $z_0, x_2$, then pair $z_0^{P_{12}}z_2^{P_2}F(z_0, z_2)$ by a homogeneous $w'\in W$, then we obtain a complex series $z_0^{q_{12}}z_2^{q_2} \langle w', F(z_0, z_2)\rangle$. We claim that this series is upper-truncated and thus defines a Laurent polynomial function in $z_0, z_2$. This follows from the observation that
    $$\langle w', (v_1)_{n_1}(v_2)_{n_2} w\rangle$$
    is nonzero only when 
    $$\wt w' = \wt v_1 - n_1-1 + \wt v_2 - n_2-1 + \wt w.$$
    Therefore, the total degree of $x_0$ and $x_2$ in (\ref{Weak-assoc-LHS}), i.e., in $x_0^{q_{12}}x_2^{q_2}\langle w', F(x_0, x_2)\rangle$, is upper-truncated. Thus the claim is proved. 
    
    Let $h(z_0, z_2) = z_0^{P_{12}}z_2^{P_2}\langle w', F(z_0, z_2)\rangle$. Then  
    \begin{align*}
        h(z_0, z_2)
        = \ & z_0^{P_{12}}z_2^{P_2}(z_0+z_2)^{P+|v_1|/2}z_2^{|v_2|/2}\langle w', Y_W^\theta(v_1, z_0+z_2)Y_W^\theta(v_2, z_2)w\rangle \\
        = \ & z_0^{P_{12}}z_2^{P_2}(z_0+z_2)^{P+|v_1|/2}z_2^{|v_2|/2}\langle w', Y_W^\theta(Y_V(v_1, z_0)v_2, z_2)w\rangle
    \end{align*}
    are all polynomial functions in $z_0, z_2$. We perform the substitution $z_0 = z_1-z_2$, so as to conclude 
    \begin{align}
        h(z_1-z_2, z_2)
        = \ & (z_1-z_2)^{P_{12}}z_2^{P_2}z_1^{P+|v_1|/2}z_2^{|v_2|/2}\langle w', Y_W^\theta(v_1, z_1)Y_W^\theta(v_2, z_2)w\rangle  \label{cplx-prop-1}\\
        = \ & (z_1-z_2)^{P_{12}}z_2^{P_2}z_1^{P+|v_1|/2}z_2^{|v_2|/2}\langle w', Y_W^\theta(Y_V(v_1, z_1-z_2)v_2, z_2)w\rangle \label{cplx-prop-2}
    \end{align}
    as polynomial functions in $z_1, z_2$. Since 
    $$z_1^{P+|v_1|/2}z_2^{|v_2|/2}\langle w', Y_W^\theta(v_1, z_1)Y_W^\theta(v_2, z_2)w\rangle$$
    is lower-truncated in $z_2$, multiplying the first equation (\ref{cplx-prop-1}) by $(z_1-z_2)^{-P_{12}}z_2^{-P_2}z_1^{-P}$ under the condition $|z_1|>|z_2|>0$ results in function equation
   \begin{align*}
        \frac{h(z_1-z_2, z_2)}{(z_1-z_2)^{P_{12}}z_2^{P_2}z_1^P} = z_1^{|v_1|/2}z_2^{|v_2|/2}\langle w', Y_W^\theta(v_1, z_1)Y_W^\theta(v_2, z_2)w\rangle
    \end{align*}
    where both sides are well-defined single-valued rational functions in $z_1, z_2$. Divide both sides by the $p$-th branch of $z_1^{|v_1|/2}$ and $z_2^{|v_2|/2}$, we conclude 
    \begin{align*}
        \langle w', (Y_W^\theta)^{p}(v_1, z_1)(Y_W^\theta)^p(v_2, z_2)w\rangle = \left(\frac{h(z_1-z_2, z_2)}{(z_1-z_2)^{P_{12}}z_1^{P+|v_1|/2}z_2^{P_2+|v_2|/2}} \right)_{p,p}
    \end{align*}
    Thus (\ref{Product-2-Y_W}) converges absolutely to (\ref{Correlation-2-Y_W}) in the region $|z_1|>|z_2|>0$ where  
    $$f(z_1, z_2) = z_1^{|v_1|/2}z_2^{|v_2|/2}\frac{h(z_1-z_2, z_2)}{(z_1-z_2)^{P_{12}}z_1^{P+|v_1|}z_2^{P_2+|v_2|}}$$
    with $g(z_1, z_2) = h(z_1-z_2, z_2), q_1 = P+|v_1|, q_2= P_2 + |v_2|$. 
        
    Similarly, since 
    $$(z_0+z_2)^{|v_1|/2}z_2^{|v_2|/2}\langle w', Y_W^\theta(Y_V(v_1, z_0)v_2, z_2)w\rangle$$
    is lower-truncated in $z_0$, multiplying the second equation (\ref{cplx-prop-2}) by $(z_1-z_2)^{-P_{12}}z_2^{-P_2}z_1^{-P}$ under the condition $|z_2|>|z_1-z_2|>0$ resulting in function equation
    \begin{align*}
        \frac{h(z_1-z_2, z_2)}{(z_1-z_2)^{P_{12}}z_2^{P_2}z_1^P} = z_1^{|v_1|/2}z_2^{|v_2|/2}\langle w', Y_W^\theta(Y_V(v_1, z_1-z_2)v_2, z_2)w\rangle
    \end{align*}
    where both sides are well-defined single-valued rational functions in $z_1, z_2$. Divide both sides by the $p$-th branch of $z_1^{|v_1|/2}$ and $z_2^{|v_2|/2}$. Note that
    $$z_1^{|v_1|/2} = (z_2 + (z_1-z_2))^{|v_1|/2} = z_2^{|v_1|/2} \left(1 + \frac{z_1-z_2}{z_2}\right)^{|v_1|/2}$$
    is supposed to be expanded as a power series in $(z_1-z_2)/z_2$. Thus, $0<|z_1-z_2|<|z_2|$ for every fixed $z_2\in \C$. In particula, we see that 
    $$z_1 = z_2\cdot \frac{z_1}{z_2} = 1 + \frac{z_1-z_2}{z_2}$$
    lies the disk centered at $z=1$ with radius $1$ on the complex plane. Thus necessarily, $|\arg z_1 - \arg z_2|<\pi/2$. Therefore, we conclude that 
    \begin{align*}
        \langle w', (Y_W^\theta)^p (Y_V(v_1, z_0)v_2, z_2)w\rangle = \left(\frac{h(z_1-z_2, z_2)}{(z_1-z_2)^{P_{12}}z_1^{P+|v_1|/2}z_2^{P_2+|v_2|/2}} \right)_{p,p}
    \end{align*}
    converges absolutely in the region $|z_2|>|z_1-z_2|>0, |\arg z_1 - \arg z_2|<\pi/2$ to (\ref{Correlation-2-Y_W}) with the same choice of $f$. 

    Assuming the associativity axiom, we may obtain the weak associativity axiom by reversing the above arguments in this proposition. We shall not repeat the details here. 
\end{proof}

\begin{rema}
    The requirement of $|\arg z_1 - \arg z_2|<\pi/2$ was pointed out by Robert McRae in the study of twisted modules for a VOA. 
\end{rema}

\begin{rema}
    \cite{H-Twist} requires an equivariance condition on the twisted vertex operator, namely, for $v\in V, p\in \Z$, 
    $$(Y_W^\theta)^p(\theta v, z) = (Y_W^\theta)^{p+1}(v, z). $$
    Although the example we construct in Section 3 - 5 (trivially) satisfy this condition, we decide not to place the equivariance condition in the axioms. 
\end{rema}

\subsection{Products and iterates of any numbers of vertex operators}

\begin{prop}\label{n-prod-prop}
    Let $(W, Y_W^\theta, \d_W, D_W)$ be a $\theta$-twisted $V$-module, then for every $n\in \N$, every $v_1, ..., v_n\in V$ of parities $|v_1|, ..., |v_n|\in \{0, 1\}$, every $w\in W, w'\in W'$, there exists an algebraic function 
    $$f(z_1, ..., z_n) = \prod_{i=1}^n z_i^{|v_i|/2} \frac{g(z_1, ..., z_n)}{\prod\limits_{i=1}^n z_i^{q_i} \prod\limits_{1\leq i < j \leq n}(z_i-z_j)^{q_{ij}}}, $$
    such that for every $p\in \Z$
    $$\langle w', (Y_W^\theta)^p(v_1, z_1) \cdots (Y_W^\theta)^p(v_n, z_n)w\rangle$$
    converges absolutely in the region $|z_1|>\cdots > |z_n|$ to the $p$-th branch 
    $$f_{p, ..., p}(z_1, ..., z_n) = \prod_{i=1}^n e^{\frac{|v_i|} 2 l_p(z_i)} \frac{g(z_1, ..., z_n)}{\prod\limits_{i=1}^n z_i^{q_i} \prod\limits_{1\leq i < j \leq n}(z_i-z_j)^{q_{ij}}},$$
    of $f(z_1, ..., z_n)$. 
\end{prop}

\begin{proof}
    It suffices to modify the weak associativity arguments of Theorem 3.4 in \cite{Q-Mod} by modifying the products of the series as 
    $$(x_1+x_3)^{|v_1|/2}(x_2+x_3)^{|v_2|/2}x_3^{p_3/2}Y_W^\theta(v_1, x_1+x_3)Y_W^\theta(v_2, x_2+x_3)Y_W^\theta(v_3, x_3)w. $$
    The rest of the arguments follow verbatim with this modification. We shall not repeat the details here. 
\end{proof}

\begin{prop}\label{n-iter-prop}
With same notations as Proposition \ref{n-prod-prop}, the series
$$\langle w', Y_W^\theta(Y_V(\cdots Y_V(Y_V(u_1, z_1-z_2)u_2, z_2-z_3)u_3 \cdots, z_{n-1}-z_n )u_n, z_n)w \rangle$$
converges absolutely in the region 
\begin{equation}\label{IterRegion}
\left\{(z_1, ..., z_n)\in \C^n: \begin{aligned}
&|z_n|>|z_i-z_n|>0, |\arg z_i - \arg z_n|< \pi/2, i = 1, ..., n; \\
&|z_i-z_{i+1}|>|z_j-z_i|>0, 1\leq j < i \leq n-1 \end{aligned}\right\}
\end{equation}
to $f_{p, p}(z_1, ..., z_n)$. 
\end{prop}

\begin{proof}
    It suffices to modify the analytic continuation arguments of Theorem 4.10 in \cite{Q-Mod} to the series 
    $$\prod_{i=1}^n z_i^{|v_i|/2} \langle w', Y_W^\theta(Y_V(\cdots Y_V(Y_V(u_1, z_1-z_2)u_2, z_2-z_3)u_3 \cdots, z_{n-1}-z_n )u_n, z_n)w \rangle.$$
    The rest of the arguments follow verbatim with this modification. We shall not repeat the details here. 
\end{proof}

\section{Fermionic construction of vertex operators}

\subsection{The algebra $N(\hat{h}_{\Z})$ and its induced modules $W$}

Let $\h = \C^{2M}$ be a vector space with a nondegenerate symmetric bilinear form $(\cdot, \cdot)$, such that 
$$\h = \mathfrak{p}\oplus \mathfrak{q}$$ 
is a polarization into maximal isotropic subspaces $\mathfrak{p}, \mathfrak{q} \simeq \C^{M}$. 

We consider the affinization 
$$\hat{\h}_{\Z} = \h \otimes \C[t, t^{-1}] \oplus \C \mathbf{k}$$
As a vector space, $\hat{\h}_{\Z} = \hat{\h}_{\Z}^+ \oplus \hat{\h}_{\Z}^- \oplus \hat{\h}_{\Z}^0$, where 
$$\hat{\h}_{\Z}^\pm = t^{\pm 1}\C[t^{\pm 1}], \hat{\h}_{\Z}^0 =\h \otimes t^0 \oplus \C \mathbf{k}$$
Let 
$$T(\hat{\h}_{\Z}) = \C \oplus  \hat{\h}_{\Z} \oplus (\hat{\h}_{\Z})^{\otimes 2} \oplus  \cdots. $$
be the tensor algebra of $\hat{\h}_{\Z}$. Consider the quotient $N(\hat{\h}_{\Z})$ of $T(\hat{\h}_{\Z})$ by the two-sided ideal $J$ generated by 
\begin{eqnarray}
&(a\otimes t^{m})\otimes (b\otimes t^{n})
+ (b\otimes t^{n})\otimes (a\otimes t^{m})\label{ideal-1}
-m(a, b)\delta_{m+n, 0}\mathbf{k},&\\
&(a\otimes t^{p})\otimes (b\otimes t^{0})
+ (b\otimes t^{0})\otimes (a\otimes t^{p})\label{ideal-2}&\\
&(a\otimes t^{q})\otimes \mathbf{k}-\mathbf{k}\otimes (a\otimes t^{q}), & \label{ideal-3}
\end{eqnarray}
for $a, b\in \mathfrak{h}$, $m\in \Z_{+}$, $n\in -\Z_{+}$, $p\in \Z\setminus\{0\}, q\in \Z$. 

\begin{rema}
    Note that we do not have any relations between $a\otimes t^{m}$ and $b\otimes t^{n}$ if $m,n$ are both positive, or both negative, or both zero. 
\end{rema}

\begin{prop}\label{PBW}
    As a vector space, $N(\hat{\h}_{\Z})$ is isomorphic to $T(\hat{\h}_{\Z}^-)\otimes T(\hat{\h}_{\Z}^+) \otimes T(\h\otimes t^0)\otimes T(\C\mathbf{k})$
    where $T(\hat{\h}_{\Z}^-)$, $T(\hat{\h}_{\Z}^+)$, $T(\h\otimes t^0)$ and $T(\C\mathbf{k})$ are tensor algebras of $\hat{\h}_{\Z}^-$, $\hat{\h}_{\Z}^+$, $\h\otimes t^0$ and $\C\mathbf{k}$, respectively. 
\end{prop}

\begin{proof}
    View $T(\hat{\h}_{\Z}^-)\otimes T(\hat{\h}_{\Z}^+) \otimes  T(\h\otimes t^0)\otimes T(\C\mathbf{k})$ as a subspace of $T(\hat{\h}_{\Z})$. We first show that
    $$T(\hat{\h}_{\Z}) = T(\hat{\h}_{\Z}^-)\otimes T(\hat{\h}_{\Z}^+)  \otimes  T(\h\otimes t^0) \otimes T(\C\mathbf{k}) + J. $$
    Clearly, $T(\hat{\h}_{\Z})$ is spanned by elements $u_1\otimes \cdots u_n$, where each $u_i$ is either of the form $h\otimes t^{m}, m\in \Z$ or $\mathbf{k}$. Using (\ref{ideal-1}), (\ref{ideal-2}) and (\ref{ideal-3}), we may arrange the terms $h\otimes t^{m}$ with $m<0$ to the left, followed by the terms $h\otimes t^m$ with $m>0$, then followed by the terms $h\otimes t^0$, then followed by $\mathfrak{k}$, by modifying an element in $J$. So each $u_1\otimes \cdots \otimes u_n$ in the spanning set can be a sum of an element of the form
    $$(h_1\otimes t^{m_1}) \otimes \cdots \otimes (h_j\otimes t^{m_j}) \otimes (h_{j+1}\otimes t^{m_{j+1}}) \otimes \cdots \otimes (h_k\otimes t^{m_k}) \otimes (h_{k+1}\otimes t^0) \otimes \cdots \otimes (h_l\otimes t^0)\otimes \mathbf{k} \otimes \cdots \otimes \mathbf{k}$$
    and an element in the ideal $J$.

    We now show that the sum is indeed direct, i.e., the intersection of $T(\hat{\h}_{\Z}^-)\otimes T(\hat{\h}_{\Z}^+) \otimes T(\h \otimes t^0)\otimes T(\C\mathbf{k})$ and $J$ is trivial. We follow the same procedure as Jacobson's proof of the linear independence part Poincar\'e-Birkhoff-Witt theorem (see \cite{G-PBW} for an exposition, see also \cite{H-MOSVA}) by constructing a map $L: T(\hat{\h}_{\Z})\to T(\hat{\h}_{\Z})$ such that 
    $$L(J) = 0, L|_{T(\hat{\h}_{\Z}^-)\otimes T(\hat{\h}_{\Z}^+) \otimes T(\h \otimes t^0)\otimes T(\C\mathbf{k}))} = 1. $$
    Then every $v$ in the intersection must satisfy both $L(v) = 0$ and $L(v) = v$, which implies the conclusion. 

    Consider a basis $e_1, ..., e_{2M}$ of $\h$. Then $T(\hat{\h}_{\Z})$ has a basis $u_1\otimes \cdots \otimes u_n$, where each $u_i$ is either of the form $e_{k_i}\otimes t^{m_i}$ or of the form $\mathbf{k}$. A pair of elements $(u_i, u_j)$ with $i<j$ is \textit{of wrong order}, if either of the following cases happen: (i) $u_i = e_{k_i}\otimes t^{m_i}, u_j = e_{k_j} \otimes t^{m_j}$ for some $k_i, k_j\in \{1, ..., 2M\}$ and some $m_i, m_j\in \Z$ satisfying one of the following: (a) $m_i > 0, m_j = 0$; (b) $m_i > 0, m_j < 0$; (c) $m_i = 0, m_j < 0$; (ii) $u_i = \mathbf{k}$, $u_j =  e_{k_j}\otimes t^{m_j}$ for some $k_i\in \{1, ..., 2M\}, m_j \in \Z$. We define the \textit{defect} of $u_1\otimes \cdots \otimes u_n$ as the number of pairs of elements $(u_i, u_j)$ with $i<j$ that are of wrong order. 
    
    We define $L$ inductively on the basis elements $u_1 \otimes \cdots \otimes u_n$ by the following rule: $L$ acts on the basis elements with zero defect as identity. In particular, $L$ acts on the degree-zero tensors in $\C$ and the degree-1 tensors in $\hat{\h}_{\Z}$ as identity. For other basis elements, we define
    \begin{align*}
        & L(u_1\otimes \cdots \otimes (e_{k_i}\otimes t^{m_i}) \otimes (e_{k_{i+1}}\otimes t^{m_{i+1}}) \otimes  \cdots \otimes u_n) \\
        = \ &-L(u_1\otimes \cdots \otimes (e_{k_{i+1}}\otimes t^{m_{i+1}})\otimes (e_{k_i}\otimes t^{m_i}) \otimes \cdots \otimes u_n)\\
         & + m_{i}(e_{k_{i}}, e_{k_{i+1}}) \delta_{m_i+m_{i+1},0}L(u_1 \otimes \cdots \otimes u_{i-1} \otimes \mathbf{k} \otimes u_{i+2} \otimes \cdots \otimes u_n)
    \end{align*}
    whenever $m_{i}>0$ and $m_{i+1}<0$; we define
    \begin{align*}
        & L(u_1\otimes \cdots \otimes (e_{k_i}\otimes t^{m_i}) \otimes (e_{k_{i+1}}\otimes t^{m_{i+1}}) \otimes  \cdots \otimes u_n) \\
        = \ &-L(u_1\otimes \cdots \otimes (e_{k_{i+1}}\otimes t^{m_{i+1}})\otimes (e_{k_i}\otimes t^{m_i}) \otimes \cdots \otimes u_n)
    \end{align*}
    whenever $m_i > 0$ and $m_{i+1} = 0$, or $m_i = 0$ and $m_{i+1}<0$ we define
    \begin{align*}
         & L(u_1\otimes \cdots \otimes \mathbf{k} \otimes (e_{k_{i+1}}\otimes t^{m_{i+1}}) \otimes  \cdots \otimes u_n) \\
        = \ & L(u_1\otimes \cdots \otimes (e_{k_{i+1}}\otimes t^{m_{i+1}})\otimes \mathbf{k} \otimes \cdots \otimes u_n) 
    \end{align*}
    whenever $m_{i+1}\in \Z$. 
    
    We use induction to show that $L$ is well-defined. For basis elements of defect 0, the definition clearly has no ambiguities. For basis element with defect 1, if the pair $(u_i, u_j)$ is of wrong order, then necessarily, $j = i+1$. The definition would then have no ambiguities. Assume that $L$ has no ambiguity for all basis elements with smaller defects and for all tensors of lower degrees. Consider a basis element with defect at least 2. Necessarily, there exists two indices $i<j$ such that the pairs $(u_i, u_{i+1})$ and $(u_j, u_{j+1})$ are both of wrong order. 
    
    If $i+1<j$, then we may express $L(u_1 \otimes \cdots \otimes u_i\otimes u_{i+1} \otimes \cdots \otimes u_j \otimes u_{j+1} \otimes \cdots \otimes u_n)$ as a unambiguous sum of $L(u_1\otimes \cdots \otimes u_{i+1} \otimes u_i \otimes \cdots \otimes u_{j+1}\otimes u_j \otimes \cdots \otimes u_n)$, which is of lower defect, and the images of $L$ on lower-degree tensors. From the induction hypothesis, they are all well-defined. For example, if $u_s = e_{k_s} \otimes t^{m_s}$ for $s = i, i+1, j, j+1$ with $m_{i}, m_{j} > 0, m_{i+1}, m_{j+1}< 0$, then 
    \begin{align*}
        & L(u_1\otimes \cdots \otimes u_i \otimes u_{i+1} \otimes \cdots \otimes u_{j}\otimes u_{j+1} \otimes \cdots \otimes u_n)\\
        = \ & - L(u_1\otimes \cdots \otimes u_{i+1} \otimes u_{i} \otimes \cdots \otimes u_{j}\otimes u_{j+1} \otimes \cdots \otimes u_n) \\ 
        & + m_i(e_{k_i}, e_{k_{i+1}})\delta_{m_i+m_{i+1}, 0} L(u_1\otimes \cdots \otimes u_{i-1} \otimes \mathbf{k} \otimes u_{i+2} \otimes \cdots \otimes u_{j}\otimes u_{j+1} \otimes \cdots \otimes u_n)\\
        = \ &  L(u_1\otimes \cdots \otimes u_{i+1} \otimes u_{i} \otimes \cdots \otimes u_{j+1}\otimes u_j \otimes \cdots \otimes u_n) \\ 
        & - m_j(e_{k_j}, e_{k_{j+1}})\delta_{m_j+m_{j+1}, 0} L(u_1\otimes \cdots \otimes u_{i-1} \otimes \mathbf{k} \otimes u_{i+2} \otimes \cdots \otimes u_{j-1} \otimes \mathbf{k}\otimes u_{j+2} \otimes \cdots \otimes u_n)\\
        & - m_i(e_{k_i}, e_{k_{i+1}})\delta_{m_i+m_{i+1}, 0} L(u_1\otimes \cdots \otimes u_{i-1} \otimes \mathbf{k} \otimes u_{i+2} \otimes \cdots \otimes u_{j}\otimes u_{j+1} \otimes \cdots \otimes u_n)\\
        & + m_i(e_{k_i}, e_{k_{i+1}})\delta_{m_i+m_{i+1}, 0}m_j(e_{k_j}, e_{k_{j+1}})\delta_{m_j+m_{j+1}, 0} \\
        & \qquad \cdot L(u_1\otimes \cdots \otimes u_{i-1} \otimes \mathbf{k} \otimes u_{i+2} \otimes \cdots \otimes u_{j-1} \otimes \mathbf{k}\otimes u_{j+2} \otimes \cdots \otimes u_n)
    \end{align*}
    Starting with $j$ instead of $i$ results in the same expression. So $L(u_1\otimes \cdots \otimes u_n)$ is unambiguously defined with the current choice of $u_s$. For other choices of $u_s$ that involve $e_{k_s\otimes t^0}$ or $\mathbf{k}$, we can proceed with a similar argument to get the same conclusion. Therefore, the inductive step is proved when $i+1 < j$. 

    If $i+1 = j$, then we may also express $L(u_1 \otimes \cdots \otimes u_i \otimes u_{i+1} \otimes u_{i+2} \otimes \cdots \otimes u_n)$ as an unambiguous sum of $L(u_1 \otimes \cdots u_{i+2} \otimes u_{i+1} \otimes u_i \otimes \cdots \otimes u_n)$, which is lower defect, and the images of $L$ on lower-degree tensors. For example, let $u_i = \mathbf{k}$, $u_t = e_{k_t}\otimes t^{m_{t}}$ for $t=i+1, i+2$ with $m_{i+1}>0, m_{i+2}<0$. If we first swap $u_i, u_{i+1}$, then swap $u_i, u_{i+2}$, finally swap $(u_{i+1}, u_{i+2})$, then we get
    \begin{align*}
        & L(u_1 \otimes \cdots \otimes u_i \otimes u_{i+1} \otimes u_{i+2} \otimes \cdots \otimes u_n) \\
        = \ & L(u_1 \otimes \cdots \otimes u_{i+1} \otimes u_{i} \otimes u_{i+2} \otimes \cdots \otimes u_n)\\
        = \ & L(u_1 \otimes \cdots \otimes u_{i+1} \otimes u_{i+2} \otimes u_{i} \otimes \cdots \otimes u_n)\\
        = \ & L(u_1 \otimes \cdots \otimes u_{i+2} \otimes u_{i+1} \otimes u_{i} \otimes \cdots \otimes u_n) \\
        & + m_{i+1}(e_{k_{i+1}}, e_{k_{i+2}}) \delta_{m_{i+1}+m_{i+2}, 0} L(u_1 \otimes \cdots \otimes u_{i-1} \otimes \mathbf{k} \otimes u_i \otimes  \cdots \otimes u_n)
    \end{align*}
    On the other hand, if we first swap $u_{i+1}, u_{i+2}$, then swap $u_i, u_{i+2}$, finally swap $u_i, u_{i+1}$, then we get
    \begin{align*}
        & L(u_1 \otimes \cdots \otimes u_i \otimes u_{i+1} \otimes u_{i+2} \otimes \cdots \otimes u_n) \\
        = \ & L(u_1 \otimes \cdots \otimes u_{i} \otimes u_{i+2} \otimes u_{i+1} \otimes \cdots \otimes u_n)\\
        & + m_{i+1}(e_{k_{i+1}}, e_{k_{i+2}}) \delta_{m_{i+1}+m_{i+2}, 0} L(u_1 \otimes \cdots \otimes u_{i-1} \otimes u_i \otimes \mathbf{k}  \otimes  \cdots \otimes u_n)\\
        = \ & L(u_1 \otimes \cdots \otimes u_{i+2} \otimes u_{i} \otimes u_{i+1} \otimes \cdots \otimes u_n)\\
        & + m_{i+1}(e_{k_{i+1}}, e_{k_{i+2}}) \delta_{m_{i+1}+m_{i+2}, 0} L(u_1 \otimes \cdots \otimes u_{i-1} \otimes \mathbf{k} \otimes u_i   \otimes  \cdots \otimes u_n)\\
        = \ & L(u_1 \otimes \cdots \otimes u_{i+2} \otimes u_{i+1} \otimes u_{i}  \otimes \cdots \otimes u_n)\\
        & + m_{i+1}(e_{k_{i+1}}, e_{k_{i+2}}) \delta_{m_{i+1}+m_{i+2}, 0} L(u_1 \otimes \cdots \otimes u_{i-1} \otimes \mathbf{k} \otimes u_i   \otimes  \cdots \otimes u_n)
    \end{align*}
    Starting with $j$ instead of $i$ results in the same expression. So $L(u_1\otimes \cdots \otimes u_n)$ is unambiguously defined with the current choice of $u_s$. For other possible choices of $u_s$, we can proceed with a similar argument to get the same conclusion. Therefore, the inductive step is proved when $i+1 = j$. 

    Therefore, $L$ is a well-defined linear map. The conclusion then follows.    
\end{proof}

\subsection{The vector space $W$, the modes and the generating function of modes}

Let $U$ be a $T(\h)$-module. We let $\hat{\h}_{\Z}^+$ acts on $U$ trivially, $\h\otimes t^0$ acts on $U$ as $\h$ acts on $U$, and $\mathbf{k}$ acts by the scalar $l\in \C$. Then $U$ is a module for the subalgebra $T(\hat{\h}_{\Z}^+) \otimes T(\h\otimes t^0) \otimes T(\C \mathbf{k})$ of $N(\hat{\h}_{\Z})$. 
Let 
$$W = N(\hat{\h}_{\Z})\otimes_{T(\hat{\h}_{\Z}^+) \otimes T(\h\otimes t^0) \otimes T(\C \mathbf{k})} U$$
be the induced module. As a vector space, $W$ is isomorphic to $T(\hat{\h}_{\Z}^-)\otimes U$. For $h\in \h, m\in \N$, we denote the action of $h\otimes t^{m}$ by $h(m)$. We refer $h(m)$ as positive modes if $m> 0$, negative modes if $m<0$, zero modes if $m=0$. We use the brace brackets $\{\cdot, \cdot\}$ to denote the anti-commutator of the modes. Clearly, for $a, b\in \h, m,n\in \N$,
\begin{align}
    \{a(m), b(-n)\} &= a(m)b(-n) + b(-n)a(m) = (a, b) \delta_{mn} \label{anticom-modes}
\end{align}
Also, it follows from Proposition \ref{PBW} that $W$ is spanned by
\begin{align}
    h_1(-m_1) \cdots h_r(-m_r)u, h_1, ..., h_r \in \h, u\in U, m_1, ..., m_r\in \N, r\in \N. \label{V-basis}
\end{align}
We assign the integer
$$m_1 + \cdots + m_r$$
as the weight of (\ref{V-basis}) and define the operator $\d$ accordingly. Clearly, for every $h\in \h$ and $n\in \Z$, the mode $h(n)$ is a homogeneous operator with weight $-n$. 

\begin{rema}
    The weight we assign on $W$ gives the canonical $\N$-grading on $W$. Note that in case $W$ is the 1-dimensional trivial $T(\h)$-module, \cite{FFR} assigns the weight $m_1 + \cdots + m_r - M/8$ that gives the conformal grading. The reader should note that conformal grading does not make sense in the current setting since a conformal element is not required to exist. 
\end{rema}

Although we did not assume any relations among positive modes on $V$, it turns out that they are all anti-commutative. More precisely, 
\begin{prop}\label{pos-mode-prop}
    For $a,b\in \h$, $m,n\in \Z_+$, 
    $$\{a(m), b(n)\} = a(m)b(n) + b(m)a(n) = 0$$
\end{prop}

\begin{proof}
    We check the action of the anti-commutator on the basis elements and show that 
    \begin{align}
        & a(m)b(n) h_1(-m_1)\cdots h_r(-m_r) u  \label{pos-mode-1}\\
        = \ &  -b(n)a(m) h_1(-m_1)\cdots h_r(-m_r) u\label{pos-mode-2}
    \end{align}
    The identity clearly holds when $r= 0$. Assume the identity holds for all smaller $r$. We compute (\ref{pos-mode-1}) as follows
    \begin{align*}
        (\ref{pos-mode-1}) = \ & a(m) (b, h_1) \delta_{n, m_1} h_2(-m_2) \cdots h_r(-m_r)u  - a(m)h_1(-m_1) b(n)  h_2(-m_2) \cdots h_r(-m_r)u\\
        = \ & a(m) (b, h_1) \delta_{n, m_1} h_2(-m_2) \cdots h_r(-m_r)u  - (a, h_1)\delta_{m,m_1} b(n)  h_2(-m_2) \cdots h_r(-m_r)u\\
        & + h_1(-m_1) a(m)b(n)  h_2(-m_2) \cdots h_r(-m_r)u\\
        (\ref{pos-mode-2})= \ & -b(n) (a,h_1)\delta_{m,m_1} h_2(-m_2)\cdots h_r(-m_r)u + b(n) h_1(-m_1) a(m) h_2(-m_2)\cdots h_r(-m_r)u \\
        = \ & -b(n) (a,h_1)\delta_{m,m_1} h_2(-m_2)\cdots h_r(-m_r)u  + (b, h_1)\delta_{n,m_1} a(m) h_2(-m_2)\cdots h_r(-m_r)u \\
        & - h_1(-m_1) a(m) b(n) h_2(-m_2)\cdots h_r(-m_r)u 
    \end{align*}
    The conclusion then follows from the induction hypothesis. 
\end{proof}


For fixed $h\in \h$, we consider the series 
$$h(x) = \sum_{n\in \Z} h(n) x^{-n-1/2}\in \End(W)[[x, x^{-1}]]$$
that is the generating functions of all modes associated with $h$. 
\begin{align*}
    h(x)^+ &= \sum_{n> 0} h(n) x^{-n-1/2}, \\
    h(x)^- &= \sum_{n> 0} h(-n) x^{n-1/2}, \\
    h(x)^0 &= h(0) x^{-1/2} 
\end{align*}
be the partial series consisting of positive modes, negative modes and the zero mode, respectively. Contrary to the notation convention in \cite{LL}, $h(x)^+ + h(x)^0$ is the singular part of $h(x)$, while $h(x)^-$ is the regular part of $h(x)$. For every $n\in\Z$, we use the notation 
$$(x^{-n-1/2})^{(m)} = \frac{1}{m!}\frac{\partial^m}{\partial x^m} x^{-n-1/2} = \binom{-n-1/2}{m}x^{-n-m-1/2}. $$
We also use the notation
$$h^{(m)}(x) = \frac{1}{m!}\frac{\partial^m}{\partial x^m} h(x) = \sum_{n\in \Z} h(n)(x^{-n-1/2})^{(m)} = \sum_{n\in \Z} \binom{-n-1/2}{m} h(n)x^{-n-m-1/2}$$
Similarly, we have
\begin{align*}
    h^{(m)}(x)^+ &= \frac{1}{m!}\frac{\partial^m}{\partial x^m} h(x)^+ = \sum_{n > 0} h(n) (x^{-n-1/2})^{(m)}= \sum_{n > 0} \binom{-n-1/2}{m} h(n)x^{-n-m-1/2}, \\
    h^{(m)}(x)^- &= \frac{1}{m!}\frac{\partial^m}{\partial x^m} h(x)^- =\sum_{n > 0} h(-n) (x^{n-1/2})^{(m)}= \sum_{n > 0} \binom{n-1/2}{m} h(n)x^{n-m-1/2}, \\
    h^{(m)}(x)^0 &= \frac{1}{m!}\frac{\partial^m}{\partial x^m} h(x)^0 = \binom{-1/2}{m} h(0) x^{-m-1/2}. 
\end{align*}

\begin{rema}
    Note that the regular-singular decomposition of $h^{(m)}(x)$ does not work compatibly with $\pm$ and $0$ superscripts.     
\end{rema}

\begin{rema}
    Note also that $h(n)$ now acts on $W$, and $h(x)$ is the generating function of the modes $h(n)$ acting on $W$. We will use $h(x)_V$ to denote the generating function of the modes on $V$ defined in Part I. We will trust the readers of the both parts not to confuse on these notations. 
\end{rema}

\begin{prop}
    Let 
    \begin{align*}
        f(x,y) &= x^{-1/2}y^{1/2}(x-y)^{-1}, 
        f_{mn}(x,y) = \frac{1}{m!n!}\frac{\partial^{m+n}}{\partial x^m \partial y^n}f(x,y), m,n\geq 0. 
    \end{align*}
    Then for $a, b\in \h$, the anticommutator of $\{a^{(m)}(x)^+, b^{(n)}(y)^-\}$ is given by
    \begin{align}
        \{a^{(m)}(x)^+, b^{(n)}(y)^-\} &= (a,b) \iota_{xy}f_{mn}(x,y). \label{anticomm-pm}
    \end{align}
    Here $\iota_{xy}$ expands negative powers of $x-y$ as a power series in $y$, i.e., 
    $$\iota_{xy}\left(\frac{1}{(x-y)^t}\right) = \sum_{i=0}^\infty \binom{-t}i x^{t-i}(-y)^i = \sum_{i=0}^\infty \binom{t+i-1}{i} x^{t-i}y^i. $$
\end{prop}
\begin{proof}
    We first work out the case $m=n=0$. In this case, we compute the anti-commutator as follows:
    \begin{align*}
        \{a(x)^+, b(y)^-\} & = \sum_{i, j > 0} \{a(i), b(-j)\}x^{-i-1/2}y^{j-1/2} =\sum_{i, j > 0} (a, b)\delta_{ij}x^{-i-1/2}y^{j-1/2}\\
        & = \sum_{i> 0} (a, b)x^{-i-1/2}y^{i-1/2} = x^{-1/2}y^{1/2}(a, b) \sum_{i\geq 1} x^{-i}y^{i-1}\\
        & =  x^{-1/2}y^{1/2}(a, b) \iota_{xy}(x-y)^{-1} = (a, b) \iota_{xy}f(x,y). 
    \end{align*}
    The conclusion for general $m, n$ follows from taking partial derivatives. 
\end{proof}

\begin{rema}
    The function $f(x,y)$ here is an algebraic function that is different from the rational function $f(x,y)=(x-y)^{-1}$ in Part II. We will trust the readers of the both parts not to confuse on these notations. 
\end{rema}

\subsection{Normal ordering and the na\"ive vertex operator}
\begin{defn}\label{nord-defn-modes}
Let $h_1, ..., h_r\in \h$, $m_1, ..., m_r\in \Z$. We define the \textit{normal ordering} of the product $h_1(m_1), ..., h_r(m_r)$ of the modes by
\begin{align*}
    & \nord h_1(m_1)\cdots h_r(m_r)\nord  \\
    = \ & (-1)^\sigma h_{\sigma(1)}(m_{\sigma(1)})\cdots h_{\sigma(\mu)}(m_{\sigma(\mu)})\\
    & \qquad  \cdot h_{\sigma(\mu+1))}(m_{\sigma(\mu+1)}) \cdots h_{\sigma(\nu)}(m_{\sigma(\nu)}) \\
    & \qquad  \cdot \nord h_{\sigma(\nu+1)}(0) \cdots h_{\sigma(r)}(0)\nord 
\end{align*}
where $\sigma$ is the unique permutation of $\{1, ..., r\}$ such that
\begin{eqnarray}
    m_{\sigma(1)} < 0, ..., m_{\sigma(\mu)} < 0, \ & m_{\sigma(\mu+1)} > 0, ..., m_{\sigma(r)} > 0, \ & m_{\sigma(\nu+1)} = 0, ..., m_{\sigma(r)} = 0\\
    \sigma(1)< \cdots < \sigma(\mu), \ & \sigma(\mu+1)<\cdots < \sigma(\nu), \ & \sigma(\nu+1) < \cdots < \sigma(r). \label{shuffle-defn}
\end{eqnarray}
$(-1)^\sigma$ is the parity of the permutation $\sigma$, and the normal ordering of the zero modes are defined recursively by 
\begin{align}
    \nord a_1(0) \cdots a_r(0) \nord = a_1(0) \nord a_2(0) \cdots a_r(0) \nord + \sum_{i=2}^r \frac{(-1)^{i+1}}2 (a_1, a_i) \nord a_2(0) \cdots \reallywidehat{a_i(0)} \cdots a_r(0)\nord. \label{nord-zero-modes}
\end{align}
For example:
\begin{align*}
    \nord a_1(0) a_2(0)\nord &= a_1(0) a_2(0) - \frac 1 2 (a_1, a_2)\\
    \nord a_1(0) a_2(0) a_3(0) \nord &= a_1(0) a_2(0) a_3(0) - \frac 1 2 (a_2, a_3) a_1(0) + \frac 1 2 (a_1, a_3) a_2(0) - \frac 1 2 (a_1, a_2) a_3(0) 
\end{align*}
Basically speaking, the normal ordering of $h_1(m_1), ..., h_r(m_r)$ arranges all the negative modes to the left, then followed by all the positive modes in the middle, then followed by all the zero modes to the right. the process does not change the ordering among positive modes and that among negative modes. For example, if $m_1, m_2, m_3, m_5, m_6, m_8 \in \Z_+$, then 
\begin{align*}
    & \nord h_1(-m_1)h_2(m_2)h_3(m_3)h_4(0)h_5(-m_5)h_6(-m_6)h_7(0)h_8(m_8)\nord\\
    = \ &  h_1(-m_1)h_5(-m_5)h_6(-m_6) h_2(m_2)h_3(m_3)h_8(m_8)\nord h_4(0) h_7(0)\nord\\
    = \ &  h_1(-m_1)h_5(-m_5)h_6(-m_6) h_2(m_2)h_3(m_3)h_8(m_8) h_4(0) h_7(0)\\
    &  - \frac 1 2 (h_4, h_7) h_1(-m_1)h_5(-m_5)h_6(-m_6) h_2(m_2)h_3(m_3)h_8(m_8).
\end{align*}
\end{defn}

\begin{rema}
    Permutations satisfying (\ref{shuffle-defn}) is called a 3-shuffle. For every fixed $\mu, \nu \in \{0, ..., r\}$, the collection of permutations satisfying (\ref{shuffle-defn}) is denoted by $J_{\mu,\nu}(1, ..., r)$. We would also identify a 3-shuffle in $J_{\mu,\nu}(1, ... r)$ with three increasing sequences 
    $$1\leq p_1 < \cdots < p_\mu \leq r, 1 \leq p_{i_1}^c < \cdots < p_{i_\nu}^c\leq r, 1\leq p_{i_1^c}^c < \cdots < p_{i_{r-\mu-\nu}^c}^c\leq r$$
    in $\{1, ..., r\}$, with 
    $$\{p_1, ..., p_\mu\} \amalg \{p_1^c, ..., p_{r-\mu}^c\} = \{1, ..., r\}, \{i_1, ..., i_\nu\}\amalg \{i_1^c, ..., i_{r-\mu-\nu}^c\} = \{1, ..., r-\mu\}. $$
    In other words, we are understanding a 3-shuffle as the iterate of two 2-shuffles 
    $$1\leq p_1 < \cdots < p_\mu \leq r, 1\leq p_1^c < \cdots < p_{r-\mu}^c \leq r$$ 
    and 
    $$1\leq i_1 < \cdots < i_\nu \leq r-\mu, 1\leq i_1^c < \cdots < i_{r-\mu-\nu}^c \leq r-\mu. $$ 
    When $\mu =0$,  $J_{0, \nu}(1, ..., r)$ is the set of 2-shuffles $J_{\nu}(1, ..., r)$ considered in \cite{FQ}. Similar comments applies to the case $\nu= 0$. If $\mu =\nu = 0$, or $\mu = r$, or $\nu = r$, then only one increasing sequence is nonempty. In this case, $J_{\mu, \nu}(1, ..., r)$ consists of only the identity permutation $(1)$. Although it is possible to obtain explicit formulas for the sequence $p^c$ and $i^c$, we shall do so only when we need them. 
\end{rema}

\begin{rema}
    If the Clifford relation exists among zero modes, i.e., 
    $$a_1(0) a_2(0) + a_2(0)a_1(0) = (a_1, a_2) 1,$$
    then from the recursion, we may obtain
    $$\nord a_1(0) \cdots a_r(0) \nord = \frac 1 {r!}\sum_{\sigma\in Sym\{1, ..., r\}} (-1)^\sigma a_{\sigma(1)}(0) \cdots a_{\sigma(r)}(0),$$
    as in \cite{FFR}. In the module $W$ discussed in this paper, we require no relations among the zero modes at all. We expect it to be useful in the study of interacting zero modes. 
\end{rema}

\begin{prop}\label{zero-mode-prop}
    The normal ordering of zero modes admits a recursion from the right: for $a_1, ..., a_r\in \h$, 
    $$\nord a_1(0) \cdots a_r(0)\nord = \nord a_1(0) \cdots a_{r-1}(0)\nord a_r(0) + \sum_{i=1}^{r-1} \frac {(-1)^{i+r}}{2} (a_i, a_r) \nord a_1(0) \cdots \reallywidehat{a_i(0)} \cdots a_{r-1}(0)\nord $$
\end{prop}

\begin{proof}
    For $r=2$ the conclusion holds. Assume the conclusion holds for smaller $r$. By definition:
\begin{align}
    & \nord a_1(0)\cdots a_r(0)\nord \nonumber \\
    =\ & a_1(0) \nord a_2(0)\cdots a_r(0)\nord \nonumber\\
    & + \sum_{i=2}^r \frac {(-1)^{i+1}}2 (a_1, a_i) \nord a_2(0) \cdots \reallywidehat{a_i(0)}\cdots a_r(0)\nord \nonumber \\
    =\ &  a_1(0)\nord a_2(0) \cdots a_{r-1}(0) \nord a_r(0) \label{nord-right-1}\\
    & +  \sum_{j=2}^{r-1}\frac{(-1)^{j-1+r-1}}{2}(a_j, a_r) a_1(0) \nord a_2(0)\cdots \reallywidehat{a_j(0)} \cdots a_{r-1}(0)\nord \label{nord-right-2}\\
    & + \sum_{i=2}^{r-1} \frac{(-1)^{i+1}}{2}(a_1, a_i)\nord a_2(0) \cdots \reallywidehat{a_i(0)}\cdots  a_{r-1}(0)\nord a_r(0) \label{nord-right-3}\\
    & + \sum_{i=2}^{r-1}\sum_{j=2}^{i-1} \frac{(-1)^{i+1}}{2} (a_1, a_i) \frac{(-1)^{j-1+r-2}}{2}(a_j, a_r) \nord a_2(0) \cdots \reallywidehat{a_j(0)} \cdots \reallywidehat{a_i(0)} \cdots a_{r-1}(0)\nord \label{nord-right-4}\\
    & + \sum_{i=2}^{r-1}\sum_{j=i+1}^{r-1} \frac{(-1)^{i+1}}{2} (a_1, a_i) \frac{(-1)^{j-2+r-2}}{2}(a_j, a_r)  \nord a_2(0) \cdots \reallywidehat{a_i(0)} \cdots \reallywidehat{a_j(0)} \cdots a_{r-1}(0)\nord \label{nord-right-5}\\
    & + \frac{(-1)^{r+1}} 2 (a_1, a_r) \nord a_2(0) \cdots a_{r-1}(0)\nord.\label{nord-right-6}
\end{align}
Note that (\ref{nord-right-1}) and (\ref{nord-right-3}) combine into
$$\nord a_1(0) \cdots a_{r-1}(0)\nord a_r(0). $$
Note that (\ref{nord-right-2}), (\ref{nord-right-4}) and (\ref{nord-right-6}) line combine into
\begin{align*}
     \sum_{j=2}^{r-1} \frac {(-1)^{j+r}}{2}(a_j, a_r) {\bigg(} & a_1(0) \nord a_2(0)\cdots \reallywidehat{a_j(0)} \cdots a_{r-1}(0) \nord  \\ & + \sum_{i=2}^{j-1} \frac{(-1)^{i+1}}{2} (a_1, a_i) \nord a_2(0) \cdots \reallywidehat{a_i(0)} \cdots \reallywidehat{a_j(0)} \cdots a_{r-1}(0) \nord \\
     & + \sum_{i=j+1}^{r-1} \frac{(-1)^{i-1+1}}2 (a_1, a_i)\nord a_2(0) \cdots \reallywidehat{a_j(0)} \cdots \reallywidehat{a_i(0)}\cdots a_{r-1}(0)\nord {\bigg)}, 
\end{align*}
which is precisely
$$\sum_{j=2}^{r-1} \frac{(-1)^{j+r}} 2 (a_j, a_r) \nord a_1(0) \cdots \reallywidehat{a_j(0)} \cdots a_{r-1}(0)\nord. $$
Then (\ref{nord-right-6}) can be included, giving the conclusion.
\end{proof}

\begin{defn}\label{nord-defn-series}
    The definition of the normal ordering extends naturally to all linear combinations of products of modes. With the same philosophy, we define the \textit{normal ordering} of the product series $h_1^{(m_1)}(x_1)\cdots  h_r^{(m_r)}(x_1)$ by components, i.e., 
    \begin{align*}
        & \nord h_1^{(m_1)}(x_1) \cdots h_r^{(m_r)}(x_r)\nord \\
        = \ & \sum_{n_1, ..., n_r\in \Z} \nord h_1(n_1)\cdots h_r(n_r)\nord (x_1^{(-n_1-1/2})^{(m_1)} \cdots (x_r^{(-n_r-1/2})^{(m_r)} 
    \end{align*}
\end{defn}

\begin{rema}
Clearly, the normal ordering operation is ``multilinear'', in the sense that for $c_1, c_2\in \C$
\begin{align*}
    & \nord h_1^{(m_1)}(x_1) \cdots \left(c_1 h_{i_1}^{(m_{i_1})}(x_i) + c_2 h_{i_2}^{(m_{i_2})}(x_i)\right) \cdots  h_r^{(m_r)}(x_r)\nord    \\
    = \ & c_1 \nord h_1^{(m_1)}(x_1) \cdots h_{i_1}^{(m_{i_1})}(x_i) \cdots  h_r^{(m_r)}(x_r)\nord  + c_2 \nord h_1^{(m_1)}(x_1) \cdots h_{i_2}^{(m_{i_2})}(x_i) \cdots  h_r^{(m_r)}(x_r)\nord  .
\end{align*}
    
\end{rema}

\begin{rema}\label{zero-suffices}
    Clearly, the normal ordering operation commutes with partial differentiation: for every $i=1, ..., r$
$$\frac{\partial}{\partial x_i}\nord h_1^{(m_1)}(x_1) \cdots h_r^{(m_r)}(x_r)\nord = \nord\frac{\partial}{\partial x_i}h_1^{(m_1)}(x_1) \cdots h_r^{(m_r)}(x_r)\nord.  $$
This property simplifies the proofs of many identities in this paper. With this property, it suffices to show those identities in the case when $m_1 = \cdots = m_r = 0$. 
\end{rema}

\begin{defn}\label{na\"ive-vo-defn}
    We now define the na\"ive vertex operator 
    $$\bar Y_W: V\otimes W \to W((x))$$
    on the basis elements of $V$: For $\one \in V$, we set 
    $$\bar Y_W(\one, x) = 1_W. $$
    For $h_1, ..., h_r\in \h, m_1, ..., m_r\in \N$, 
    \begin{align}
        & \bar Y_W (h_1(-m_1-1/2)\cdots h_r(-m_r-1/2)\one, x) \nonumber\\
        = \ & \nord h_1^{(m_1)}(x)\cdots h_r^{(m_r)}(x)\nord\nonumber\\
        = \ & \sum_{\mu = 0}^r\sum_{\nu=0}^{r-\mu} \sum_{\sigma\in J_{\mu, \nu}(1, ..., r)} h_{\sigma(1)}^{(m_{\sigma(1)})}(x)^- \cdots h_{\sigma(\mu)}^{(m_{\sigma(\mu)})}(x)^- \nonumber \\
        & \qquad \qquad \qquad \qquad \ \cdot h_{\sigma(\mu+1)}^{(m_{\sigma(\mu+1)})}(x)^+ \cdots h_{\sigma(\nu)}^{(m_{\sigma(\nu)})}(x)^+ \nonumber \\
        & \qquad \qquad \qquad \qquad \ \cdot \nord h_{\sigma(\nu+1)}^{(m_{\sigma(\nu+1)})}(x)^0 \cdots h_{\sigma(r)}^{(m_{\sigma(r)})}(x)^0\nord  
    \end{align}
    The vertex operator extends to every element $v\in V$. It is clear that for every fixed element $w\in V$, $\bar Y_W(v, x)w\in W((x^{1/2}))$. 
\end{defn}

\section{Products and iterates of na\"ive vertex operators}

In this section we compute the products and iterates of na\"ive vertex operator and show the na\"ive vertex operators do not satisfy weak associativity.

\subsection{Some lemmas regarding shuffles}

Unlike the situation in \cite{FQ}, we do not have any reasonable recurrence relation of normal-ordered products in terms of positive, negative, or zero modes. We will have to prove most of the identities by manipulating shuffles. The following lemma will be useful for the process. 

\begin{lemma}\label{Comb-Id-Lemma}
For every fixed $\mu$ and $\nu$, any algebraic expression $\Phi$ with $r-1$ indices and any algebraic expression $\Psi$ with 1 index, 
\begin{align}
    & \sum_{\sigma\in J_{\mu,\nu}(1, ..., r)} \sum_{i=1}^\mu (-1)^{i-1}(-1)^\sigma \Phi(\sigma(1), ..., \reallywidehat{\sigma(i)}, ...,
    \sigma(\mu), \sigma(\mu+1), ..., \sigma(r))\Psi(\sigma(i)) \nonumber\\
    & = \sum_{j=1}^r (-1)^{j-1} \Psi(j) \sum_{\tau \in J_{\mu-1, \nu}(1,...,\widehat{j},...,r)} (-1)^\tau \Phi(\tau(1), ..., \reallywidehat{\tau(j)}, ..., \tau(r)) \label{Comb-Id-1}\\
    & \sum_{\sigma\in J_{\mu, \nu}(1, ..., r)} \sum_{i=1}^{\nu} (-1)^{\mu+ i-1}(-1)^\sigma \Phi(\sigma(1), ..., \sigma(\mu), \sigma(\mu+1), ..., \reallywidehat{\sigma(\mu+i)}, ...,
    \sigma(r))\Psi(\sigma(i)) \nonumber\\
    & = \sum_{j=1}^r (-1)^{j-1} \Psi(j) \sum_{\tau \in J_{\mu, \nu-1}(1,...,\widehat{j},...,r)} (-1)^\tau \Phi(\tau(1), ..., \reallywidehat{\tau(j)}, ..., \tau(r)) \label{Comb-Id-2}\\
   & \sum_{\sigma\in J_{\mu, \nu}(1, ..., r)} \sum_{i=1}^{r-\mu-\nu} (-1)^{\mu+\nu+i-1}(-1)^\sigma \Phi(\sigma(1), ..., 
    \sigma(\mu+\nu), \sigma(\mu+\nu+1), ..., \reallywidehat{\sigma(\mu+\nu+i)}, ...,\sigma(r))\Psi(\sigma(i)) \nonumber\\
    & = \sum_{j=1}^r (-1)^{j-1} \Psi(j) \sum_{\tau \in J_{\mu, \nu}(1,...,\widehat{j},...,r)} (-1)^\tau \Phi(\tau(1),...,  \reallywidehat{\tau(j)}, ..., \tau(r)) \label{Comb-Id-3}
\end{align}
\end{lemma}

\begin{proof}
We first focus on (\ref{Comb-Id-1}). The key is to view $\sigma\in J_{\mu, \nu}(1, ..., r)$ as three sequences of numbers $\{1, ..., r\}$. For each fixed $i$, we set $j=\sigma(i)$. Removing $\sigma(i)$ from the sequence amounts to removing $j$ from the sequence. What remains is a sequence that does not include $j$. More precisely, the left-hand-side of (\ref{Comb-Id-1}) can be written as 
\begin{align*}
    & \sum_{\sigma\in J_{\mu, \nu}(1, ..., r)} \sum_{i=1}^\mu  (-1)^{i-1}(-1)^\sigma\Phi(\sigma(1), ..., \reallywidehat{\sigma(i)}, ...,
    \sigma(\mu), \sigma(\mu+1), ..., \sigma(r))\Psi(\sigma(i)) \nonumber\\
    & = \sum_{\sigma\in J_{\mu, \nu}(1, ..., r)}  (-1)^{1-1}(-1)^\sigma\Phi(\reallywidehat{\sigma(1)}, \sigma(2), ...,
    \sigma(\mu), \sigma(\mu+1), ..., \sigma(r))\Psi(\sigma(1))\nonumber\\
    &+\cdots +\sum_{\sigma\in J_{\mu, \nu}(1, ..., r)}  (-1)^{i-1}(-1)^\sigma\Phi(\sigma(1), ..., \reallywidehat{\sigma(i)}, ...,
    \sigma(\mu), \sigma(\mu+1), ..., \sigma(r))\Psi(\sigma(i))\nonumber\\
    &+\cdots +\sum_{\sigma\in J_{\mu, \nu}(1, ..., r)}  (-1)^{\mu-1}(-1)^\sigma\Phi(\sigma(1), ..., \sigma(\mu-1), \reallywidehat{\sigma(\mu)}, \sigma(\mu+1), ..., \sigma(r))\Psi(\sigma(\mu))\nonumber\\
    & = \sum_{j=1}^r\sum_{\sigma\in J_{\mu, \nu}(1, ..., r), \sigma(1)=j}  (-1)^{1-1}(-1)^\sigma\Phi(\widehat{j}, \sigma(2), ...,
    \sigma(\mu), \sigma(\mu+1), ..., \sigma(r))\Psi(j)\nonumber\\
    &+\cdots +\sum_{j=1}^r\sum_{\sigma\in J_{\mu, \nu}(1, ..., r), \sigma(i)=j}  (-1)^{i-1}(-1)^\sigma\Phi(\sigma(1), ..., \widehat{j}, ...,
    \sigma(\mu), \sigma(\mu+1), ..., \sigma(r))\Psi(j)\nonumber\\
    &+\cdots +\sum_{j=1}^r\sum_{\sigma\in J_{\mu, \nu}(1, ..., r), \sigma(\mu)=j}  (-1)^{\mu-1}(-1)^\sigma\Phi(\sigma(1), ..., \sigma(\mu-1), \widehat{j}, \sigma(\mu+1), ..., \sigma(r))\Psi(j) 
\end{align*}
Fix $i=1,...,\mu$ and $\sigma\in J_{\mu, \nu}(1, ..., r)$ with $\sigma(i)=j$. We regard $\sigma$ as three increasing sequences of numbers in $\{1, ..., r\}$ respectively of size $\mu, \nu$ and $r-\mu-\nu$, with $j$ sitting in the $i$-th position the first sequence. Removing $j$ results in three increasing sequences of numbers in $\{1, ..., \widehat{j}, ..., r\}$ respectively of sizes $\mu-1, \nu$ and $r-1-\mu-\nu$. The number at the $(i-1)$-th position is strictly less than $j$, while that at $i$-th position is strictly larger than $j$. Let $\tau$ be the corresponding permutation in $J_{\mu-1, \nu}(1, ..., \widehat{j}, ..., r)$. Then in case $i<j$, we know 
\begin{align*}
    & \tau(1)=\sigma(1), ..., \tau(i-1)=\sigma(i-1)< j < \tau(i)= \sigma(i+1), \\
    & \tau(i+1)=\sigma(i+2), ..., \tau(j-1)=\sigma(j), \tau(j+1)=\sigma(j+1), ..., \tau(r)=\sigma(r).
\end{align*} 
In case $i = j$, we know 
\begin{align*}
    & \tau(1)=\sigma(1), ..., \tau(i-1)=\sigma(i-1) \\   & \tau(i+1)=\sigma(i+1), ..., \tau(r)=\sigma(r).
\end{align*} 
Notice that $i > j$ is impossible. 
To understand the relation of $(-1)^\tau$ and $(-1)^\sigma$, we note that from $\tau$ one can recover $\sigma$ by first defining
$$\sigma_0(k) = \left\{\begin{array}{ll}
\tau(k) & k\neq j\\
j & k = j
\end{array}\right.,$$
then performing transpositions to move $j$ consecutively to the $i$-th position. It is clear from this procedure that 
$$(-1)^\tau = (-1)^{\sigma_0} = (-1)^{\sigma}(-1)^{j-i} \Rightarrow (-1)^\sigma(-1)^{i-1} = (-1)^\tau (-1)^{j-1}. $$
Therefore, the left-hand-side of  (\ref{Comb-Id-1}) can be rewritten as
\begin{align*}
    & \sum_{j=1}^r\sum_{\tau\in J_{\mu-1, \nu}(1, ..., \widehat{j},..., r), \tau(1)>j} (-1)^{\tau}(-1)^{j-1}\Phi(\widehat{j}, \tau(1), ...,
    \tau(\mu), \tau(\mu+1), ..., \tau(r))\Psi(j))\\
    &+\cdots +\sum_{j=1}^r\sum_{\tau\in J_{\mu-1, \nu}(1, ..., \widehat{j},..., r), \tau(i-1)<j<\tau(i)}  (-1)^{\tau}(-1)^{j-1} \Phi(\tau(1), ..., \widehat{j}, ...,
    \tau(\mu), \tau(\mu+1), ..., \tau(r))\Psi(j)\\
    &+\cdots +\sum_{j=1}^r\sum_{\tau\in J_{\mu-1, \nu}(1, ..., \widehat{j},..., r), \tau(\mu)<j}  (-1)^{\tau}(-1)^{j-1}\Phi(\tau(1), ..., \tau(\mu-1), \widehat{j}, \tau(\mu+1), ..., \tau(r))\Psi(j)\\
    &= \sum_{j=1}^r(-1)^{j-1}\Psi(j))\sum_{\tau\in J_{\mu-1, \nu}(1, ..., \widehat{j},..., r)}  (-1)^{\tau} \Phi(\tau(1), ...,
    \tau(\mu), \tau(\mu+1), ..., \tau(r))
\end{align*}
The arguments for (\ref{Comb-Id-2}) and (\ref{Comb-Id-3}) are very similar, though the parity analysis is more complicated. To obtain the identity $(-1)^\sigma (-1)^{\mu+i-1} = (-1)^{j-1}(-1)^{\tau}$ in the context of (\ref{Comb-Id-2}), one has to consider all possible arrangements of the positions of $i, j, \mu$ and $\mu+i$ (there are at most $\binom{5}2+5=15$ cases). To obtain the identity $(-1)^\sigma(-1)^{\mu+\nu+i-1}= (-1)^{j-1}(-1)^\tau$ in the context of (\ref{Comb-Id-3}), one has to consider all possible arrangements of the positions of $i, j, \mu, \nu$ and $\nu+i$ (there are at most $\binom{7}2+7 = 28$ cases). We shall not repeat the details here. 
\end{proof}

\begin{rema}\label{2-shuffle-summation}
For two shuffles similar results hold. 
\begin{align}
    & \sum_{\sigma\in J_\mu(1, ..., r)} \sum_{i=1}^\mu (-1)^{i-1}(-1)^\sigma \Phi(\sigma(1), ..., \reallywidehat{\sigma(i)}, ..., \sigma(\mu), \sigma(\mu+1), ..., \sigma(r))\Psi(\sigma(i))
    \nonumber\\
    & = \sum_{j=1}^{r} (-1)^{j-1}\Psi(j) \sum_{\tau\in J_{\mu-1}(1, ..., \widehat{j}, ..., r)} (-1)^\tau \Phi(\tau(1), ..., \reallywidehat{\tau(j)}, ..., \tau(\mu-1), \tau(\mu), ..., \tau(r)). \label{Comb-Id-4}\\
    & \sum_{\sigma\in J_\mu(1, ..., r)} \sum_{i=1}^\mu (-1)^{\mu+i-1}(-1)^\sigma \Phi(\sigma(1), ..., \reallywidehat{\sigma(i)}, ..., \sigma(\mu), \sigma(\mu+1), ..., \sigma(r))\Psi(\sigma(i))
    \nonumber\\
    & = \sum_{j=1}^{r} (-1)^{j-1}\Psi(j) \sum_{\tau\in J_\mu(1, ..., \widehat{j}, ..., r)} (-1)^\tau \Phi(\tau(1), ..., \reallywidehat{\tau(j)}, ..., \tau(\mu-1), \tau(\mu), ..., \tau(r)). \label{Comb-Id-5}
\end{align}
These two formulas will play the central role in proving weak associativity of vertex operators. 
\end{rema}


\begin{lemma}\label{2-shuffle-lemma}
    Let $\sigma\in J_\mu(1, ..., r)$ be a 2-shuffle. Let 
    $$p_1 = \sigma(1), ..., p_\mu = \sigma(\mu).$$
    Then $\sigma$ can be expressed as a product of cycles
    $$\sigma = (p_1, p_1-1, ..., 2, 1) (p_2, p_2-1,..., 3, 2)\cdots (p_\mu, p_{\mu-1}, ..., \mu+1, \mu)$$
    In particular, we have 
    $$(-1)^\sigma = (-1)^{p_1-1+p_2-2+\cdots + p_\mu-\mu} = (-1)^{p_1+\cdots + p_\mu} (-1)^{\frac{\mu(\mu+1)}2}. $$
\end{lemma}

\begin{proof}
    Clearly the conclusion holds when $r=1$. Assume it holds for all smaller $r$. 
    We consider the position of the largest number $r$ in the two sequences. Only two cases may happen: $\sigma(r) = r$ or $\sigma(\mu)=r$. If $\sigma(r) = r$, then $\sigma$ itself can be regarded as an element in $J_{\mu}(1,..., r-1)$. The conclusion follows by induction hypothesis. If $\sigma(\mu)=r$, then the permutation 
    \begin{align*}
        &\tau(1) = \sigma(1), ..., \tau(\mu-1) = \sigma(\mu-1)\\
        &\tau(\mu)=\sigma(\mu+1), ..., \tau(r-1) = \sigma(r)
    \end{align*}
    is an element in $J_{\mu-1}(1, ..., r-1)$. By the induction hypothesis, 
    \begin{align}
        \tau = (p_1, p_1-1, ..., 2, 1) (p_2, p_2-1, ..., 3, 2) \cdots (p_{\mu-1}, p_{\mu-1}-1, ..., \mu, \mu-1) \label{2-shuffle-1}
    \end{align}
    We may also regard $\tau$ as an element in the symmetric group of $\{1, ..., r\}$ sending $r$ to $r$. Then (\ref{2-shuffle-1}) still holds. From the expression 
    \begin{align*}
        \sigma&=\begin{pmatrix}
            1 & \cdots & \mu-1 & \mu & \mu+1 & \cdots & r-1 & r\\
            \sigma(1) & \cdots & \sigma(\mu-1) & r & \sigma(\mu+1) & \cdots & \sigma(r-1) & \sigma(r)
        \end{pmatrix},\\
        \tau &= \begin{pmatrix}
            1 & \cdots & \mu-1 & \mu & \mu+1 & \cdots & r-1 & r \\
            \sigma(1) & \cdots & \sigma(\mu-1) & \sigma(\mu+1) & \sigma(\mu+2) & \cdots & \sigma(r) & r
        \end{pmatrix}, 
    \end{align*}
    it is clear that 
    $$\sigma = \tau (r, r-1, ...,\mu+1, \mu) = \tau (p_\mu, p_\mu-1, ..., \mu+1, \mu).$$
    So the conclusion is proved. 
\end{proof}


\subsection{Product of two normal-ordered products - simple case}

We start by the following special case. 

\begin{prop} \label{1-s-product-prop}
For $s\in \Z_+, a, b_1, ..., b_s\in \h$, 
\begin{align}
    & a(x) \nord b_1(y_1)\cdots b_s(y_s)\nord - \nord a(x)b_1(y_1)\cdots b_s(y_s)\nord \nonumber\\
    & = \sum_{i=1}^s (-1)^{i-1}(a, b_i) \iota_{xy_i}f(x, y_i) \nord b_1(y_1) \cdots \reallywidehat{b_i(y_i)} \cdots b_s(y_s)\nord \label{1-s-product-line-1}\\
    & + \sum_{i=1}^s \frac{(-1)^{i-1}}{2}(a, b_i) x^{-1/2}y_i^{-1/2} \nord b_1(y_1) \cdots \reallywidehat{b_i(y_i)} \cdots b_s(y_s)\nord\label{1-s-product-line-2}
\end{align}
\end{prop}

\begin{proof} Write $$a(x) = a(x)^+ + a(x)^0 + a(x)^-.$$ 
For the positive mode: 
\begin{align}
    & a(x)^+\nord b_1(y_1)\cdots b_s(y_s)\nord \nonumber\\
    & \begin{aligned}
    = a(x)^+\sum_{\mu =0}^r\sum_{\nu=0}^{r-\mu}\sum_{\sigma\in J(\{1, ..., r\}; \mu, \nu)} (-1)^\sigma & b_{\sigma(1)}(y_{\sigma(1)})^- \cdots b_{\sigma(\mu)}(y_{\sigma(\mu)})^-\\
    & \cdot  \nord b_{\sigma(\mu+1)}(y_{\sigma(\mu+1)})^0 \cdots b_{\sigma(\mu+\nu)}(y_{\sigma(\mu+\nu)})^0\nord\\
    & \cdot  b_{\sigma(\mu+\nu+1)}(y_{\sigma(\mu+\nu+1)})^+ \cdots b_{\sigma(r)}(y_{\sigma(r)})^+
    \end{aligned}\nonumber\\
    & \begin{aligned}
    = \sum_{\mu =0}^r\sum_{\nu=0}^{r-\mu}\sum_{\sigma\in J(\{1, ..., r\}; \mu, \nu)} (-1)^\sigma (-1)^{\mu}& b_{\sigma(1)}(y_{\sigma(1)})^- \cdots b_{\sigma(\mu)}(y_{\sigma(\mu)})^-\\
    & \cdot  a(x)^+\cdot  \nord b_{\sigma(\mu+1)}(y_{\sigma(\mu+1)})^0 \cdots b_{\sigma(\mu+\nu)}(y_{\sigma(\mu+\nu)})^0\nord\\
    & \cdot  b_{\sigma(\mu+\nu+1)}(y_{\sigma(\mu+\nu+1)})^+ \cdots b_{\sigma(r)}(y_{\sigma(r)})^+
    \end{aligned}\label{1-s-product-pos-line-1}\\
    & \begin{aligned}
    + \sum_{\mu =0}^r\sum_{\nu=0}^{r-\mu}\sum_{\sigma\in J(\{1, ..., r\}; \mu, \nu)} (-1)^\sigma \sum_{i=1}^\mu &(-1)^{i-1} (a, b_{\sigma(i)}) f(x_0, y_{\sigma(i)}) \\
    & \cdot  b_{\sigma(1)}(y_{\sigma(1)})^- \cdots \reallywidehat{b_{\sigma(i)}(y_{\sigma(i)})^-}\cdots  b_{\sigma(\mu)}(y_{\sigma(\mu)})^-\\
    & \cdot  \nord b_{\sigma(\mu+1)}(y_{\sigma(\mu+1)})^0 \cdots b_{\sigma(\mu+\nu)}(y_{\sigma(\mu+\nu)})^0\nord\\
    & \cdot  b_{\sigma(\mu+\nu+1)}(y_{\sigma(\mu+\nu+1)})^+ \cdots b_{\sigma(r)}(y_{\sigma(r)})^+
    \end{aligned}\label{1-s-product-pos-line-2}
\end{align}
For (\ref{1-s-product-pos-line-1}), we move $a(x)^+$ further to the right. So (\ref{1-s-product-pos-line-1}) becomes 
\begin{align*}
    \sum_{\mu =0}^r\sum_{\nu=0}^{r-\mu}\sum_{\sigma\in J(\{1, ..., r\}; \mu, \nu)} (-1)^\sigma (-1)^{\mu+\nu}& b_{\sigma(1)}(y_{\sigma(1)})^- \cdots b_{\sigma(\mu)}(y_{\sigma(\mu)})^-\nonumber\\
    & \cdot    \nord b_{\sigma(\mu+1)}(y_{\sigma(\mu+1)})^0 \cdots b_{\sigma(\mu+\nu)}(y_{\sigma(\mu+\nu)})^0\nord\nonumber\\
    & \cdot  a(x)^+\cdot b_{\sigma(\mu+\nu+1)}(y_{\sigma(\mu+\nu+1)})^+ \cdots b_{\sigma(r)}(y_{\sigma(r)})^+, 
\end{align*}
which is precisely $\nord a(x)^+ b_1(y_1)\cdots b_s(y_s)\nord$. 

For (\ref{1-s-product-pos-line-2}), we claim that it is precisely  (\ref{1-s-product-line-1}), which is indeed
\begin{align*}
    \sum_{i=1}^r(-1)^{i-1}(a,b_i)f(x, y_i) \sum_{\mu=0}^{r-1} \sum_{\nu=0}^{r-1-\mu} \sum_{\tau\in J(\{1, ..., \widehat{i}, ..., r\}), \mu, \nu) }(-1)^\tau& b_{\tau(1)}(y_{\tau(1)})^- \cdots \reallywidehat{b_{\tau(i)}(y_{\tau(i)})^-}\cdots b_{\tau(r)}(y_{\tau(r)})^-\\
    & \cdot  b_{\tau(\mu+1)}(y_{\tau(\mu+1)})^0 \cdots b_{\tau(\mu+\nu)}(y_{\tau(\mu+\nu)})^0 \\
    & \cdot  b_{\tau(\mu+\nu+1)}(y_{\tau(\mu+\nu+1)})^+ \cdots b_{\tau(r)}(y_{\tau(r)})^+. 
\end{align*}
This can be directly seen from (\ref{Comb-Id-1}) in Lemma \ref{Comb-Id-Lemma}. 

For the zero mode: 
\begin{align}
    & a(x)^0\nord b_1(y_1)\cdots b_s(y_s)\nord \nonumber\\
    & \begin{aligned}
    = a(x)^0\sum_{\mu =0}^r\sum_{\nu=0}^{r-\mu}\sum_{\sigma\in J(\{1, ..., r\}; \mu, \nu)} (-1)^\sigma& b_{\sigma(1)}(y_{\sigma(1)})^- \cdots b_{\sigma(\mu)}(y_{\sigma(\mu)})^-\\
    & \cdot  \nord b_{\sigma(\mu+1)}(y_{\sigma(\mu+1)})^0 \cdots b_{\sigma(\mu+\nu)}(y_{\sigma(\mu+\nu)})^0\nord\\
    & \cdot  b_{\sigma(\mu+\nu+1)}(y_{\sigma(\mu+\nu+1)})^+ \cdots b_{\sigma(r)}(y_{\sigma(r)})^+
    \end{aligned}\nonumber\\
    & \begin{aligned}
    = \sum_{\mu =0}^r\sum_{\nu=0}^{r-\mu}\sum_{\sigma\in J(\{1, ..., r\}; \mu, \nu)} (-1)^\sigma (-1)^{\mu}& b_{\sigma(1)}(y_{\sigma(1)})^- \cdots b_{\sigma(\mu)}(y_{\sigma(\mu)})^-\\
    & \cdot  a(x)^0\cdot  \nord b_{\sigma(\mu+1)}(y_{\sigma(\mu+1)})^0 \cdots b_{\sigma(\mu+\nu)}(y_{\sigma(\mu+\nu)})^0\nord\\
    & \cdot  b_{\sigma(\mu+\nu+1)}(y_{\sigma(\mu+\nu+1)})^+ \cdots b_{\sigma(r)}(y_{\sigma(r)})^+
    \end{aligned}\label{1-s-product-zero-1}
\end{align}
From the definition of normal ordering of zero mode, it is clear that 
\begin{align*}
    \nord a(x)^0 b_1(y_1)^0\cdots b_s(y_s)^0\nord &= a(x)^0 \nord b_1(y_1)^0\cdots b_s(y_s)^0 \nord \\ 
    & -  \sum_{i=1}^r \frac {(-1)^{i-1}} 2 (a, b_i)x^{-1/2}y_i^{-1/2} \nord b_1(y_1)^0 \cdots \reallywidehat{b_i(y_i)^0} \cdots b_s(y_s)^0\nord 
\end{align*} 
So (\ref{1-s-product-zero-1}) can be rewritten as 
\begin{align}
    \sum_{\mu =0}^r\sum_{\nu=0}^{r-\mu}\sum_{\sigma\in J(\{1, ..., r\}; \mu, \nu)} (-1)^\sigma (-1)^{\mu}& b_{\sigma(1)}(y_{\sigma(1)})^- \cdots b_{\sigma(\mu)}(y_{\sigma(\mu)})^-\nonumber\\
    & \cdot   \nord a(x)^0 b_{\sigma(\mu+1)}(y_{\sigma(\mu+1)})^0 \cdots b_{\sigma(\mu+\nu)}(y_{\sigma(\mu+\nu)})^0\nord\nonumber\\
    & \cdot  b_{\sigma(\mu+\nu+1)}(y_{\sigma(\mu+\nu+1)})^+ \cdots b_{\sigma(r)}(y_{\sigma(r)})^+\label{1-s-product-zero-2-line-1}\\
    + \sum_{\mu =0}^r\sum_{\nu=0}^{r-\mu}\sum_{\sigma\in J(\{1, ..., r\}; \mu, \nu)} (-1)^\sigma (-1)^{\mu}&\sum_{i=\mu+1}^{\mu+\nu}\frac{(-1)^{i-\mu-1}}2 b_{\sigma(1)}(y_{\sigma(1)})^- \cdots b_{\sigma(\mu)}(y_{\sigma(\mu)})^-\nonumber\\
    & \cdot   (a, b_i)x^{-1/2}y_i^{-1/2} \nord b_{\sigma(\mu+1)}(y_{\sigma(\mu+1)})^0 \cdots \reallywidehat{b_{\sigma(i)}(y_{\sigma(i)})^0}\cdots  b_{\sigma(\mu+\nu)}(y_{\sigma(\mu+\nu)})^0\nord\nonumber \\
    & \cdot  b_{\sigma(\mu+\nu+1)}(y_{\sigma(\mu+\nu+1)})^+ \cdots b_{\sigma(r)}(y_{\sigma(r)})^+\label{1-s-product-zero-2-line-2}
\end{align}
From (\ref{Comb-Id-2}) in Lemma \ref{Comb-Id-Lemma}, (\ref{1-s-product-zero-2-line-2}) can be  computed as 
\begin{align*}
    \sum_{j=1}^r \frac{(-1)^{j-1}}{2}(a, b_j)x_0^{-1/2}x_j^{-1/2} \sum_{\mu=0}^r \sum_{\nu=0}^{r-\mu} \sum_{\tau\in J\{1, ..., \widehat{j}, ..., r\}} &  b_{\tau(1)}(x_{\tau(1)})^- \cdots b_{\tau(\mu)}(x_{\tau(\mu)})^-\\
    & b_{\tau(\mu+1)}(x_{\tau(\mu+1)})^0 \cdots \reallywidehat{b_{\tau(j)}(x_{\tau(j)})^0}\cdots b_{\tau(\mu+\nu)}(x_{\tau(\mu+\nu)})^0 \\
    & b_{\tau(\mu+\nu+1)}(x_{\tau(\mu+\nu+1)})^+ \cdots b_{\tau(r)}(x_{\tau(r)})^+, 
\end{align*}
which is precisely (\ref{1-s-product-line-2}). And (\ref{1-s-product-zero-2-line-1}) is precisely $\nord a(x)^0 b_1(y_1)\cdots b_s(y_s)\nord$. 

Finally for the negative mode, it is straightforward that
$$a(x)^- b_1(y_1)\cdots b_s(y_s) = \nord a(x)^- b_1(y_1)\cdots b_s(y_s)\nord $$
The conclusion then follows by combining the formulas for the positive mode, zero mode and negative modes. 
\end{proof}

\begin{rema}
If we introduce 
\begin{align*}
    g(x, y) &= x^{-1/2}y^{-1/2}\cdot \frac 1 2 (x+y)(x-y)^{-1}\\
    &= f(x, y) + \frac 1 2 x^{-1/2} y^{-1/2},\\
    G(x, y) &= \iota_{xy}g(x,y), 
\end{align*}
then the conclusion is simply 
\begin{align*}
    & a(x) \nord b_1(y_1)\cdots b_s(y_s)\nord - \nord a(x)b_1(y_1)\cdots b_s(y_s)\nord \\
    & = \sum_{i=1}^r (-1)^{i-1}(a, b_i) G(x, y_i) \nord b_1(y_1) \cdots \reallywidehat{b_i(y_i)} \cdots b_s(y_s)\nord 
\end{align*}
To some extent this may be used as a ``definition'' for the normal ordering of these series, if we rewrite the formula as 
\begin{align}
    & \nord a(x)b_1(y_1)\cdots b_s(y_s)\nord\nonumber  \\
    = \ & a(x) \nord b_1(y_1)\cdots b_s(y_s)\nord +  \sum_{i=1}^r (-1)^{i}(a, b_i) G(x, y_i) \nord b_1(y_1) \cdots \reallywidehat{b_i(y_i)} \cdots b_s(y_s)\nord \label{nord-left-recurrence}
\end{align}
\end{rema}

\begin{prop}
\begin{align*}
    & \nord a_1(x_1)\cdots a_{r}(x_{r})\nord b(y)-  \nord a_1(x_1)\cdots a_{r}(x_{r}) b(y)\nord \\
    =& \sum_{i=1}^r (-1)^{r-i} (a_i, b)\iota_{x_iy}f(x_i, y) \nord a_1(x_1)\cdots \widehat{a_i}(x_i) \cdots a_r(x_r) \nord \\
    &+ \sum_{i=1}^r \frac{(-1)^{r-i}}{2} (a_i, b) x_i^{-1/2}y^{-1/2}\nord a_1(x_1) \cdots \widehat{a_i}(x_i) \cdots a_r(x_r)\nord \\
    =& \sum_{i=1}^r (-1)^{r-i} (a_i, b) G(x_i, y) \nord a_1(x_1)\cdots \widehat{a_i}(x_i) \cdots a_r(x_r) \nord.
\end{align*}
\end{prop}

\begin{proof}
One could certainly repeat the analysis of shuffles as in Proposition \ref{1-s-product-prop}. A more conceptual approach is to use induction as in Proposition \ref{zero-mode-prop} with the recurrence (\ref{nord-left-recurrence}). We shall omit the details here. 
\end{proof}

\subsection{Product of two normal-ordered products - general case}

\begin{thm}\label{Product-r-s-thm} 
    Let 
    \begin{align*}
        g_{mn}(x,y) & = \frac 1 {m!n!}\frac{\partial^{m+n}}{\partial x^m \partial y^n} g(x,y), \\
        G_{mn}(x,y) & = \iota_{xy}g_{mn}(x,y).
    \end{align*}
    Then for $r, s\in \Z_+, a_1, ..., a_r, b_1, ..., b_s\in \h, m_1, ..., m_r, n_1, ..., n_s\in \N$. Then 
    \begin{align}
        & \nord a_1^{(m_1)}(x_1)\cdots a_r^{(m_r)}(x_r)\nord  \cdot \nord b_1^{(n_1)}(y_1)\cdots b_s^{(n_s)}(y_s)\nord  \nonumber \\
        = \ & \nord a_1^{(m_1)}(x_1)\cdots a_r^{(m_r)}(x_r)b_1^{(n_1)}(y_1)\cdots b_s^{(n_s)}(y_s)\nord \nonumber \\
        & + \sum_{\rho=1}^{\min(r,s)}\sum_{\substack{ 1 \leq i_1 < \cdots < i_\rho\leq r\\ 1 \leq j_1 < \cdots < j_\rho \leq s}}(-1)^{i_1+\cdots + i_\rho + j_1+\cdots + j_\rho }(-1)^{r\rho + \rho(\rho+1)/2}\nonumber \\
        & \hspace{9.7 em}\cdot \begin{vmatrix}
            (a_{i_1}, b_{j_1})G_{m_{i_1}n_{j_1}}(x_{i_1}, y_{j_1}) & \cdots & (a_{i_1}, b_{j_\rho})G_{m_{i_1}n_{j_\rho}}(x_{i_1}, y_{j_\rho})\\
            \vdots & & \vdots\nonumber  \\
            (a_{i_\rho}, b_{j_1})G_{m_{i_\rho}n_{j_1}}(x_{i_\rho}, y_{j_1}) & \cdots & (a_{i_\rho}, b_{j_\rho})G_{m_{i_\rho}n_{j_\rho}}(x_{i_\rho}, y_{j_\rho})
        \end{vmatrix}\nonumber \\
        & \hspace{10 em}\cdot \nord a_{1}^{(m_1)}(x_1) \cdots \reallywidehat{a_{i_1}^{(m_{i_1})}(x_{i_1})} \cdots \reallywidehat{a_{i_\rho}^{(m_{i_\rho})}(x_{i_\rho})} \cdots a_{r}^{(m_r)}(x_r)\nonumber  \\
        & \hspace{11 em}\cdot b_{1}^{(n_1)}(y_1) \cdots \reallywidehat{b_{j_1}^{(n_{j_1})}(y_{j_1})} \cdots \reallywidehat{b_{j_\rho}^{(n_{j_\rho})}(y_{j_\rho})} \cdots b_{s}^{(n_s)}(y_s)\nord \label{Product-r-s-formula}     
    \end{align}
\end{thm}

\begin{proof}
It suffices to replace the symbol $:$ and the function $f$ in the proof of Theorem 4.3 in \cite{FQ} respectively by the symbol $\nord$ and the series $G$. Other arguments are verbatim. We shall not repeat the details here. 
\end{proof}

\begin{rema}
Similarly, the result can be viewed as the Wick's theorem in quantum field theory. Theorem \ref{Product-r-s-thm} shows that Wick's theorem holds without any assumption on the relations among the creation operators and the zero modes. 
\end{rema}

\subsection{Product and iterate formulas}

\begin{cor}\label{Product-thm}
    For $r, s\in \N, a_1, ..., a_r, b_1, ..., b_s\in \h, m_1, ..., m_r, n_1, ..., n_s\in \N$
    \begin{align}
        & \bar Y_W(a_1(-m_1-1/2)\cdots a_r(-m_r-1/2)\one, x)\bar Y_W(b_1(-n_1-1/2)\cdots b_s(-n_s-1/2)\one, y) \nonumber\\
        = \ & \nord a_1^{(m_1)}(x)\cdots a_r^{(m_r)}(x)b_1^{(n_1)}(y)\cdots b_s^{(n_s)}(y)\nord \nonumber \\
        & + \sum_{\rho=1}^{\min(r,s)}\sum_{\substack{ 1 \leq i_1 < \cdots < i_\rho\leq r\\ 1 \leq j_1 < \cdots < j_\rho \leq s}}(-1)^{i_1+\cdots + i_\rho + j_1+\cdots + j_\rho }(-1)^{r\rho + \rho(\rho+1)/2}\nonumber \\
        & \hspace{9.7 em}\cdot \begin{vmatrix}
            (a_{i_1}, b_{j_1})G_{m_{i_1}n_{j_1}}(x, y) & \cdots & (a_{i_1}, b_{j_\rho})G_{m_{i_1}n_{j_\rho}}(x, y)\\
            \vdots & & \vdots\nonumber  \\
            (a_{i_\rho}, b_{j_1})G_{m_{i_\rho}n_{j_1}}(x, y) & \cdots & (a_{i_\rho}, b_{j_\rho})G_{m_{i_\rho}n_{j_\rho}}(x, y)
        \end{vmatrix}\nonumber \\
        & \hspace{10 em}\cdot \nord a_{1}^{(m_1)}(x) \cdots \reallywidehat{a_{i_1}^{(m_{i_1})}(x)} \cdots \reallywidehat{a_{i_\rho}^{(m_{i_\rho})}(x)} \cdots a_{r}^{(m_r)}(x)\nonumber  \\
        & \hspace{11 em}\cdot b_{1}^{(n_1)}(y) \cdots \reallywidehat{b_{j_1}^{(n_{j_1})}(y)} \cdots \reallywidehat{b_{j_\rho}^{(n_{j_\rho})}(y)} \cdots b_{s}^{(n_s)}(y)\nord \label{Product-formula}
    \end{align}
\end{cor}

\begin{proof}
    Directly substitute $x_1 = x, ..., x_r = r, y_1=y, ..., y_s = y$ in Theorem \ref{Product-r-s-thm} to see the conclusion.  
\end{proof}

\begin{cor}\label{Iterate-thm}
    For $r, s\in \N, a_1, ..., a_r, b_1, ..., b_s\in \h, m_1, ..., m_r, n_1, ..., n_s\in \N$
    \begin{align}
        & \bar Y_W(Y_V(a_1(-m_1-1/2)\cdots a_r(-m_r-1/2)\one, x)b_1(-n_1-1/2)\cdots b_s(-n_s-1/2)\one, y)\nonumber \\
        = \ & \nord a_1^{(m_1)}(y+x)\cdots a_r^{(m_r)}(y+x)b_1^{(n_1)}(y)\cdots b_s^{(n_s)}(y)\nord \nonumber \\
        & + \sum_{\rho=1}^{\min(r,s)}\sum_{\substack{ 1 \leq i_1 < \cdots < i_\rho\leq r\\ 1 \leq j_1 < \cdots < j_\rho \leq s}}(-1)^{i_1+\cdots + i_\rho + j_1+\cdots + j_\rho }(-1)^{r\rho + \rho(\rho+1)/2}\nonumber \\
        & \hspace{9.7 em}\cdot \begin{vmatrix}
            (a_{i_1}, b_{j_1})(x^{-n_{j_1}-1})^{(m_{i_1})} & \cdots & (a_{i_1}, b_{j_\rho})(x^{-n_{j_\rho}-1})^{(m_{i_1})}\\
            \vdots & & \vdots\nonumber  \\
            (a_{i_\rho}, b_{j_1})(x^{-n_{j_1}-1})^{(m_{i_\rho})} & \cdots & (a_{i_\rho}, b_{j_\rho})(x^{-n_{j_\rho}-1})^{(m_{i_\rho})} 
        \end{vmatrix}\nonumber \\
        & \hspace{10 em}\cdot \nord a_{1}^{(m_1)}(y+x) \cdots \reallywidehat{a_{i_1}^{(m_{i_1})}(y+x)} \cdots \reallywidehat{a_{i_\rho}^{(m_{i_\rho})}(y+x)} \cdots a_{r}^{(m_r)}(y+x)\nonumber  \\
        & \hspace{11 em}\cdot b_{1}^{(n_1)}(y) \cdots \reallywidehat{b_{j_1}^{(n_{j_1})}(y)} \cdots \reallywidehat{b_{j_\rho}^{(n_{j_\rho})}(y)} \cdots b_{s}^{(n_s)}(y)\nord \label{Iterate-formula}
    \end{align}
\end{cor}

\begin{proof}
    Recall from Corollary 4.6 in \cite{FQ} that
       \begin{align*}
        & Y_V(a_1(-m_1-1/2)\cdots a_r(-m_r-1/2)\one, x)b_1(-n_1-1/2)\cdots b_s(-n_s-1/2)\one \\
        = \ & (a_1^{(m_1)})_V(x)^-\cdots (a_r^{(m_r)})_V(x)^-b_1(-n_1-1/2)\cdots b_s(-n_s-1/2)\one\nonumber \\
        & + \sum_{\rho=1}^{\min(r,s)}\sum_{\substack{ 1 \leq i_1 < \cdots < i_\rho\leq r\\ 1 \leq j_1 < \cdots < j_\rho \leq s}}(-1)^{i_1+\cdots + i_\rho + j_1+\cdots + j_\rho }(-1)^{r\rho + \rho(\rho+1)/2}\nonumber \\
        & \hspace{9.7 em}\cdot \begin{vmatrix}
            (a_{i_1}, b_{j_1})(x^{-n_{j_1}-1})^{(m_{i_1})} & \cdots & (a_{i_1}, b_{j_\rho})(x^{-n_{j_\rho}-1})^{(m_{i_1})}\\
            \vdots & & \vdots\nonumber  \\
            (a_{i_\rho}, b_{j_1})(x^{-n_{j_1}-1})^{(m_{i_\rho})} & \cdots & (a_{i_\rho}, b_{j_\rho})(x^{-n_{j_\rho}-1})^{(m_{i_\rho})} 
        \end{vmatrix}\nonumber \\
        & \hspace{10 em}\cdot (a_{1}^{(m_1)})_V(x)^- \cdots \reallywidehat{(a_{i_1}^{(m_{i_1})})_V(x)^-} \cdots \reallywidehat{(a_{i_\rho}^{(m_{i_\rho})})_V(x)^-} \cdots (a_{r}^{(m_r)})_V(x)^-\nonumber  \\
        & \hspace{10 em}\cdot b_{1}(-n_1/1-2) \cdots \reallywidehat{b_{j_1}(-n_{j_1}-1/2)} \cdots \reallywidehat{b_{j_\rho}(-n_{j_\rho}-1/2)} \cdots b_{s}(-n_s-1/2)\one
    \end{align*}
    We first compute
    \begin{align}
        & \bar Y_W((a_1^{(m_1)})_V(x)^- \cdots (a_r^{(m_r)})_V(x)^- b_1(-n_1-1/2)\cdots b_s(-n_s-1/2)\one, y)\label{Iterate-1} \\
        = \ & \sum_{i_1, ..., i_r\geq 0} \bar Y_W(a_1(-i_1-1/2)\cdots a_r(-i_r-1/2)b_1(-n_1-1/2)\cdots b_s(-n_s-1/2)\one, y)(x^{i_1})^{(m_1)}\cdots (x^{i_r})^{(m_r)}\nonumber\\
        = \ & \sum_{i_1, ..., i_r\geq 0} \nord a_1^{(i_1)}(y)\cdots  a_r^{(i_r)}(y)b_1^{(n_1)}(y)\cdots b_s^{(n_s)}(y)\nord (x^{i_1})^{(m_1)}\cdots (x^{i_r})^{(m_r)}\nonumber\\
        = \ & \nord \frac{1}{m_1!}\frac{\partial^{m_1}}{\partial x^{m_1}}\left(\sum_{i_1\geq 0} \frac{x^{i_1}}{i_1!}\frac{\partial^{i_1}}{\partial y^{i_1}}a_1(y)\right) \cdots \frac{1}{m_r!}\frac{\partial^{m_r}}{\partial x^{m_r}}\left(\sum_{i_r\geq 0} \frac{x^{i_r}}{i_r!}\frac{\partial^{i_r}}{\partial y^{i_r}}a_r(y)\right) b_1^{(n_1)}(y)\cdots b_s^{(n_s)}(y)\nord \nonumber
    \end{align}
    By the formal Taylor theorem, 
    $$\sum_{i\geq 0} \frac{x^i}{i_1!}\frac{\partial}{\partial x}a(y) = a(y+x), $$
    thus
    \begin{align*}
        (\ref{Iterate-1}) = \ &  \nord \frac{1}{m_1!}\frac{\partial^{m_1}}{\partial x^{m_1}}\left(a_1(y+x)\right) \cdots \frac{1}{m_r!}\frac{\partial^{m_r}}{\partial x^{m_r}}\left(a_r(y+x)\right)b_1^{(n_1)}(y)\cdots b_s^{(n_s)}(y)\nord     
    \end{align*}
    Noticing from chain rule that 
    $$\frac{\partial }{\partial x} = \frac{\partial}{\partial(y+x)}\cdot \frac{\partial(y+x)}{\partial x} = \frac{\partial}{\partial(y+x)}, $$
    we conclude that 
    \begin{align*}
        (\ref{Iterate-1}) = \ &  \nord a_1^{(m_1)}(y+x) \cdots a_r^{(m_r)}(y+x) b_1^{(n_1)}(y)\cdots b_s^{(n_s)}(y)\nord 
    \end{align*}
    By variation of the choices of $a_1, ..., a_r$ and $m_1, ..., m_r$, we may conclude the proof of the corollary. 
\end{proof}

We see that the determinants in the product formula involves the algebraic function $G_{m_in_j}(x_i,y_j)$, while the determinants in the iterate formula involves the rational function $(x^{-n_j-1})^{(m_i)}$. So weak associativity does not hold with the na\"ive vertex operators $\bar Y_W$. It is necessary to introduce a correction. 

\section{The actual vertex operator}

Recall that $\h = \mathfrak{p} \oplus \mathfrak{q}$ is polarized as a direct sum of maximal isotropic subspaces. Let $e_1, ..., e_M$ be a basis of $\mathfrak{p}$, $\bar e_1, ..., \bar e_M$ be a basis of $\mathfrak{q}$ such that 
$$(e_i, \bar e_j) = \delta_{ij}, (e_i, e_j) = (\bar e_i, \bar e_j) = 0, i,j = 1, ..., M. $$
Using the idea from \cite{FLM} and \cite{FFR}, we define the actual vertex operator as a correction from the na\"ive vertex operator by the exponential of a series $\Delta(x)$, i.e., 
i.e., 
$$Y_W(u, x) = \bar Y_W (\exp (\Delta(x)) u, x)$$
where 
$$\Delta(x) = \sum_{i=1}^M\sum_{m,n \geq 0}C_{mn} e_i(m+1/2) \bar e_i(n+1/2) x^{-m-n-1}. $$
To make sure that $\Delta(x)$ is independent of the choice of polarizations and basis, from the anti-commutativity of the positive modes, it is necessary and sufficient that $C_{mn} = - C_{nm}$.

\subsection{$\exp(\Delta(x))$ on a basis element of $V$}

\begin{nota}\label{Index-Notation-1}
    To avoid the clumsy iterated subscripts, for $a_1, ...., a_r\in \h$, $m_1, ..., m_r\in \N$ and $1\leq i_1< ...< i_{k}\leq r$, we introduce the following notations. 
    \begin{enumerate}
        \item For $1\leq p < q \leq k$, the notation $$\langle i_p, i_q \rangle =(a_{i_p}, a_{i_q}) C_{m_{i_p}m_{i_q}}$$ 
        stands for a number in $\C$. 
        \item The notation 
        $$S_{a_1,... a_r}^{m_1, ..., m_r}(i_1, ..., i_k) = a_{i_1}(-m_{i_1}-1/2)\cdots a_{i_k}(-m_{i_k}-1/2)\one$$
        stands for an element in $V$
        \item The notation
        $$S(x)_{a_1,... a_r}^{m_1, ..., m_r}(i_1, ..., i_k) = a_{i_1}^{(m_{i_1})}(x)^-\cdots a_{i_k}^{(m_{i_k})}(x)^-$$
        stands for a series in $\End(V)[[x]]$. 
    \end{enumerate}
    These notations will be used in the proofs and in the computations. We will not use them in the statements of the definitions, lemmas, propositions and theorems. 
\end{nota}

\begin{defn}
    Let $a_1, ..., a_r\in \h, m_1, ..., m_r\in \N$. Then for each $t\in \Z_+$ increasing sequence $1\leq i_1 < \cdots < i_{2t}$, we define the \textit{total contraction number} $T_{a_1, ..., a_r}^{m_1, ..., m_r}(i_1, ..., i_{2t})$ recursively by
    \begin{align*}
    T_{a_1,...,a_r}^{m_1,...,m_r}(i_1, i_2) &= (a_{i_1},a_{i_2})C_{m_{i_1}m_{i_2}}\\
    T_{a_1,...,a_r}^{m_1,...,m_r}(i_1,..., i_{2t}) &= \sum_{k=2}^{2t} (-1)^{k}(a_{i_1}, a_{i_k}) C_{m_{i_1}m_{i_k}} T^{m_1\cdots m_r}_{a_1,...,a_r}(\widehat{i_1}, i_2, ..., \widehat{i_k}, i_{k+1}, ..., i_{2t}). 
\end{align*}
\end{defn}

Recall Corollary 5.7 in \cite{FQ}:

\begin{lemma}
    For $a_1, ..., a_r\in \h, m_1, ..., m_r\in \N, $
    \begin{align*}
        & \exp(\Delta(x)) a_1(-m_1-1/2)\cdots a_r(-m_r-1/2)\one \\
        = \ & \sum_{t=0}^\infty \sum_{1\leq i_1 < \cdots < 2t \leq r} (-1)^{i_1 + \cdots + i_{2t}} T_{a_1, .. a_r}^{m_1, ..., m_r}(i_1, ..., i_{2t})\\
        & \qquad \qquad \qquad \cdot a_1(-m_1-1/2)\cdots \reallywidehat{a_{i_1}(-m_{i_1}-1/2)}\cdots \reallywidehat{a_{i_{2t}}(-m_{i_{2t}}-1/2)} \cdots a_{r}(-m_r-1/2)\one. 
    \end{align*}
\end{lemma}

\subsection{Choosing the constants $C_{mn}$}
To understand the choice of $C_{mn}$, we first compute the product and iterate of the vertex operators associatied with for $a(-1/2)\one, b(-1/2)\one\in V$. 
\begin{align*}
    & Y_W(a(-m-1/2)\one, x+y)Y_W(b(-n-1/2)\one, y) \\
    = \ & \bar Y_W(a(-m-1/2)\one, x+y) \bar Y_W(b(-n-1/2)\one, y) \\
    = \ & \nord a^{(m)}(x+y) b^{(n)}(y)\nord + (a,b) \iota_{xy}g_{mn}(x+y,y)
\end{align*}
while 
\begin{align*}
    & Y_W(Y_V(a(-m-1/2)\one, x)b(-n-1/2)\one, y) \\
    = \ & \bar Y_W(Y_V(a(-m-1/2)\one, x)b(-n-1/2)\one, y) + \bar Y_W(\Delta(y) Y_V(a(-m-1/2)\one, x)b(-n-1/2)\one, y)\\
    = \ & \nord a^{(m)}(y+x) b^{(n)}(y)\nord + (a,b) (x^{-n-1})^{(m)} + \bar{Y}_W\left(\Delta(y)\sum_{p\geq 0} a(-p-1/2)b(-n-1/2)\one x^{p}, y\right)\\
    = \ & \nord a^{(m)}(y+x) b^{(n)}(y)\nord + (a,b) (x^{-n-1})^{(m)} + \bar{Y}_W\left(\sum_{p\geq 0} (-1)^{1+2}C_{pn}(a, b) y^{-p-n-1}\one (x^{p})^{(m)}, y\right)\\
    = \ & \nord a^{(m)}(y+x) b^{(n)}(y)\nord + (a,b) \left((x^{-n-1})^{(m)} + \sum_{p\geq 0} C_{np} y^{-p-n-1} \binom{p}{m}x^{p-m}\right)\\
    = \ & \nord a^{(m)}(y+x) b^{(n)}(y)\nord + (a,b) \left((x^{-n-1})^{(m)} + y^{-m-n-1} \sum_{p\geq 0} C_{n,p+m}  \binom{p+m}{m}\left(\frac x y\right)^{p}\right)
\end{align*}
The associativity axiom basically requires that 
\begin{align}
    (x^{-n-1})^{(m)} + y^{-n-m-1}\sum_{p\geq 0} C_{n,p+m}\binom{p+m}{m} \left(\frac{x}{y}\right)^p= \iota_{yx}g_{mn}(y+x,y)\label{C_mn-g-formula}
\end{align}
which is the same equation required in Lemma 3.31 in \cite{FFR}. Thus $C_{mn}$ should be chosen in the same way as in \cite{FFR}, Formula (3.55) and (3.56), namely, 
$$C_{mn}=\frac 1 2 \frac {m-n}{m+n+1} \binom{-1/2}{m} \binom{-1/2}{n}.$$
Moreover, the identity
\begin{align}
    \sum_{m=0}^k \binom{m+r} r \binom{k-m+t} t C_{(m+r)(k-m+t)} = \binom{-r-t-1}k C_{rt}\label{C_rt-formula}
\end{align}
holds (cf. Lemma 3.9, \cite{FFR}). 

With such choice of $C_{mn}$, we enhance Proposition 5.8 in \cite{FQ} as follows

\begin{prop}\label{Exp-Delta-Neg-Comm}
    For $a\in \h, m\in \N$, 
    \begin{align}
        [\exp(\Delta(y)), a^{(m)}(x)_V^-] &= \sum_{\alpha\geq 0}\sum_{\beta\geq 0}C_{\beta\alpha}a(\beta + 1/2)y^{-\beta-\alpha-1}(x^\alpha)^{(m)}\label{Exp-Delta-Neg-Comm-1} \exp(\Delta(y))\\
        &= \sum_{\beta\geq 0}a(\beta+1/2)\left(\iota_{yx}g_{m\beta}(y+x, y) - (x^{-\beta-1})^{(m)}\right)\exp(\Delta(y))\label{Exp-Delta-Neg-Comm-2}
    \end{align}
\end{prop}

\subsection{Iterate formula}

Recall from Corollary 4.6 of \cite{FQ} that 
\begin{align*}
        & :a_1^{(m_1)}(x)_V\cdots a_r^{(m_r)}(x)_V:b_1(-n_1-1/2)\cdots b_s(-n_s-1/2)\one \\
        = \ & \sum_{\rho=0}^{\min(r,s)}\sum_{\substack{ 1 \leq i_1 < \cdots < i_\rho\leq r\\ 1 \leq j_1 < \cdots < j_\rho \leq s}}(-1)^{i_1+\cdots + i_\rho + j_1+\cdots + j_\rho }(-1)^{r\rho + \rho(\rho+1)/2}\nonumber \\
        & \hspace{9.7 em}\cdot \begin{vmatrix}
            (a_{i_1}, b_{j_1})(x^{-n_{j_1}-1})^{(m_{i_1})} & \cdots & (a_{i_1}, b_{j_\rho})(x^{-n_{j_\rho}-1})^{(m_{i_1})}\\
            \vdots & & \vdots\nonumber  \\
            (a_{i_\rho}, b_{j_1})(x^{-n_{j_1}-1})^{(m_{i_\rho})} & \cdots & (a_{i_\rho}, b_{j_\rho})(x^{-n_{j_\rho}-1})^{(m_{i_\rho})} 
        \end{vmatrix}\nonumber \\
        & \hspace{10 em}\cdot (a_{1}^{(m_1)})_V(x)^- \cdots \reallywidehat{(a_{i_1}^{(m_{i_1})})_V(x)^-} \cdots \reallywidehat{(a_{i_\rho}^{(m_{i_\rho})})_V(x)^-} \cdots (a_{r}^{(m_r)})_V(x)^-\nonumber  \\
        & \hspace{10 em}\cdot b_{1}(-n_1/1-2) \cdots \reallywidehat{b_{j_1}(-n_{j_1}-1/2)} \cdots \reallywidehat{b_{j_\rho}(-n_{j_\rho}-1/2)} \cdots b_{s}(-n_s-1/2)\one
    \end{align*}

    \begin{thm}
        For $r, s \in \N, a_1, ..., a_r, b_1, ..., b_s\in \h, m_1, ..., m_r, n_1, ..., n_s\in \N$, 
        \begin{align}
        & \exp(\Delta(y)):a_1^{(m_1)}(x)_V\cdots a_r^{(m_r)}(x)_V: b_1(-n_1-1/2)\cdots b_s(-n_s-1/2)\one \nonumber \\
        = \ & \sum_{k= 0}^\infty \sum_{\substack{1\leq p_1 < \cdots < p_{2k} \leq r}}  (-1)^{p_1+\cdots + p_{2k}}T_{a_1, ..., a_r}^{m_1,...,m_r}(p_1, ..., p_{2k})\iota_{yx}(y+x)^{-m_{p_1}-\cdots - m_{p_{2k}} - k} \nonumber \\
        & \quad \cdot \sum_{l= 0}^\infty \sum_{\substack{
        1\leq q_1 < \cdots < q_{2l} \leq s}}(-1)^{q_1+\cdots + q_{2l}}T_{b_1, ..., b_s}^{n_1, ..., n_s}(q_1, ..., q_{2l}) y^{-n_{q_1}-\cdots - n_{q_{2l}} - l}\nonumber \\
        & \qquad \cdot \sum_{\rho=0}^{\min\{r-2k, s-2l\}}\sum_{\substack{1\leq i_1 < \cdots < i_\rho\leq r-2k \\
        1\leq j_1 < \cdots < j_\rho \leq s-2l}}(-1)^{i_1+\cdots + i_{\rho} + j_1 + \cdots + j_\rho} (-1)^{(r-2k)\rho + \frac{\rho(\rho+1)}{2}}\nonumber \\
        & \qquad\qquad \qquad \cdot \begin{vmatrix}
            (a_{p_{i_1}^c}, b_{q_{j_1}^c})\iota_{yx}g_{m_{p_{i_1}^c}n_{q_{j_1}^c}}(x+y, y) & \cdots & (a_{p_{i_1}^c}, b_{q_{j_\rho}^c})\iota_{yx}g_{m_{p_{i_1}^c}n_{q_{j_\rho}^c}}(x+y, y)\\
            \vdots & & \vdots \\
            (a_{p_{i_\rho}^c}, b_{q_{j_1}^c})\iota_{yx}g_{m_{p_{i_\rho}^c}n_{q_{j_1}^c}}(x+y, y) & \cdots & (a_{p_{i_\rho}^c}, b_{q_{j_\rho}^c})\iota_{yx}g_{m_{p_{i_\rho}^c}n_{q_{j_\rho}^c}}(x+y, y)
        \end{vmatrix}\nonumber \\
        & \qquad\qquad \qquad \cdot a_{p_1^c}^{(m_{p_1^c})}(x)_V^-\cdots \reallywidehat{a_{p_{i_1}^c}^{(m_{p_{i_1}^c})}(x)_V^-} \cdots \reallywidehat{a_{p_{i_\rho}^c}^{(m_{p_{i_\rho}^c})}(x)_V^-} \cdots a_{p_{r-2k}^c}(x)_V^- \nonumber \\
        & \qquad \qquad \qquad \cdot b_{q_1^c}(-n_{q_1^c}-1/2)\cdots \reallywidehat{b_{q_{j_1}^c}(-n_{q_{j_1}^c}-1/2)}\cdots\reallywidehat{b_{q_{j_\rho}^c}(-n_{q_{j_\rho}^c}-1/2)}\cdots b_{q_{s-2l}^c}(-n_{q_{s-2l}^c}-1/2)\one \label{exp-Delta-A-B}
    \end{align}
    as a series in $V[[x, x^{-1}, y, y^{-1}]]$. Here for $(p_1, \cdots , p_{2k})$, $(q_1, ..., q_{2l})$, $(i_1, ..., i_\rho)$ and $(j_1, ..., j_\rho)$, the sequences $1\leq p_1^c < \cdots < p_{r-2k}^c \leq r$, $1\leq q_1^c < \cdots < q_{s-2l}^c\leq s$, $1\leq i_1^c < \cdots < i_{r-2k-\rho} \leq r-2k$ and $1\leq j_1^c < \cdots < j_{s-2l-\rho}^c \leq s-2l$ are their complementary sequence, respectively. 
\end{thm}

\begin{proof}
    We first explain the shuffles hidden in the formula. 
    Note that for fixed $1\leq p_1 < \cdots < p_{2k} \leq r$ and $1\leq i_1 < \cdots < i_\rho \leq r-2k$, the sequences 
    \begin{align*}
        & 1\leq p_1 < \cdots < p_{2k}\leq r, 1\leq p_{i_1}^c < \cdots < p_{i_\rho}^c\leq r, 1\leq p_{i_1^c}^c < \cdots < p^c_{i_{r-2k-\rho}^c} \leq r
    \end{align*}
    indeed forms a 3-shuffle that is understood as an ``iterate'' of two 2-shuffles, namely, $\sigma\in J_{2k}(1, ..., r)$ satisfying 
    $$\sigma(\alpha) = p_\alpha, 1\leq \alpha\leq 2k, \sigma(2k+\alpha) = p_\alpha^c, 1\leq \alpha \leq r-2k$$ 
    and $\tau\in J_{\rho}(2k+1, ..., r)$ satisfying 
    $$\tau(2k+\alpha) = i_\alpha, 1 \leq \alpha \leq \rho, \tau(2k+\rho+\alpha) = i_\alpha^c, $$
    So we have 
    $$p_{i_1}^c = (\sigma\circ\tau)(2k+1), ..., p_{i_\rho}^c =(\sigma\circ\tau)(2k+\rho), $$    $$p_{i_1^c}^c = (\sigma\circ\tau)(2k+\rho+1), ...,  p_{i_{r-2k-\rho}^c}^c = (\sigma\circ\tau)(r). $$
    From Lemma \ref{2-shuffle-lemma}, we have
    $$(-1)^{p_1 + \cdots + 2k} = (-1)^\sigma (-1)^{\frac{2k(2k+1)}{2}} = (-1)^\sigma (-1)^{k}$$  $$(-1)^{i_1 + \cdots + i_\rho} = (-1)^\tau (-1)^{\frac{\rho(\rho+1)}{2}}. $$
    We shall also use the notations in Notation \ref{Index-Notation-1} to emphasize the role of the subscripts. With the discussion above, we may rewrite Formula (\ref{exp-Delta-A-B}) in terms of 2-shuffles: 
    \begin{align}
        & \exp(\Delta(y)):a_1^{(m_1)}(x)_V\cdots a_r^{(m_r)}(x)_V: b_1(-n_1-1/2)\cdots b_s(-n_s-1/2)\one \nonumber \\
        = \ & \sum_{k= 0}^\infty \sum_{\sigma_a\in J_{2k}(1, ..., r) }  (-1)^{\sigma_a}(-1)^{k}T_{a_1, ..., a_r}^{m_1,...,m_r}(\sigma_a(1), ..., \sigma_a(2k))\iota_{yx}(y+x)^{-m_{\sigma_a(1)}-\cdots - m_{\sigma_a(2k)} - k} \nonumber \\
        &  \cdot \sum_{l= 0}^\infty \sum_{\sigma_b\in J_{2l}(1, .., s)}(-1)^{\sigma_b}(-1)^{l}T_{b_1, ..., b_s}^{n_1, ..., n_s}(\sigma_b(1), ..., \sigma_b(2l)) y^{-n_{\sigma_b(1)}-\cdots - n_{\sigma_b(2l)} - l}\nonumber \\
        &  \cdot \sum_{\rho=0}^{\min\{r-2k, s-2l\}}\sum_{\substack{\tau_a\in J_\rho(2k+1,..., r)}}\sum_{\substack{\tau_b\in J_{\rho}(2l+1, ..., s)}}(-1)^{\tau_a}(-1)^{\tau_b}  (-1)^{r\rho + \frac{\rho(\rho+1)}{2}}\nonumber  \\
        &   \qquad \qquad \cdot \begin{vmatrix}
            \langle (\sigma_a\circ \tau_a)(2k+1), (\sigma_b\circ\tau_b)(2l+1)\rangle_g & \cdots & \langle (\sigma_a\circ \tau_a)(2k+1), (\sigma_b\circ\tau_b)(2l+\rho)\rangle_g\\
            \vdots & & \vdots \\
            \langle (\sigma_a\circ \tau_a)(2k+\rho), (\sigma_b\circ\tau_b)(2l+1)\rangle_g & \cdots & \langle (\sigma_a\circ \tau_a)(2k+\rho), (\sigma_b\circ\tau_b)(2l+\rho)\rangle_g
        \end{vmatrix}\nonumber\\
        & \qquad \qquad \cdot S(x)_{a_1, ..., a_r}^{(m_1, ..., m_r)}((\sigma_a\circ\tau_a)(2k+\rho+1), ..., (\sigma_a\circ \tau_a)(r)) \nonumber \\
        & \qquad \qquad \cdot S_{b_1, ..., b_s}^{n_1, ..., n_s}((\sigma_b\circ\tau_b)(2l+\rho+1), ..., (\sigma_b\circ\tau_b)(s))\label{exp-Delta-A-B-shuffle}
    \end{align}
    Now, we prove by induction. The conclusion clearly holds when $r=0$ and $s$ is arbitrary. Assume the conclusion for smaller $r$ and arbitrary $s$. Recall from Proposition 3.10 in \cite{FQ} that
    \begin{align*}
        :a_1^{(m_1)}(x)_V \cdots a_r^{(m_r)}(x)_V:  =\ &  a_1^{(m_1)}(x)_V^-:a_2^{(m_2)}(x)_V \cdots a_r^{(m_r)}(x)_V: \\
        & + (-1)^{r-1}:a_2^{(m_2)}(x)_V \cdots a_r^{(m_r)}(x)_V: a_1^{(m_1)}(x)_V^+
    \end{align*}
    Thus 
    \begin{align}
        & \exp(\Delta(y)):a_1^{(m_1)}(x)_V \cdots a_r^{(m_r)}(x)_V: \nonumber\\
        =\ &  \exp(\Delta(y))a_1^{(m_1)}(x)_V^-:a_2^{(m_2)}(x)_V \cdots a_r^{(m_r)}(x)_V: \nonumber\\
        & + (-1)^{r-1}\exp(\Delta(y)):a_2^{(m_2)}(x)_V \cdots a_r^{(m_r)}(x)_V: a_1^{(m_1)}(x)_V^+\nonumber\\
        =\ &  a_1^{(m_1)}(x)_V^-\exp(\Delta(y)) :a_2^{(m_2)}(x)_V \cdots a_r^{(m_r)}(x)_V: \label{exp-Delta-A-B-1}\\
        & + \sum_{\alpha\geq 0}\sum_{\beta\geq 0}C_{\beta\alpha}a(\beta + 1/2)y^{-\beta-\alpha-1}(x^\alpha)^{(m)}\exp(\Delta(y)) :a_2^{(m_2)}(x)_V \cdots a_r^{(m_r)}(x)_V:\label{exp-Delta-A-B-2}\\
        & + (-1)^{r-1}\exp(\Delta(y)):a_2^{(m_2)}(x)_V \cdots a_r^{(m_r)}(x)_V: a_1^{(m_1)}(x)_V^+\label{exp-Delta-A-B-3}
    \end{align}
    where we used the conclusion of Proposition \ref{Exp-Delta-Neg-Comm}. We apply each of (\ref{exp-Delta-A-B-1}), (\ref{exp-Delta-A-B-2}) and (\ref{exp-Delta-A-B-3}) to $b_1(-n_1-1/2) \cdots b_s(-n_s-1/2)\one$ and apply induction to see that they are parts of (\ref{exp-Delta-A-B-shuffle})

    \noindent \textbf{Application of (\ref{exp-Delta-A-B-1}): }We obtain 
    \begin{align*}
        & \sum_{k=0}^\infty \sum_{\sigma_a\in J_{2k}(2, ..., r)}(-1)^{\sigma_a}(-1)^{k}T_{a_1, a_2, ..., a_r}^{m_1, m_2, ..., m_r}(\sigma_a(2), ..., \sigma_a(2k+1))\iota_{yx}(y+x)^{-m_{\sigma_a(2)}-\cdots - m_{\sigma_a(2k+1)}-k}\\
        & \cdot \sum_{l=0}^\infty \sum_{\sigma_b\in J_{2l}(1, ..., s)}(-1)^{\sigma_b}(-1)^l T_{b_1, ..., b_s}^{n_1, ..., n_s}(\sigma_b(1), ..., \sigma_b(2l))y^{-n_{\sigma_b(1)}-\cdots - n_{\sigma_b(2l)}-l}\\
        & \cdot \sum_{\rho=0}^{\min\{r-1-2k, s-2l\}}\sum_{\tau_a\in J_\rho(2k+2, ..., r)}\sum_{\tau_b\in J_\rho(2l+1, ..., s)}(-1)^{\tau_a}(-1)^{\tau_b} (-1)^{(r-1)\rho+\frac{\rho(\rho+1)}2}\\
        & \qquad \qquad\cdot  \begin{vmatrix}
            \langle (\sigma_a\circ \tau_a)(2k+2), \tilde{q}_{2l+1}\rangle_g & \cdots & \langle (\sigma_a\circ \tau_a)(2k+2), \tilde{q}_{2l+\rho}\rangle_g\\
            \vdots & & \vdots \\
            \langle (\sigma_a\circ \tau_a)(2k+1+\rho), \tilde{q}_{2l+1}\rangle_g & \cdots & \langle (\sigma_a\circ \tau_a)(2k+1+\rho), \tilde{q}_{2l+1}\rangle_g
        \end{vmatrix}\nonumber\\
        &\qquad \qquad  \cdot a_1^{(m_1)}(x)_V^- \cdot S(x)_{a_1, ..., a_{r}}^{m_1, ..., m_r}((\sigma_a\circ\tau_a)(2k+\rho+2), ...,(\sigma_a\circ\tau_a)(r))\\
        &\qquad \qquad \cdot S_{b_1, ..., b_s}^{n_1, ..., n_s}(\tilde{q}_{2l+\rho+1}, ..., \tilde{q}_{s})
    \end{align*}
    Here $\tilde{q}_\alpha$ is the abbreviation of $(\sigma_b\circ\tau_b)(\alpha)$ for $\alpha = 2l+1, ..., s$. This abbreviation is introduced only to make the formula shorter. Note that we should still use $T_{a_1, a_2, ..., a_r}^{m_1, m_2, ..., m_r}$ in the induction hypothesis. Otherwise, the indices $\sigma_a(\alpha)$ in the arguments must be shifted by 1, becoming $\sigma_a(\alpha) - 1$ for $\alpha = 2, ..., 2k+1$. 
    
    Conceptually, this expression should contribute to the part of (\ref{exp-Delta-A-B-shuffle}) where $1$ is placed in the third sequence of the corresponding 3-shuffle. We now proceed to justify this conceptual understanding. We first restrict our attention to the summand with fixed $l, \sigma_b, \rho, \tau_b$, then further restrict our attention to the summand with fixed $k$ and $\sigma_a$. 
    \begin{itemize}
        \item We may regard $\sigma_a$ as a permutation of $\{1, ..., r\}$ satisfies $\sigma_a(1) = 1$. Define $\sigma_a^{new}\in J_{2k}(1, ..., r)$ by 
    $$\sigma_a^{new} =\begin{pmatrix}
        1 & 2 & \cdots & 2k & 2k+1 & 2k+2 & \cdots & r\\
        \sigma_a(2) & \sigma_a(3) & \cdots & \sigma_a(2k+1) & 1 & \sigma_a(2k+2) & \cdots &\sigma_a(r)
    \end{pmatrix}$$
    Then $\sigma_a^{new}$ corresponds to two increasing sequences 
    $$\sigma_a(2)< \cdots < \sigma_a(2k+1), 1 < \sigma_a(2k+2) < \cdots < \sigma_a(r)$$
    where the first sequence coincides with that of $\sigma_a$, and the second series is modified from that of $\sigma_a$ by attaching 1 in the front. Moreover, 
    $$\sigma_a^{new} = \sigma_a \circ (2k+1, 2k,..., 1)\Rightarrow (-1)^{\sigma_a^{new}}=(-1)^{\sigma_a} (-1)^{2k} = (-1)^{\sigma_a}.$$
    Summing over $\sigma_a\in J_{2k}(2, ..., r)$ amounts to summing over $\sigma_a^{new}\in J_{2k}(1,..., r)$ with $\sigma_a^{new}(2k+1) = 1 \Leftrightarrow \colorbox{ffff00}{$\sigma_a^{new}(1) > 1$}$. 
    \end{itemize}
    We now further focus on the summand for fixed $\tau_a$. 
    \begin{itemize}
        \item Similarly, we may regard $\tau_a\in J_\rho(2k+2, ..., r)$ as a permutation of $\{2k+1, 2k+2, ..., r\}$ such that $\tau_a(2k+1) = 2k+1$. Define $\tau_a^{new}\in J_{\rho}(2k+1, ..., r)$ by 
        \begin{align*}
            \tau_a^{new} = \begin{pmatrix}
                2k+1 & 2k+2 & \cdots & 2k+\rho & 2k+\rho+1 & 2k+\rho+2 & \cdots & r\\
                \tau_{a}(2k+2) & \tau_a(2k+3) & \cdots & \tau_a(2k+\rho+1) & 2k+1 & \tau_a(2k+\rho+2) & \cdots & \tau(r)
            \end{pmatrix}
        \end{align*}
        Then $\tau_a^{new}$ corresponds to two increasing series 
        $$\tau_a(2k+2)< \cdots < \tau_a(2k+\rho+1), 2k+1 < \tau_a(2k+\rho+2) < \cdots < \tau_a(r)$$
        where the first sequence coincides with that of $\tau_a$, and the second series is modified from that of $\tau_a$ by attaching $2k+1$ in the front. Moreover, 
        $$\tau_a^{new} = \tau_a \circ (2k+\rho+1, 2k+\rho, ..., 2k+1)\Rightarrow (-1)^{\tau_a^{new}} = (-1)^{\tau_a}(-1)^{\rho}. $$
        Summing over $\tau_a\in J_{\rho}(2k+2, ..., r)$ amounts to summing over $\tau_a^{new}\in J_{\rho}(2k+1, ..., r)$ with $\tau_a^{new}(2k+\rho+1) = 2k+1 \Leftrightarrow \colorbox{ffff00}{$\tau_a^{new}(2k+1)>2k+1$}$. 
        \item For $i=1, ..., \rho$, the $(\sigma_a\circ\tau_a)(2k+i+1)$ appearing in in the determinant now becomes 
        $(\sigma_a^{new} \circ \tau_a^{new})(2k+i)$. 
        \item For $i=\rho+1, ..., r-1$, the $(\sigma_a\circ \tau_a)(2k+i+1)$ appearing in the generating functions of negative modes now becomes $(\sigma_a^{new}\circ\tau_a^{new})(2k+i)$. In particular, $(\sigma_a^{new}\circ\tau_a^{new})(2k+\rho+1) = \sigma_a^{new}(2k+1) = 1$. 
    \end{itemize}
    Therefore, we showed that the application of (\ref{exp-Delta-A-B-1}) on $b_1(-n_1-1/2)\cdots b_s(-n_s-1/2)\one$ results in the following series (with the removal of the $new$-superscript)
    \begin{align}
        & \sum_{k= 0}^\infty \sum_{\substack{\sigma_a\in J_{2k}(1, ..., r)\\
        {\scriptsize \colorbox{ffff00}{$\sigma_a(1) > 1$}}}}  (-1)^{\sigma_a}(-1)^{k}T_{a_1, ..., a_r}^{m_1,...,m_r}(\sigma_a(1), ..., \sigma_a(2k))\iota_{yx}(y+x)^{-m_{\sigma_a(1)}-\cdots - m_{\sigma_a(2k)} - k} \nonumber \\
        &  \cdot \sum_{l= 0}^\infty \sum_{\sigma_b\in J_{2l}(1, .., s)}(-1)^{\sigma_b}(-1)^{l}T_{b_1, ..., b_s}^{n_1, ..., n_s}(\sigma_b(1), ..., \sigma_b(2l)) y^{-n_{\sigma_b(1)}-\cdots - n_{\sigma_b(2l)} - l}\nonumber \\
        &  \cdot \sum_{\rho=0}^{\min\{r-2k, s-2l\}}\sum_{\substack{\tau_a\in J_\rho(2k+1,..., r)\\
        {\scriptsize \colorbox{ffff00}{$\tau_a(2k+1) > 2k+1$}}}}
        \sum_{\substack{\tau_b\in J_{\rho}(2l+1, ..., s)}}(-1)^{\tau_a}(-1)^{\tau_b}  (-1)^{r\rho + \frac{\rho(\rho+1)}{2}}\nonumber  \\
        &   \qquad \qquad \cdot \begin{vmatrix}
            \langle (\sigma_a\circ \tau_a)(2k+1), (\sigma_b\circ\tau_b)(2l+1)\rangle_g & \cdots & \langle (\sigma_a\circ \tau_a)(2k+1), (\sigma_b\circ\tau_b)(2l+\rho)\rangle_g\\
            \vdots & & \vdots \\
            \langle (\sigma_a\circ \tau_a)(2k+\rho), (\sigma_b\circ\tau_b)(2l+1)\rangle_g & \cdots & \langle (\sigma_a\circ \tau_a)(2k+\rho), (\sigma_b\circ\tau_b)(2l+\rho)\rangle_g
        \end{vmatrix}\nonumber\\        
        & \qquad \qquad \cdot S(x)_{a_1, ..., a_r}^{(m_1, ..., m_r)}((\sigma_a\circ\tau_a)(2k+\rho+1), ..., (\sigma_a\circ \tau_a)(r)) \nonumber \\
        & \qquad \qquad \cdot S_{b_1, ..., b_s}^{n_1, ..., n_s}((\sigma_b\circ\tau_b)(2l+\rho+1), ..., (\sigma_b\circ\tau_b)(s))\label{exp-Delta-A-B-F1}
    \end{align}

    \noindent\textbf{Application of (\ref{exp-Delta-A-B-2}).} We obtain
    \begin{align*}
        & \sum_{k=0}^\infty \sum_{\sigma_a\in J_{2k}(2, ..., r)}(-1)^{\sigma_a}(-1)^{k}T_{a_1, a_2, ..., a_r}^{m_1, m_2, ..., m_r}(\sigma_a(2), ..., \sigma_a(2k+1))\iota_{yx}(y+x)^{-m_{\sigma_a(2)}-\cdots - m_{\sigma_a(2k+1)}-k}\\
        & \cdot \sum_{l=0}^\infty \sum_{\sigma_b\in J_{2l}(1, ..., s)}(-1)^{\sigma_b}(-1)^l T_{b_1, ..., b_s}^{n_1, ..., n_s}(\sigma_b(1), ..., \sigma_b(2l))y^{-n_{\sigma_b(1)}-\cdots - n_{\sigma_b(2l)}-l}\\
        & \cdot \sum_{\rho=0}^{\min\{r-1-2k, s-2l\}}\sum_{\tau_a\in J_\rho(2k+2, ..., r)}\sum_{\tau_b\in J_\rho(2l+1, ..., s)}(-1)^{\tau_a}(-1)^{\tau_b} (-1)^{(r-1)\rho+\frac{\rho(\rho+1)}2}\\
        & \qquad \qquad\cdot  \begin{vmatrix}
            \langle (\sigma_a\circ \tau_a)(2k+2), \tilde{q}_{2l+1}\rangle_g & \cdots & \langle (\sigma_a\circ \tau_a)(2k+2), \tilde{q}_{2l+\rho}\rangle_g\\
            \vdots & & \vdots \\
            \langle (\sigma_a\circ \tau_a)(2k+1+\rho), \tilde{q}_{2l+1}\rangle_g & \cdots & \langle (\sigma_a\circ \tau_a)(2k+1+\rho), \tilde{q}_{2l+\rho}\rangle_g
        \end{vmatrix}\nonumber\\
        & \qquad \qquad \sum_{\alpha\geq 0}\sum_{\beta\geq 0} C_{\beta\alpha}a_1(\beta+1/2)y^{-\beta-\alpha-1}(x^\alpha)^{(m_1)} \\
        &\qquad \qquad  \qquad \quad \cdot  S(x)_{a_1, ..., a_{r}}^{m_1, ..., m_r}((\sigma_a\circ\tau_a)(2k+\rho+2), ...,(\sigma_a\circ\tau_a)(r))\\
        &\qquad \qquad \qquad \quad \cdot S_{b_1, ..., b_s}^{n_1, ..., n_s}(\tilde{q}_{2l+\rho+1}, ..., \tilde{q}_{s})
    \end{align*}
    We focus on the last three lines:
    \begin{align}
        & \sum_{\alpha\geq 0}\sum_{\beta\geq 0}C_{\beta\alpha}a_1(\beta+1/2)y^{-\beta-\alpha-1}(x^\alpha)^{m}\nonumber\\
        &\qquad \quad  \cdot S(x)_{a_1, ..., a_{r}}^{m_1, ..., m_r}((\sigma_a\circ\tau_a)(2k+\rho+2), ...,(\sigma_a\circ\tau_a)(r))\nonumber\\
        &\qquad \quad \cdot S_{b_1, ..., b_s}^{n_1, ..., n_s}(\tilde{q}_{2l+\rho+1}, ..., \tilde{q}_{s}) \nonumber\\
        = \ & \sum_{\alpha, \beta\geq 0}\sum_{\lambda=2k+2+\rho}^{r} (-1)^{\lambda-2k-2-\rho} C_{\beta\alpha} y^{-\beta-\alpha-1}(x^\alpha)^{(m_1)} \nonumber\\
        &\qquad \quad  \cdot a_{(\sigma_a\circ\tau_a)(2k+2+\rho)}^{(m_{(\sigma_a\circ\tau_a)(2k+2+\rho)})}(x)_V^-\cdots \left\{a_1(\beta+1/2), a_{(\sigma_a\circ\tau_a)(\lambda)}^{(m_{(\sigma_a\circ\tau_a)(\lambda)})}(x)^- \right\}\cdots a_{(\sigma_a\circ\tau_a)(r)}^{(m_{(\sigma_a\circ\tau_a)(r)})}(x)_V^-\nonumber\\
        &\qquad \quad \cdot b_{\tilde{q}_{2l+\rho+1}}(-n_{\tilde{q}_{2l+\rho+1}}-1/2) \cdots b_{\tilde{q}_s}(-n_{\tilde{q}_s}-1/2)\one \nonumber\\
        & + \sum_{\beta\geq 0}\sum_{\lambda=2l+\rho+1}^{s}(-1)^{r-2k-2-\rho+1} (-1)^{\lambda-2l-\rho-1}\left(\iota_{yx}g_{m_1 \beta}(y+x, y)-(x^{-\beta-1})^{(m_1)}\right) \nonumber\\
        &\qquad \qquad  \cdot a_{(\sigma_a\circ\tau_a)(2k+2+\rho)}^{(m_{(\sigma_a\circ\tau_a)(2k+2+\rho)})}(x)_V^-\cdots a_{(\sigma_a\circ\tau_a)(r)}^{(m_{(\sigma_a\circ\tau_a)(r)})}(x)_V^-\nonumber\\
        &\qquad \qquad \cdot b_{\tilde{q}_{2l+\rho+1}}(-n_{\tilde{q}_{2l+\rho+1}}-1/2)\cdots \left\{a_1(\beta+1/2), b_{\tilde{q}_\lambda}(-n_{\tilde{q}_{\lambda}}-1/2)\right\} \cdots b_{\tilde{q}_s}(-n_{\tilde{q}_s}-1/2)\one \nonumber\\
        = \ & \sum_{\alpha, \beta\geq 0}\sum_{\lambda=2k+2+\rho}^{r} (-1)^{\lambda-\rho} C_{\beta\alpha} y^{-\beta-\alpha-1}(x^\alpha)^{(m_1)} (a_1, a_{(\sigma_a\circ\tau_a)(\lambda)}) (x^{\beta})^{(m_{(\sigma_a\circ\tau_a)(\lambda)})}\nonumber\\
        &\qquad \quad  \cdot S(x)_{a_1, ..., a_r}^{m_1, ..., m_r}\left((\sigma_a\circ\tau_a)(2k+2+\rho), ..., \reallywidehat{(\sigma_a\circ\tau_a)(\lambda)}, ...., (\sigma_a\circ\tau_a)(r)\right)\nonumber\\
        &\qquad \quad \cdot S_{b_1, ..., b_s}^{n_1, ..., n_s}(\tilde{q}_{2l+\rho+1}, ... \tilde{q}_s) \label{exp-Delta-A-B-2-1}\\
        & + \sum_{\lambda=2l+\rho+1}^{s} (-1)^{\lambda+r} \left(\iota_{yx}g_{m_1 n_{\tilde{q}_\lambda}}(y+x, y)-(x^{-n_{\tilde{q}_\lambda}-1})^{(m_1)}\right) (a_1, b_{\tilde{q}_\lambda})\nonumber\\
        &\qquad \qquad  \cdot S(x)_{a_1, ..., a_{r}}^{m_1, ..., m_r}((\sigma_a\circ\tau_a)(2k+\rho+2), ...,(\sigma_a\circ\tau_a)(r))\nonumber\\
        &\qquad \qquad \cdot S_{b_1, ..., b_s}^{n_1, ..., n_s}(\tilde{q}_{2l+\rho+1}, ..., \widehat{\tilde{q}_{\lambda}}, ..., \tilde{q}_{s}) \label{exp-Delta-A-B-2-2}
    \end{align}
    where the last equality follows from Proposition \ref{Exp-Delta-Neg-Comm}. 

    In what follows, we will separately discuss the contribution by (\ref{exp-Delta-A-B-2-1}) and by (\ref{exp-Delta-A-B-2-2}) to the application of (\ref{exp-Delta-A-B-2}).

    \noindent\textbf{$\blacktriangleright$ Contribution to the application of (\ref{exp-Delta-A-B-2}) by (\ref{exp-Delta-A-B-2-1})}. 
    Note that for every $\lambda'=2, ..., r$, 
    \begin{align*}
        & \sum_{\alpha, \beta\geq 0} C_{\beta\alpha}y^{-\beta-\alpha-1} (x^\alpha)^{(m_1)} (x^\beta)^{(m_{\lambda'})} =  \sum_{\alpha, \beta\geq 0} C_{\beta\alpha}y^{-\beta-\alpha-1} \binom{\alpha}{m_1}\binom{\beta}{m_{\lambda'}}x^{\alpha+\beta-m_1-m_{\lambda'}} \\
        = \ & \sum_{\alpha, \beta\geq 0} C_{\beta+m_{\lambda'},\alpha+m_1}y^{-\beta-\alpha-m_1-m_{\lambda'}-1} \binom{\alpha+m_1}{m_1}\binom{\beta+m_{\lambda'}}{m_{\lambda'}}x^{\alpha+\beta} \\
        = \ & \sum_{\alpha\geq 0} \sum_{\beta= 0}^\alpha C_{\beta+m_{\lambda'},\alpha-\beta +m_1}y^{-\alpha-m_1-m_{\lambda'}-1} \binom{\alpha-\beta+m_1}{m_1}\binom{\beta+m_{\lambda'}}{m_{\lambda'}}x^{\alpha}\\
        = \ & \sum_{\alpha\geq 0}\binom{-m_{\lambda'}-m_1-1}{\alpha}C_{m_{\lambda'}m_1} \cdot y^{-\alpha-m_1-m_{\lambda'}-1}x^\alpha     =  -C_{m_1m_{\lambda'}} \iota_{yx}(y+x)^{-m_1-m_{\lambda'}-1},
    \end{align*}
    where we used Formula (\ref{C_rt-formula}). Thus the part of (\ref{exp-Delta-A-B-2}) contributed by (\ref{exp-Delta-A-B-2-1}) is 
    \begin{align*}
        & \sum_{k=0}^\infty \sum_{\sigma_a\in J_{2k}(2, ..., r)}(-1)^{\sigma_a}(-1)^{k}T_{a_2, ..., a_r}^{m_2, ..., m_r}(\sigma_a(2), ..., \sigma_a(2k+1))\iota_{yx}(y+x)^{-m_{\sigma_a(2)}-\cdots - m_{\sigma_a(2k+1)}-k}\\
        & \cdot \sum_{l=0}^\infty \sum_{\sigma_b\in J_{2l}(1, ..., s)}(-1)^{\sigma_b}(-1)^l T_{b_1, ..., b_s}^{n_1, ..., n_s}(\sigma_b(1), ..., \sigma_b(2l))y^{-n_{\sigma_b(1)}-\cdots - n_{\sigma_b(2l)}-l}\\
        & \cdot \sum_{\rho=0}^{\min\{r-1-2k, s-2l\}}\sum_{\tau_a\in J_\rho(2k+2, ..., r)}\sum_{\tau_b\in J_\rho(2l+1, ..., s)}(-1)^{\tau_a}(-1)^{\tau_b} (-1)^{(r-1)\rho+\frac{\rho(\rho+1)}2}\\
        & \qquad \qquad\cdot  \begin{vmatrix}
            \langle (\sigma_a\circ \tau_a)(2k+2), \tilde{q}_{2l+1}\rangle_g & \cdots & \langle (\sigma_a\circ \tau_a)(2k+2), \tilde{q}_{2l+\rho}\rangle_g\\
            & & \\
            \langle (\sigma_a\circ \tau_a)(2k+1+\rho), \tilde{q}_{2l+1}\rangle_g & \cdots & \langle (\sigma_a\circ \tau_a)(2k+1+\rho), \tilde{q}_{2l+\rho}\rangle_g
        \end{vmatrix}\nonumber\\
        & \qquad \qquad \qquad \cdot \sum_{\lambda = 2k + 2 + \rho}^r (-1)^{\lambda-\rho-1} \left(a_1, a_{(\sigma_a\circ \tau_a)(\lambda)}\right) C_{m_1 m_{(\sigma_a\circ\tau_a)(\lambda)}}\iota_{yx}(y+x)^{-m_1-m_{(\sigma_a\circ\tau_a)(\lambda)}-1}\\
        &\qquad \qquad \qquad \qquad \quad  \cdot S(x)_{a_1,...,a_r}^{m_1,...,m_r}((\sigma_a\circ\tau_a)(2k+2+\rho), ..., \reallywidehat{(\sigma_a\circ\tau_a)(\lambda)}, ..., (\sigma_a\circ\tau_a)(r))\\
        &\qquad \qquad \qquad \qquad \quad \cdot S_{b_1, ..., b_s}^{n_1, ..., n_s}(\tilde{q}_{2l+\rho+1}, ... \tilde{q}_s) 
    \end{align*}
    We now use Remark \ref{2-shuffle-summation} to change the order of summation. More precisely, we will perform the following changes
        \begin{align*}
            & \sum_{\sigma_a \in J_{2k}(2, ..., r)} (-1)^{\sigma_a}\sum_{\tau_a\in J_{\rho}(2k+2, ..., r)}(-1)^{\tau_a}\sum_{\lambda=2k+2+\rho}^{r}(-1)^{\lambda-2k-2-\rho-1}\\
            \mapsto \ & \sum_{\sigma_a \in J_{2k}(2, ..., r)} (-1)^{\sigma_a}\sum_{\lambda=2k+2}^{r} (-1)^{\lambda-2k-2-1}\sum_{\tau_a\in J_{\rho}(2k+2,...,\widehat{\lambda}, ..., r)}(-1)^{\tau_a}
        \end{align*}
    and $(\sigma_a\circ \tau_a)(\lambda) \mapsto \sigma_a(\lambda)$. So the contribution is rewritten as 
    \begin{align*}
        & \sum_{k=0}^\infty \sum_{\sigma_a\in J_{2k}(2, ..., r)}\sum_{\lambda=2k+2}^r (-1)^{\lambda -1}\left(a_1, a_{\sigma_a(\lambda)}\right) C_{m_1 m_{\sigma_a(\lambda)}}\iota_{yx}(y+x)^{-m_1-m_{\sigma_a(\lambda)}-1}\\
        &\qquad \qquad \cdot  (-1)^{\sigma_a}(-1)^{k}T_{a_2, ..., a_r}^{m_2, ..., m_r}(\sigma_a(2), ..., \sigma_a(2k+1))\iota_{yx}(y+x)^{-m_{\sigma_a(2)}-\cdots - m_{\sigma_a(2k+1)}-k}\\
        & \cdot \sum_{l=0}^\infty \sum_{\sigma_b\in J_{2l}(1, ..., s)}(-1)^{\sigma_b}(-1)^l T_{b_1, ..., b_s}^{n_1, ..., n_s}(\sigma_b(1), ..., \sigma_b(2l))y^{-n_{\sigma_b(1)}-\cdots - n_{\sigma_b(2l)}-l}\\
        & \cdot \sum_{\rho=0}^{\min\{r-1-2k, s-2l\}}\sum_{\tau_a\in J_\rho(2k+2, ...,\widehat{\lambda},..., r)}\sum_{\tau_b\in J_\rho(2l+1, ..., s)}(-1)^{\tau_a}(-1)^{\tau_b} (-1)^{(r-1)\rho+\frac{\rho(\rho+1)}2}\\
        & \qquad \qquad\cdot  \begin{vmatrix}
            \langle (\sigma_a\circ \tau_a)(2k+2), \tilde{q}_{2l+1}\rangle_g & \cdots & \langle (\sigma_a\circ \tau_a)(2k+2), \tilde{q}_{2l+\rho}\rangle_g\\
            & & \\
            \langle (\sigma_a\circ \tau_a)(2k+1+\rho), \tilde{q}_{2l+1}\rangle_g & \cdots & \langle (\sigma_a\circ \tau_a)(2k+1+\rho), \tilde{q}_{2l+\rho}\rangle_g
        \end{vmatrix}\nonumber\\
        &\qquad \qquad \qquad \qquad \quad  \cdot S(x)_{a_1,...,a_r}^{m_1,...,m_r}((\sigma_a\circ\tau_a)(2k+2+\rho), ..., \reallywidehat{\sigma_a(\lambda)}, ..., (\sigma_a\circ\tau_a)(r))\\
        &\qquad \qquad \qquad \qquad \quad \cdot S_{b_1, ..., b_s}^{n_1, ..., n_s}(\tilde{q}_{2l+\rho+1}, ... \tilde{q}_s) 
    \end{align*}
    Conceptually, we would like to move the $(\sigma_a\circ\tau_a)(\lambda)$ from the third sequence to the first sequence. But since algebraic expression is slightly different from what is discussed in Lemma \ref{Comb-Id-Lemma}, we shall go through all the details. 
    \begin{itemize}
        \item For each $\lambda = 2k+2, ..., r$, we would like to modify the shuffle $\sigma_{a}$ by moving $\sigma_a(\lambda)$ from the second sequence to the first sequence. Let $\nu\in \{2, ..., 2k+2\}$ be the unique number such that $\sigma_a(\nu-1)<\sigma_a(\lambda)<\sigma_a(\nu)$. Let $\sigma_a^{new}$ be the modified shuffle. Then $\sigma_a^{new}$ is the following permutation

        {\small $$\left(\begin{array}{c c c c c c c c c c c c c}
            2 & \cdots & \nu-1 & \nu & \nu+1 & \cdots & 2k+2 & 2k+3 & \cdots & \lambda & \lambda+1 & \cdots & r\\
            \sigma_a(2) & \cdots & \sigma_a(\nu-1) & \sigma_a(\lambda) & \sigma_a(\nu) & \cdots & \sigma_a(2k+1) & \sigma_a(2k+2) & \cdots & \sigma_a(\lambda-1) & \sigma_a(\lambda+1) & \cdots & \sigma_a(r)
        \end{array}\right)$$}
        Clearly, 
        $$\sigma_a^{new} = \sigma_a \circ (\lambda, \lambda-1, ..., \nu+1, \nu) \Rightarrow (-1)^{\sigma_a^{new}}(-1)^{\nu-1}=(-1)^{\sigma_a}(-1)^{\lambda-1}$$
        and $\sigma_a^{new}$ is a 2-shuffle in $J_{2k+1}(2, ..., r)$. 
        \item We further fix $\tau_a\in J_\rho(2k+2, ..., \widehat{\lambda}, ..., r)$ and define $\tau_a^{new}$ by 
        $$\tau_a^{new}= \begin{pmatrix}
            2k+3 & \cdots & \lambda & \lambda+1 & \cdots & r\\
            \tau_a(2k+2) & \cdots &\tau_a(\lambda-1) & \tau_a(\lambda+1) & \cdots &\tau_a(r)
        \end{pmatrix} $$
        Essentially, $\tau_a^{new}$ is obtained from $\tau_a$ by renaming the indices. Thus $(-1)^{\tau_a^{new}} = (-1)^{\tau_a}$. The set of $\tau_a\in J_\rho(2k+2, ..., \widehat{\lambda}, ..., r)$ is in bijective correspondence with the set of $\tau_a^{new}\in J_{\rho}(2k+3, ..., r)$. Summing over $\tau_a\in J_\rho(2k+2, ..., \widehat{\lambda}, ..., r)$ amounts to summing over $\tau_a^{new}\in J_{\rho}(2k+3, ..., r)$
        \item Clearly, for each $i=1, ..., r-2k+2$, we have 
        $$(\sigma_a^{new}\circ\tau_a^{new})(2k+2+i) =\left\{\begin{array}{ll}(\sigma_a\circ\tau_a)(2k+2+i) & \text{ if }2k+2+i \leq \lambda - 1 \\
        (\sigma_a\circ\tau_a)(2k+2+i-1) & \text{ if }2k+2+i \geq \lambda  \\
        \end{array}\right.$$ 
        \item We rewrite the total contraction number $T_{a_2, ..., a_r}^{m_2, ..., m_r}(\sigma_a(2), ..., \sigma_a(2k+1))$ as \\
        $T_{a_1, ..., a_r}^{m_1, ..., m_r}(\widehat{1}, \sigma_a^{new}(2), ..., \widehat{\sigma_a^{new}(\nu)}, ..., \sigma_a^{new}(2k+2))$. We also rewrite other terms involving $\sigma_a$ and $\tau_a$ accordingly in terms of $\sigma_a^{new}$ and $\tau_a^{new}$. 
        \item Finally, since for each fixed $\lambda$ and $\sigma_a$, the number $\nu$ and the permutation $\sigma_a^{new}$ is uniquely determined, and vice versa, the set of $(\lambda, \sigma_a)$ is in bijective correspondence with the set of $(\nu, \sigma_a^{new})$. Summing over $\lambda \in \{2k+2, ..., r\}$ and $\sigma_a\in J_{2k}(1, ..., r)$ of the original summand amounts to summing over $\nu\in \{2, ..., 2k+2\}$ and $\sigma_a^{new}\in J_{2k+1}(2, ..., r)$ the modified expression. 
    \end{itemize}
    So the contribution of (\ref{exp-Delta-A-B-2-1}) to (\ref{exp-Delta-A-B-2}), after the reorganization above and the removal of the $new$-superscript, becomes
    \begin{align*}
        & \sum_{k=0}^\infty \sum_{\sigma_a\in J_{2k+1}(2, ..., r)}\sum_{\nu=2}^{2k+2} (-1)^{\nu -1}\left(a_1, a_{\sigma_a(\nu)}\right) C_{m_1 m_{\sigma_a(\nu)}}\iota_{yx}(y+x)^{-m_1-m_{\sigma_a(2)}-\cdots - m_{\sigma_a(2k+2)}-k}\\
        &\qquad \qquad \cdot  (-1)^{\sigma_a}(-1)^{k}T_{a_1, ..., a_r}^{m_1, ..., m_r}(\widehat{1},\sigma_a(2), ..., \widehat{\sigma_a(\nu)},...,\sigma_a(2k+2))\\
        & \cdot \sum_{l=0}^\infty \sum_{\sigma_b\in J_{2l}(1, ..., s)}(-1)^{\sigma_b}(-1)^l T_{b_1, ..., b_s}^{n_1, ..., n_s}(\sigma_b(1), ..., \sigma_b(2l))y^{-n_{\sigma_b(1)}-\cdots - n_{\sigma_b(2l)}-l}\\
        & \cdot \sum_{\rho=0}^{\min\{r-1-2k, s-2l\}}\sum_{\tau_a\in J_\rho(2k+3, ..., r)}\sum_{\tau_b\in J_\rho(2l+1, ..., s)}(-1)^{\tau_a}(-1)^{\tau_b} (-1)^{(r-1)\rho+\frac{\rho(\rho+1)}2}\\
        & \qquad \qquad\cdot  \begin{vmatrix}
            \langle (\sigma_a\circ \tau_a)(2k+3), \tilde{q}_{2l+1}\rangle_g & \cdots & \langle (\sigma_a\circ \tau_a)(2k+3), \tilde{q}_{2l+\rho}\rangle_g\\
            & & \\
            \langle (\sigma_a\circ \tau_a)(2k+2+\rho), \tilde{q}_{2l+1}\rangle_g & \cdots & \langle (\sigma_a\circ \tau_a)(2k+2+\rho), \tilde{q}_{2l+\rho}\rangle_g
        \end{vmatrix}\nonumber\\        &\qquad \qquad \qquad \qquad \quad  \cdot a_{(\sigma_a\circ\tau_a)(2k+3+\rho)}^{(m_{(\sigma_a\circ\tau_a)(2k+3+\rho)})}(x)_V^-\cdots   a_{(\sigma_a\circ\tau_a)(r)}^{(m_{\sigma_a(r)})}(x)_V^-\\
        &\qquad \qquad \qquad \qquad \quad \cdot b_{(\sigma_b\circ\tau_b)(2l+\rho+1)}(-n_{(\sigma_b\circ\tau_b)(2l+\rho+1)}-1/2) \cdots b_{(\sigma_b\circ\tau_b)(s)}(-n_{(\sigma_b\circ\tau_b)(s)}-1/2)\one 
    \end{align*}    
    Finally, we regard $\sigma_a\in J_{2k+1}(2, ..., r)$ as an element in $J_{2k+2}(1,...,r)$ with \colorbox{ffff00}{$\sigma_a(1)=1$}, reparametrize $k\mapsto k-1$, then use the definition of the total contraction number, to conclude that the contribution of (\ref{exp-Delta-A-B-2-1}) in (\ref{exp-Delta-A-B-2}) is 
    \begin{align}
        & \sum_{k= 1}^\infty \sum_{\substack{\sigma_a\in J_{2k}(1, ..., r)\\
        {\scriptsize \colorbox{ffff00}{$\sigma_a(1) = 1$}}}}  (-1)^{\sigma_a}(-1)^{k}T_{a_1, ..., a_r}^{m_1,...,m_r}(\sigma_a(1), ..., \sigma_a(2k))\iota_{yx}(y+x)^{-m_{\sigma_a(1)}-\cdots - m_{\sigma_a(2k)} - k} \nonumber \\
        &  \cdot \sum_{l= 0}^\infty \sum_{\sigma_b\in J_{2l}(1, .., s)}(-1)^{\sigma_b}(-1)^{l}T_{b_1, ..., b_s}^{n_1, ..., n_s}(\sigma_b(1), ..., \sigma_b(2l)) y^{-n_{\sigma_b(1)}-\cdots - n_{\sigma_b(2l)} - l}\nonumber \\
        &  \cdot \sum_{\rho=0}^{\min\{r-2k, s-2l\}}\sum_{\substack{\tau_a\in J_\rho(2k+1,..., r)}}
        \sum_{\substack{\tau_b\in J_{\rho}(2l+1, ..., s)}}(-1)^{\tau_a}(-1)^{\tau_b}  (-1)^{r\rho + \frac{\rho(\rho+1)}{2}}\nonumber  \\
        &   \qquad \qquad \cdot \begin{vmatrix}
            \langle (\sigma_a\circ \tau_a)(2k+1), (\sigma_b\circ\tau_b)(2l+1)\rangle_g & \cdots & \langle (\sigma_a\circ \tau_a)(2k+1), (\sigma_b\circ\tau_b)(2l+\rho)\rangle_g\\
            \vdots & & \vdots \\
            \langle (\sigma_a\circ \tau_a)(2k+\rho), (\sigma_b\circ\tau_b)(2l+1)\rangle_g & \cdots & \langle (\sigma_a\circ \tau_a)(2k+\rho), (\sigma_b\circ\tau_b)(2l+\rho)\rangle_g
        \end{vmatrix}\nonumber\\        
        & \qquad\qquad \qquad \cdot a_{(\sigma_a\circ\tau_a)(2k+\rho+1)}^{(m_{(\sigma_a\circ\tau_a)(2k+\rho+1)})}(x)_V^-\cdots a_{(\sigma_a\circ\tau_a)(r)}^{(m_{(\sigma_a\circ\tau_a)(r)})}(x)_V^- \nonumber \\
        & \qquad \qquad \qquad \cdot b_{(\sigma_b\circ\tau_b)(2l+\rho+1)}(-n_{(\sigma_b\circ\tau_b)(2l+\rho+1)}-1/2)\cdots  b_{(\sigma_b\circ\tau_b)(s)}(-n_{(\sigma_b\circ\tau_b)(s)}-1/2)\one\label{exp-Delta-A-B-F2}
    \end{align}
    where the summand corresponding to $k=0$ is vacuous. 
  
    \noindent\textbf{$\blacktriangleright$ Contribution of the Application of (\ref{exp-Delta-A-B-2}) by (\ref{exp-Delta-A-B-2-2})} To analyze this term, since we will have to operate on $\tau_b$, but not $\sigma_a$ and $\tau_a$, we will use the abbreviation $\tilde{p}_\alpha = (\sigma_a\circ\tau_a)(\alpha)$ for $\alpha = 2k+2, ..., r-1$ and will not use the previous abbreviation $\tilde{q}$. Supplementing the other lines, we see that the part of the application of (\ref{exp-Delta-A-B-2}) contributed by (\ref{exp-Delta-A-B-2-2}) is 
    \begin{align*}
        &  \sum_{k=0}^\infty \sum_{\sigma_a\in J_{2k}(2, ..., r)}(-1)^{\sigma_a}(-1)^{k}T_{a_1, a_2, ..., a_r}^{m_1, m_2, ..., m_r}(\sigma_a(2), ..., \sigma_a(2k+1))\iota_{yx}(y+x)^{-m_{\sigma_a(2)}-\cdots - m_{\sigma_a(2k+1)}-k}\\
        & \cdot \sum_{l=0}^\infty \sum_{\sigma_b\in J_{2l}(1, ..., s)}(-1)^{\sigma_b}(-1)^l T_{b_1, ..., b_s}^{n_1, ..., n_s}(\sigma_b(1), ..., \sigma_b(2l))y^{-n_{\sigma_b(1)}-\cdots - n_{\sigma_b(2l)}-l}\\
        & \cdot \sum_{\rho=0}^{\min\{r-1-2k, s-2l\}}\sum_{\tau_a\in J_\rho(2k+2, ..., r)}\sum_{\tau_b\in J_\rho(2l+1, ..., s)}(-1)^{\tau_a}(-1)^{\tau_b} (-1)^{(r-1)\rho+\frac{\rho(\rho+1)}2}\\
        & \qquad \qquad\cdot  \begin{vmatrix}
            \langle \tilde{p}_{2k+2}, (\sigma_b\circ\tau_b)(2l+1)\rangle_g & \cdots & \langle \tilde{p}_{2k+2}, (\sigma_b\circ\tau_b)(2l+\rho)\rangle_g\\
            \vdots & & \vdots \\
            \langle \tilde{p}_{2k+1+\rho}, (\sigma_b\circ\tau_b)(2l+1)\rangle_g & \cdots & \langle \tilde{p}_{2k+1+\rho}, (\sigma_b\circ\tau_b)(2l+\rho)\rangle_g
        \end{vmatrix}\nonumber\\
        & \qquad \qquad \cdot \sum_{\lambda=2l+1+\rho}^s (-1)^{\lambda+r} \left(\iota_{yx}g_{m_1,n_{(\sigma_b\circ\tau_b)(\lambda)}}(y+x, y) -(x^{-n_{(\sigma_b\circ\tau_b)(\lambda)}-1})^{(m_1)}\right)\left(a_{1}, b_{(\sigma_a\circ \tau_a)(\lambda)}\right)\\ 
        &\qquad \qquad  \cdot S(x)_{a_1, ..., a_{r}}^{m_1, ..., m_r}(\tilde{p}_{2k+\rho+2}, ...,\tilde{p}_r)\nonumber\\
        &\qquad \qquad \cdot S_{b_1, ..., b_s}^{n_1, ..., n_s}((\sigma_b\circ\tau_b)(2l+\rho+1), ..., \reallywidehat{(\sigma_b\circ\tau_b)(\lambda)}, ..., (\sigma_b\circ\tau_b)(s))
    \end{align*}
    Apply the equations (\ref{Comb-Id-4}) and (\ref{Comb-Id-5}) in Remark \ref{2-shuffle-summation} to the summation regarding $\tau_b\in J_\rho(2l+1, ..., s)$. More precisely, we will perform the following changes
    \begin{align*}
        \sum_{\tau_b\in J_\rho(2l+1, ..., s)}\sum_{\lambda=2l+\rho+1}^s \mapsto \sum_{\lambda'=2l+1}^s \sum_{\tau_b\in J_{\rho}(2l+1, ..., \widehat{\lambda'}, ..., s)} \mapsto \sum_{\tau_b\in J_{\rho+1}(2l+1, ..., s)}\sum_{\nu=2l+1}^{2l+\rho}.
    \end{align*}
    We obtain
    \begin{align*}
        & \cdot \sum_{k=0}^\infty \sum_{\sigma_a\in J_{2k}(2, ..., r)}(-1)^{\sigma_a}(-1)^{k}T_{a_1, a_2, ..., a_r}^{m_1, m_2, ..., m_r}(\sigma_a(2), ..., \sigma_a(2k+1))\iota_{yx}(y+x)^{-m_{\sigma_a(2)}-\cdots - m_{\sigma_a(2k+1)}-k}\\
        & \cdot \sum_{l=0}^\infty \sum_{\sigma_b\in J_{2l}(1, ..., s)}(-1)^{\sigma_b}(-1)^l T_{b_1, ..., b_s}^{n_1, ..., n_s}(\sigma_b(1), ..., \sigma_b(2l))y^{-n_{\sigma_b(1)}-\cdots - n_{\sigma_b(2l)}-l}\\
        & \cdot \sum_{\rho=0}^{\min\{r-1-2k, s-2l\}}\sum_{\tau_a\in J_\rho(2k+2, ..., r)}\sum_{\tau_b\in J_{\rho+1}(2l+1, ..., s)}(-1)^{\tau_a}(-1)^{\tau_b} (-1)^{(r-1)\rho+\frac{\rho(\rho+1)}2}\\
        & \cdot \sum_{\nu = 2l+1}^{2l+\rho} (-1)^{\nu+r}\left(\langle 1, b_{(\sigma_b\circ\tau_b)(\nu)}\rangle_g -\langle 1, b_{(\sigma_b\circ\tau_b)(\nu)}\rangle_x \right)\\
        & \qquad \qquad\cdot  \begin{vmatrix}
            \langle \tilde{p}_{2k+2}, (\sigma_b\circ\tau_b)(2l+1)\rangle_g & \cdots & \hbox{\sout{$\langle \tilde{p}_{2k+2}, (\sigma_b\circ\tau_b)(\nu)\rangle$}}& \cdots & \langle \tilde{p}_{2k+2}, (\sigma_b\circ\tau_b)(2l+\rho)\rangle_g\\
            \vdots & & \vdots & & \vdots\\
            \langle \tilde{p}_{2k+1+\rho}, (\sigma_b\circ\tau_b)(2l+1)\rangle_g & \cdots & \hbox{\sout{$\langle \tilde{p}_{2k+1+\rho}, (\sigma_b\circ\tau_b)(\nu)\rangle$}}& \cdots & \langle \tilde{p}_{2k+1+\rho}, (\sigma_b\circ\tau_b)(2l+\rho)\rangle_g
        \end{vmatrix}\nonumber\\
        &\qquad \qquad  \cdot S(x)_{a_1, ..., a_{r}}^{m_1, ..., m_r}(\tilde{p}_{2k+\rho+2}, ...,\tilde{p}_r)\nonumber\\
        &\qquad \qquad \cdot S_{b_1, ..., b_s}^{n_1, ..., n_s}((\sigma_b\circ\tau_b)(2l+\rho+2), ..., (\sigma_b\circ\tau_b)(s))
    \end{align*}
    Finally, we get rid of the abbreviation $\tilde{p}$, rearrange the factors $(-1)^{(r-1)\rho + \frac{\rho(\rho+1)}2}$ and $(-1)^{\nu +r}$ as $(-1)^{1+\nu} (-1)^{r(\rho+1) + \frac{(\rho+1)(\rho+2)}2}$, then combine the factor $(-1)^{\nu+1}$ with the determinant and the term $\left(\langle 1, b_{(\sigma_b\circ\tau_b)(\nu)}\rangle_g -\langle 1, b_{(\sigma_b\circ\tau_b)(\nu)}\rangle_x \right)$. The summation with respect to $\nu$ is precisely the expansion of the determinant appearing below along the first row: 
    \begin{align}
        &  \sum_{k=0}^\infty \sum_{\sigma_a\in J_{2k}(2, ..., r)}(-1)^{\sigma_a}(-1)^{k}T_{a_1, a_2, ..., a_r}^{m_1, m_2, ..., m_r}(\sigma_a(2), ..., \sigma_a(2k+1))\iota_{yx}(y+x)^{-m_{\sigma_a(2)}-\cdots - m_{\sigma_a(2k+1)}-k}\nonumber\\
        & \cdot \sum_{l=0}^\infty \sum_{\sigma_b\in J_{2l}(1, ..., s)}(-1)^{\sigma_b}(-1)^l T_{b_1, ..., b_s}^{n_1, ..., n_s}(\sigma_b(1), ..., \sigma_b(2l))y^{-n_{\sigma_b(1)}-\cdots - n_{\sigma_b(2l)}-l}\nonumber\\
        & \cdot \sum_{\rho=0}^{\min\{r-1-2k, s-2l\}}\sum_{\tau_a\in J_\rho(2k+2, ..., r)}\sum_{\tau_b\in J_{\rho+1}(2l+1, ..., s)}(-1)^{\tau_a}(-1)^{\tau_b} (-1)^{r(\rho+1) + \frac{(\rho+1)(\rho+2)}2}\nonumber\\
        & \qquad \qquad\cdot  \begin{vmatrix}
            \langle 1, b_{(\sigma_b\circ\tau_b)(2l+1)}\rangle_g -\langle 1, b_{(\sigma_b\circ\tau_b)(2l+1)}\rangle_x & \cdots & \langle 1, b_{(\sigma_b\circ\tau_b)(2l+\rho)}\rangle_g -\langle 1, b_{(\sigma_b\circ\tau_b)(2l+\rho)}\rangle_x\\
            \langle (\sigma_a\circ\tau_a)(2k+2), (\sigma_b\circ\tau_b)(2l+1)\rangle_g & \cdots & \langle (\sigma_a\circ\tau_a)(2k+2), (\sigma_b\circ\tau_b)(2l+\rho)\rangle_g\\
            \vdots &  & \vdots\\
            \langle (\sigma_a\circ\tau_a)(2k+1+\rho), (\sigma_b\circ\tau_b)(2l+1)\rangle_g & \cdots & \langle (\sigma_a\circ\tau_a)(2k+1+\rho), (\sigma_b\circ\tau_b)(2l+\rho)\rangle_g
        \end{vmatrix}\nonumber\nonumber\\
        &\qquad \qquad  \cdot S(x)_{a_1, ..., a_{r}}^{m_1, ..., m_r}((\sigma_a\circ\tau_a)(2k+\rho+2), ...,(\sigma_a\circ\tau_a)(r))\nonumber\\
        &\qquad \qquad \cdot S_{b_1, ..., b_s}^{n_1, ..., n_s}((\sigma_b\circ\tau_b)(2l+\rho+2), ..., (\sigma_b\circ\tau_b)(s)) \label{exp-Delta-A-B-2-2f}
    \end{align}   
    We shall turn to (\ref{exp-Delta-A-B-3}) now and combine what we shall get with (\ref{exp-Delta-A-B-2-2f}). 

    \noindent\textbf{Application of (\ref{exp-Delta-A-B-3})}. Clearly, 
    \begin{align*}
        & a_1^{(m_1)}(x)^+ b_1(-n_1-1/2)\cdots b_s(-n_s-1/2) \\
        = \ & \sum_{\lambda=1}^s (-1)^{\lambda-1} (a_1, b_\lambda) (x^{-n_\lambda-1})^{(m_1)} b_1(-n_1-1/2) \cdots \reallywidehat{b_\lambda(-n_\lambda-1/2)}  \cdots b_s(-n_s-1/2)\one.
    \end{align*}
    Therefore, the application of (\ref{exp-Delta-A-B-3}) yields
    \begin{align*}
        & (-1)^{r-1} \sum_{\lambda=1}^{s}(-1)^{\lambda-1}\langle 1, \lambda\rangle_x \\
        & \cdot \exp(\Delta(y)):a_2^{(m_2)}(x)\cdots a_r^{(m_r)}(x): b_1(-n_1-1/2)\cdots \reallywidehat{b_\lambda(-n_\lambda-1/2)}  \cdots b_s(-n_s-1/2)\one \\
        = \ & (-1)^{r-1} \sum_{\lambda=1}^{s}(-1)^{\lambda-1}\langle 1, \lambda\rangle_x\\
        & \cdot \sum_{k=0}^\infty \sum_{\sigma_a\in J_{2k}(2, ..., r)}T_{a_1,..., a_r}^{m_1, ..., m_r}(\sigma_a(2), ..., \sigma_a(2k+1)) \iota_{yx}(y+x)^{-m_{\sigma_a(2)} - \cdots - m_{\sigma_a(r)} - k}\\
        & \cdot \sum_{l=0}^\infty \sum_{\sigma_b\in J_{2l}(1, ..., \widehat{\lambda}, ..., s)}(-1)^{\sigma_b} (-1)^l T_{b_1, ..., b_s}^{n_1, ..., n_s}(\tilde{q}_1, ..., \tilde{q}_{2l}) y^{-n_{\tilde{q}_1}-\cdots - n_{\tilde{q}_{2l}}-l}\\
        & \cdot \sum_{\rho=0}^{\min\{r-2k, s-2l\}}\sum_{\tau_a\in J_\rho(2k+2, ..., r)}(-1)^{\tau_a} \sum_{\tau_b\in J_\rho(2l+1, ..., \widehat{\lambda},..., s)}(-1)^{\tau_b}(-1)^{(r-1)\rho +\frac{\rho(\rho+1)}{2}}\\
        & \qquad \cdot \begin{vmatrix}
            \langle (\sigma_a\circ \tau_a)(2k+2), \tilde{q}_{2l+1}\rangle_g & \cdots &  \langle (\sigma_a\circ \tau_a)(2k+2), \tilde{q}_{2l+\rho})\rangle_g \\
            \vdots & & \vdots \\
            \langle (\sigma_a\circ \tau_a)(2k+\rho+1), \tilde{q}_{2l+1}\rangle_g & \cdots  &  \langle (\sigma_a\circ \tau_a)(2k+\rho+1), \tilde{q}_{2l+\rho})\rangle_g 
        \end{vmatrix}\\
        & \qquad \cdot S(x)_{a_1, ..., a_r}^{m_1, ..., m_r}((\sigma_a\circ\tau_a)(2k+\rho+2), ..., (\sigma_a\circ\tau_a)(r))\\
        & \qquad \cdot S_{b_1, ..., b_s}^{n_1, ..., n_s}(\tilde{q}_1, ..., \tilde{q}_{s-1}).
    \end{align*}
    Here the sequence $(\tilde{q}_1, ...., \tilde{q}_{s-1})$ depends on the choice of $\lambda$ (we would refrain from adding a superscript for brevity) and is defined as follows. 
    \begin{itemize}
        \item For fixed $\sigma_b\in J_{2l}(1, ..., \widehat{\lambda}, ..., s)$, $\tilde{q}_1, ..., \tilde{q}_{2l}$ are the first $2l$ numbers in the images of $\sigma_b$. In case $\lambda > 2l$, we have 
    $$\tilde{q}_1 = \sigma_b(1), ..., p(2l)=  \sigma_b(2l); $$
    In case $\lambda \leq 2l$, 
    $$\tilde{q}_1 = \sigma_b(1), ..., \tilde{q}_{\lambda-1} = \sigma_b(\lambda-1), \tilde{q}_{\lambda} = \sigma_b(\lambda+1), ..., \tilde{q}_{2l}=\sigma_b(2l+1).$$
        \item For fixed $\sigma_b\in J_{2l}(1, ..., \widehat{\lambda}, , ..., s)$ and $\tau_b\in J_{\rho}(2l+1, ..., \widehat{\lambda}, ..., s)$, 
        $\tilde{q}_{2l+1}, ..., \tilde{q}_{2l+\rho}$ are the first $\rho$ numbers in the image of $\sigma_b\circ \tau_b$. In case $\lambda > 2l+\rho$, we have 
            $$\tilde{q}_{2l+1} = (\sigma_b\circ\tau_b)(2l+1), ..., \tilde{q}_{2l+\rho}= (\sigma_b\circ\tau_b)(2l+\rho); $$
        in case $2l < \lambda \leq 2l+\rho$, we have
        \begin{align*}
            & \tilde{q}_{2l+1} = (\sigma_b\circ\tau_b)(2l+1), ..., \tilde{q}_{\lambda-1} = (\sigma_b\circ\tau_b)(\lambda-1), \\
            & \tilde{q}_\lambda = (\sigma_b\circ\tau_b)(\lambda+1), ..., \tilde{q}_{2l+\rho}= (\sigma_b\circ\tau_b)(2l+\rho+1); 
        \end{align*}
        in case $1\leq \lambda \leq 2l$, we have
        $$\tilde{q}_{2l+1} = (\sigma_b\circ\tau_b)(2l+2), ..., \tilde{q}_{2l+\rho}= (\sigma_b\circ\tau_b)(2l+\rho+1); $$        
        \item $\tilde{q}_{2l+\rho+1}, ..., \tilde{q}_{s-1}$ forms the complementary sequence of $\tilde{q}_1, ..., \tilde{q}_{2l+\rho}$ in $\{1, ...,\widehat{\lambda}, ..., s\}$ 
    \end{itemize}
    We would perform the following operations: 
    \begin{itemize}
        \item Introduce the notations $\tilde{p}_{2k+2}, .... \tilde{p}_r$ as we did for the contribution by (\ref{exp-Delta-A-B-2-2}). 
        \item Incorporate $(-1)^{r-1}$ and $(-1)^{(r-1)\rho + \frac{\rho(\rho+1)}2}$, resulting in $(-1)^{r(\rho+1) + \frac{(\rho+1)(\rho+2)}2}$.
        \item Viewing the sequence $\tilde{q}_1 < \cdots < \tilde{q}_{2l}, \tilde{q}_{2l+1}< \cdots < \tilde{q}_{2l+\rho}, \tilde{q}_{2l+\rho+1} < \cdots < \tilde{q}_{s-1}$ as a 3-shuffle, we use Formula (\ref{Comb-Id-2}) to interchange the summation of $\lambda$ and $\sigma_b$. After the interchange, we have a 3-shuffle in $J_{2l, \rho+1}(1, ..., s)$, which we shall denote by $\sigma_b(1), ..., \sigma_b(2l), (\sigma_b\circ\tau_b)(2l+1), ..., (\sigma_b\circ\tau_b)(s)$. The missing term now sits between the $(2l+1)$-th position and the $(2l+\rho)$-th position. 
    \end{itemize}
    After performing these operations, we obtain
    \begin{align*}
        \ & \sum_{k=0}^\infty \sum_{\sigma_a\in J_{2k}(2, ..., r)}T_{a_1,..., a_r}^{m_1, ..., m_r}(\sigma_a(2), ..., \sigma_a(2k+1)) \iota_{yx}(y+x)^{-m_{\sigma_a(2)} - \cdots - m_{\sigma_a(r)} - k}\\
        & \cdot \sum_{l=0}^\infty \sum_{\sigma_b\in J_{2l}(1, ..., s)}(-1)^{\sigma_b} (-1)^l T_{b_1, ..., b_s}^{n_1, ..., n_s}(\sigma_b(1), ..., \sigma_b(2l)) y^{-n_{\sigma_b(1)}-\cdots - n_{\sigma_b(2l)}-l}\\
        & \cdot \sum_{\rho=0}^{\min\{r-2k, s-2l\}}\sum_{\tau_a\in J_\rho(2k+2, ..., r)}(-1)^{\tau_a} \sum_{\tau_b\in J_{\rho+1}(2l+1, ..., s)}(-1)^{\tau_b}(-1)^{r(\rho+1) + \frac{(\rho+1)(\rho+2)}2}\\
        & \cdot \sum_{\lambda =2l+1}^{2l+\rho+1}(-1)^{\lambda-1}\langle 1, (\sigma_b\circ\tau_b)(\lambda)\rangle_x\\
        & \cdot \begin{vmatrix}
            \langle \tilde{p}_{2k+2}, (\sigma_b\circ\tau_b)(2l+1)\rangle_g & \cdots &  \hbox{\sout{$\langle  \tilde{p}_{2k+2}, (\sigma_b\circ\tau_b)(\lambda) \rangle_g$}} & \cdots &  \langle  \tilde{p}_{2k+2},(\sigma_b\circ\tau_b)(2l+\rho+1))\rangle_g \\
            \vdots & & \vdots & & \vdots \\
            \langle \tilde{p}_{2k+\rho+1}, (\sigma_b\circ\tau_b)(2l+1)\rangle_g & \cdots &  \hbox{\sout{$\langle\tilde{p}_{2k+\rho+1},, (\sigma_b\circ\tau_b)(\lambda) \rangle_g$}} & \cdots &  \langle \tilde{p}_{2k+\rho+1},,(\sigma_b\circ\tau_b)(2l+\rho+1))\rangle_g \\
        \end{vmatrix}\\
        & \qquad \cdot S(x)_{a_1, ..., a_r}^{m_1, ..., m_r}((\sigma_a\circ\tau_a)(2k+\rho+2), ..., (\sigma_a\circ\tau_a)(r))\\
        & \qquad \cdot S_{b_1, ..., b_s}^{n_1, ..., n_s}((\sigma_b\circ\tau_b)(2l+\rho+2),..., (\sigma_b\circ\tau_b)(s)).
    \end{align*}
    The summation with respect to $\lambda$ is precisely the expansion of the determinant appearing below along the first row: 
    \begin{align}
        \ & \sum_{k=0}^\infty \sum_{\sigma_a\in J_{2k}(2, ..., r)}T_{a_1,..., a_r}^{m_1, ..., m_r}(\sigma_a(2), ..., \sigma_a(2k+1)) \iota_{yx}(y+x)^{-m_{\sigma_a(2)} - \cdots - m_{\sigma_a(r)} - k}\nonumber\\
        & \cdot \sum_{l=0}^\infty \sum_{\sigma_b\in J_{2l}(1, ..., s)}(-1)^{\sigma_b} (-1)^l T_{b_1, ..., b_s}^{n_1, ..., n_s}(\sigma_b(1), ..., \sigma_b(2l)) y^{-n_{\sigma_b(1)}-\cdots - n_{\sigma_b(2l)}-l}\nonumber\\
        & \cdot \sum_{\rho=0}^{\min\{r-2k, s-2l\}}\sum_{\tau_a\in J_\rho(2k+2, ..., r)}(-1)^{\tau_a} \sum_{\tau_b\in J_{\rho+1}(2l+1, ..., s)}(-1)^{\tau_b}(-1)^{r(\rho+1) + \frac{(\rho+1)(\rho+2)}2}\nonumber\\
        & \qquad \qquad \qquad\cdot \begin{vmatrix}
            \langle 1, (\sigma_b\circ\tau_b)(2l+1)\rangle_x & \cdots &  \langle  1,(\sigma_b\circ\tau_b)(2l+\rho+1))\rangle_x \\
            \langle \tilde{p}_{2k+2}, (\sigma_b\circ\tau_b)(2l+1)\rangle_g & \cdots &  \langle  \tilde{p}_{2k+2},(\sigma_b\circ\tau_b)(2l+\rho+1))\rangle_g \\
            \vdots & & \vdots \\
            \langle \tilde{p}_{2k+\rho+1}, (\sigma_b\circ\tau_b)(2l+1)\rangle_g & \cdots &  \langle \tilde{p}_{2k+\rho+1},,(\sigma_b\circ\tau_b)(2l+\rho+1))\rangle_g \\
        \end{vmatrix}\nonumber\\
        & \qquad \qquad \qquad \cdot S(x)_{a_1, ..., a_r}^{m_1, ..., m_r}((\sigma_a\circ\tau_a)(2k+\rho+2), ..., (\sigma_a\circ\tau_a)(r))\nonumber\\
        & \qquad \qquad \qquad \cdot S_{b_1, ..., b_s}^{n_1, ..., n_s}((\sigma_b\circ\tau_b)(2l+\rho+2),..., (\sigma_b\circ\tau_b)(s)).\label{exp-Delta-A-B-3f}
    \end{align}
    Now we may combine (\ref{exp-Delta-A-B-2-2f}) and (\ref{exp-Delta-A-B-3f}), to obtain 
    \begin{align}
        \ & \sum_{k=0}^\infty \sum_{\sigma_a\in J_{2k}(2, ..., r)}T_{a_1,..., a_r}^{m_1, ..., m_r}(\sigma_a(2), ..., \sigma_a(2k+1)) \iota_{yx}(y+x)^{-m_{\sigma_a(2)} - \cdots - m_{\sigma_a(r)} - k}\nonumber\\
        & \cdot \sum_{l=0}^\infty \sum_{\sigma_b\in J_{2l}(1, ..., s)}(-1)^{\sigma_b} (-1)^l T_{b_1, ..., b_s}^{n_1, ..., n_s}(\sigma_b(1), ..., \sigma_b(2l)) y^{-n_{\sigma_b(1)}-\cdots - n_{\sigma_b(2l)}-l}\nonumber\\
        & \cdot \sum_{\rho=0}^{\min\{r-2k, s-2l\}}\sum_{\tau_a\in J_\rho(2k+2, ..., r)}(-1)^{\tau_a} \sum_{\tau_b\in J_{\rho+1}(2l+1, ..., s)}(-1)^{\tau_b}(-1)^{r(\rho+1) + \frac{(\rho+1)(\rho+2)}2}\nonumber\\
        & \qquad \qquad \qquad\cdot \begin{vmatrix}
            \langle 1, (\sigma_b\circ\tau_b)(2l+1)\rangle_g & \cdots &  \langle  1,(\sigma_b\circ\tau_b)(2l+\rho+1))\rangle_g \\
            \langle \tilde{p}_{2k+2}, (\sigma_b\circ\tau_b)(2l+1)\rangle_g & \cdots &  \langle  \tilde{p}_{2k+2},(\sigma_b\circ\tau_b)(2l+\rho+1))\rangle_g \\
            \vdots & & \vdots \\
            \langle \tilde{p}_{2k+\rho+1}, (\sigma_b\circ\tau_b)(2l+1)\rangle_g & \cdots &  \langle \tilde{p}_{2k+\rho+1},,(\sigma_b\circ\tau_b)(2l+\rho+1))\rangle_g \\
        \end{vmatrix}\nonumber\\
        & \qquad \qquad \qquad \cdot S(x)_{a_1, ..., a_r}^{m_1, ..., m_r}((\sigma_a\circ\tau_a)(2k+\rho+2), ..., (\sigma_a\circ\tau_a)(r))\nonumber\\
        & \qquad \qquad \qquad \cdot S_{b_1, ..., b_s}^{n_1, ..., n_s}((\sigma_b\circ\tau_b)(2l+\rho+2),..., (\sigma_b\circ\tau_b)(s)).\nonumber
    \end{align}
    By a very similar procedure as we did for (\ref{exp-Delta-A-B-1}), we may add $1$ into the second sequence of the 3-shuffle 
    $$2\leq \sigma_a(1)< ...< \sigma_a(2k)\leq r, 2\leq \tilde{p}_{2k+1}, ..., \tilde{p}_{2k+\rho}\leq r, 2\leq \tilde{p}_{2k+\rho+1}, ..., \tilde{p}_{r}\leq r. $$
    resulting in a new 3-shuffle with \colorbox{ffff00}{$\sigma_a(1) > 1$} and \colorbox{ffff00}{$\tau_a(2k+1) = 2k+1$}. Details shall not be repeated. We also shift $\rho\mapsto \rho-1$. At the end of the day, $(\ref{exp-Delta-A-B-2-2f}) + (\ref{exp-Delta-A-B-3f})$ is equal to 
    \begin{align}
        & \sum_{k= 1}^\infty \sum_{\substack{\sigma_a\in J_{2k}(1, ..., r)\\
        {\scriptsize \colorbox{ffff00}{$\sigma_a(1) > 1$}}}}  (-1)^{\sigma_a}(-1)^{k}T_{a_1, ..., a_r}^{m_1,...,m_r}(\sigma_a(1), ..., \sigma_a(2k))\iota_{yx}(y+x)^{-m_{\sigma_a(1)}-\cdots - m_{\sigma_a(2k)} - k} \nonumber \\
        &  \cdot \sum_{l= 0}^\infty \sum_{\sigma_b\in J_{2l}(1, .., s)}(-1)^{\sigma_b}(-1)^{l}T_{b_1, ..., b_s}^{n_1, ..., n_s}(\sigma_b(1), ..., \sigma_b(2l)) y^{-n_{\sigma_b(1)}-\cdots - n_{\sigma_b(2l)} - l}\nonumber \\
        &  \cdot \sum_{\rho=0}^{\min\{r-2k, s-2l\}}\sum_{\substack{\tau_a\in J_\rho(2k+1,..., r)\\
        {\scriptsize \colorbox{ffff00}{$\tau_a(2k+1) = 2k+1$}}}}
        \sum_{\substack{\tau_b\in J_{\rho}(2l+1, ..., s)}}(-1)^{\tau_a}(-1)^{\tau_b}  (-1)^{r\rho + \frac{\rho(\rho+1)}{2}}\nonumber  \\
        &  \qquad \qquad \qquad \cdot \det\left((a_{(\sigma_a\circ\tau_a)(2k+i)}, b_{(\sigma_b\circ\tau_b)(2l+j)}))\iota_{yx}g_{m_{(\sigma_a\circ\tau_a)(2k+i)}n_{(\sigma_b\circ\tau_b)(2l+j)}}(y+x, y)\right)_{i, j = 1}^\rho\nonumber \\
        &   \qquad \qquad \cdot \begin{vmatrix}
            \langle (\sigma_a\circ \tau_a)(2k+1), (\sigma_b\circ\tau_b)(2l+1)\rangle_g & \cdots & \langle (\sigma_a\circ \tau_a)(2k+1), (\sigma_b\circ\tau_b)(2l+\rho)\rangle_g\\
            \vdots & & \vdots \\
            \langle (\sigma_a\circ \tau_a)(2k+\rho), (\sigma_b\circ\tau_b)(2l+1)\rangle_g & \cdots & \langle (\sigma_a\circ \tau_a)(2k+\rho), (\sigma_b\circ\tau_b)(2l+\rho)\rangle_g
        \end{vmatrix}\nonumber\\        
        & \qquad \qquad \qquad \cdot b_{(\sigma_b\circ\tau_b)(2l+\rho+1)}(-n_{(\sigma_b\circ\tau_b)(2l+\rho+1)}-1/2)\cdots  b_{(\sigma_b\circ\tau_b)(s)}(-n_{(\sigma_b\circ\tau_b)(s)}-1/2)\one \label{exp-Delta-A-B-F3}
    \end{align}
    where the summand with $\rho=0$ is vacuous. 

    To conclude the proof, we see that (\ref{exp-Delta-A-B-F1}), (\ref{exp-Delta-A-B-F2}), (\ref{exp-Delta-A-B-F3}) are simply parts in (\ref{exp-Delta-A-B-shuffle}). Clearly, the latter consists of the sum of the first three. Thus we managed to prove (\ref{exp-Delta-A-B-shuffle}). The conclusion of the theorem follows.
\end{proof}

\begin{cor}
    For $r, s\in \N, a_1, ..., a_r, b_1, ..., b_s\in \h, m_1, ...,m_r, n_1, ..., n_s\in \N$, 
        \begin{align*}
        & Y_W(Y_V(a_1(-m_1-1/2)\cdots a_r(-m_r-1/2)\one, x)b_1(-n_1-1/2)\cdots b_s(-n_s-1/2)\one, y) \\
        = \ & \sum_{k= 0}^\infty \sum_{\substack{1\leq p_1 < \cdots < p_{2k} \leq r}}  (-1)^{p_1+\cdots + p_{2k}}T_{a_1, ..., a_r}^{m_1,...,m_r}(p_1, ..., p_{2k})\iota_{yx}(y+x)^{-m_{p_1}-\cdots - m_{p_{2k}} - k} \nonumber \\
        & \quad \cdot \sum_{l= 0}^\infty \sum_{\substack{
        1\leq q_1 < \cdots < q_{2l} \leq s}}(-1)^{q_1+\cdots + q_{2l}}T_{b_1, ..., b_s}^{n_1, ..., n_s}(q_1, ..., q_{2l}) y^{-n_{q_1}-\cdots - n_{q_{2l}} - l}\nonumber \\
        & \qquad \cdot \sum_{\rho=0}^{\min\{r-2k, s-2l\}}\sum_{\substack{1\leq i_1 < \cdots < i_\rho\leq r-2k \\
        1\leq j_1 < \cdots < j_\rho \leq s-2l}}(-1)^{i_1+\cdots + i_{\rho} + j_1 + \cdots + j_\rho} (-1)^{(r-2k)\rho + \frac{\rho(\rho+1)}{2}}\nonumber \\
        & \qquad\qquad \qquad \cdot \begin{vmatrix}
            (a_{p_{i_1}^c}, b_{q_{j_1}^c})\iota_{yx}g_{m_{p_{i_1}^c}n_{q_{j_1}^c}}(x+y, y) & \cdots & (a_{p_{i_1}^c}, b_{q_{j_\rho}^c})\iota_{yx}g_{m_{p_{i_1}^c}n_{q_{j_\rho}^c}}(x+y, y)\\
            \vdots & & \vdots \\
            (a_{p_{i_\rho}^c}, b_{q_{j_1}^c})\iota_{yx}g_{m_{p_{i_\rho}^c}n_{q_{j_1}^c}}(x+y, y) & \cdots & (a_{p_{i_\rho}^c}, b_{q_{j_\rho}^c})\iota_{yx}g_{m_{p_{i_\rho}^c}n_{q_{j_\rho}^c}}(x+y, y)
        \end{vmatrix}\nonumber \\
        & \qquad\qquad \qquad \cdot \nord a_{p_1^c}^{(m_{p_1^c})}(y+x)\cdots \reallywidehat{a_{p_{i_1}^c}^{(m_{p_{i_1}^c})}(y+x)} \cdots \reallywidehat{a_{p_{i_\rho}^c}^{(m_{p_{i_\rho}^c})}(y+x)} \cdots a_{p_{r-2k}^c}(y+x) \nonumber \\
        & \qquad \qquad \qquad \quad \cdot b_{q_1^c}(-n_{q_1^c}-1/2)\cdots \reallywidehat{b_{q_{j_1}^c}(-n_{q_{j_1}^c}-1/2)}\cdots\reallywidehat{b_{q_{j_\rho}^c}(-n_{q_{j_\rho}^c}-1/2)}\cdots b_{q_{s-2l}^c}(-n_{q_{s-2l}^c}-1/2)\one 
    \end{align*}

\end{cor}

\begin{proof}
    Apply the na\"ive vertex operator $\bar{Y}_W$ to Formula (\ref{exp-Delta-A-B}). The conclusion directly follows from the equality
    \begin{align*}
        & \bar{Y}_W(a_1^{(m_1)}(x)_V^- \cdots a_r^{(m_r)}(x)_V^- b_1(-n_1-1/2)\cdots b_s(-n_s-1/2)\one, y)\\
        = \ & \nord a_1^{(m_1)}(y+x) \cdots a_r^{(m_r)}(y+x)b_1^{(n_1)}(y)\cdots b_s^{(n_s)}(y)\nord.
    \end{align*}
    The proof of the equality is verbatim as that for $V$ in Corollary 4.7 in \cite{FQ}. We should not repeat the details. 
\end{proof}

\subsection{Verifying the axioms}

\begin{thm}
    The vector space $W$ considered in Section 2, together with the vertex operator $Y_W$ constructed in Section 4, forms a canonicaly $\Z_2$-twisted $V$-module. 
\end{thm}

\begin{proof}
    We check the axioms. 
    \begin{enumerate}
        \item Axiom of grading: Clearly, $W$ is $\N$-graded with each homogeneous space $W_{(n)}$ spanned by 
        $$h_1(-m_1)\cdots h_r(-m_r)u, m_1, ..., m_r\in \N, m_1+\cdots + m_r = n. $$
        The grading is lower bounded by zero. The $\d$-grading condition follows directly from its definition. For the $\d$-commutator formula, first notice that from a verbatim argument in the proof of Theorem 3.12 in \cite{FQ}, the na\"ive vertex operator satisfies
        \begin{align*}
            [\d_W, \bar Y_W(v, x)] = \bar Y_W(\d_V v, x)+ x\frac{d}{dx}\bar{Y}_W(v, x). 
        \end{align*}
        Equivalently, for every basis element $v = a_1(-m_1-1/2)\cdots a_r(-m_r-1/2)\one$ of $V$, from the coefficient of $x^{-n-1}$, we have 
        $$\wt (\bar{Y}_W)_n(v) = \wt v - n - 1 = m_1 + \cdots + m_r + \frac r 2 - n -1. $$
        To study the coefficient of $x^{-n-1}$ of the series $[\d_W, \bar{Y}_W(\exp(\Delta(x))v, x)]$, we first note that for every fixed $1\leq i_1 < \cdots < i_{2t}  \leq r$, 
        \begin{align*}
            & \wt (\bar{Y}_W)_n(S_{a_1, ..., a_r}^{(m_1, ..., m_r)}(1, ..., \widehat{i_1}, ..., \widehat{i_t}, ..., r)) \\
            = \ & m_1 + \cdots + \widehat{m_{i_1}} + \cdots + \widehat{m_{i_{2t}}} + \cdots + m_r + \frac {r-2t} 2 - n - 1. 
        \end{align*}
        Thus, the coefficient of $x^{-n-1}$ in 
        $$x^{-m_{i_1} - \cdots - m_{i_{2t}} - t}[\d_W, \bar{Y}_W(S_{a_1, ..., a_r}^{(m_1, ..., m_r)}(1, ..., \widehat{i_1}, ..., \widehat{i_t}, ..., r), x)]$$ 
        is precisely the coefficient of $x^{-n+m_{i_1} + \cdots + m_{i_{2t}} + t-1}$ of $[\d_W, \bar{Y}_W(S_{a_1, ..., a_r}^{(m_1, ..., m_r)}(1, ..., \widehat{i_1}, ..., \widehat{i_t}, ..., r), x)]$, which is 
        $$\left(m_1 + \cdots + m_r + \frac r 2 - n - 1\right) (\bar Y_W)_{n-m_{i_1 - \cdots - m_{i_{2t}}-t}}(S_{a_1, ..., a_r}^{(m_1, ..., m_r)}(1, ..., \widehat{i_1}, ..., \widehat{i_t}, ..., r))$$
        Supplementing the coefficient $T_{a_1, ..., a_r}^{m_1, ..., m_r}(i_1, ..., i_{2t})$, and taking the sum over all $i_1, ..., i_{2t}$ and all $t\in \N$, we recover the coefficient of $x^{-n-1}$ in $[\d_W, Y_W(v, x)]$, which is 
        $$\left(m_1 + \cdots + m_r + \frac r 2 - n - 1\right) (Y_W)_n(v).$$
        Thus, 
        $$\wt (Y_W)_n(v) = \wt v - n -1. $$
        So the $\d$-commutator formula holds for $Y_W$.         
        \item The identity axiom:  It follows directly the definition of the na\"ive operator and the fact that $\Delta(x) \one = 0$. 
        \item $D$-derivative formula: We use the iterate formula $Y_W(Y_V(u, x)v, y)$ with $v = \one$. So for $u=a_1(-m_1-1/2)\cdots a_r(-m_r-1/2)\one$, with the creation property of $Y_V$
        \begin{align*}
            & Y_W(Y_V(a_1(-m_1-1/2)\cdots a_r(-m_r-1/2)\one, x)\one, y) \\
            = \ & Y_W(e^{xD}a_1(-m_1-1/2)\cdots a_r(-m_r-1/2)\one, y) \\
            = \ & \sum_{k=0}^\infty \sum_{1\leq p_1 < \cdots < p_{2k} \leq r} T_{a_1, ..., a_r}^{m_1, ..., m_r}(p_1, ..., p_{2k})\iota_{yx}(y+x)^{-m_{p_1}-\cdots -m_{p_{2k}}-k}\\
            & \qquad \qquad \qquad \qquad \cdot \nord a_{p_1^c}^{(m_{p_1^c})}(y+x) \cdots a_{p_{r-2k}^c}^{(m_{p_{r-2k}^c}})(y+x)\nord 
        \end{align*}
        Taking the coefficient of $x$, i.e., taking $\Res_x x^{-2}$, we obtain 
        \begin{align*}
            & Y_W(D a_1(-m_1-1/2)\cdots a_r(-m_r-1/2)\one, y) \\
            = \ & \sum_{k=0}^\infty \sum_{1\leq p_1 < \cdots < p_{2k} \leq r} T_{a_1, ..., a_r}^{m_1, ..., m_r}(p_1, ..., p_{2k}) \frac{\partial}{\partial y}y^{-m_{p_1}-\cdots -m_{p_{2k}}-k}\\
            & \qquad \qquad \qquad \qquad \cdot \nord a_{p_1^c}^{(m_{p_1^c})}(y+x) \cdots a_{p_{r-2k}^c}^{(m_{p_{r-2k}^c}})(y+x)\nord \\
            & + \sum_{k=0}^\infty \sum_{1\leq p_1 < \cdots < p_{2k} \leq r} T_{a_1, ..., a_r}^{m_1, ..., m_r}(p_1, ..., p_{2k}) y^{-m_{p_1}-\cdots -m_{p_{2k}}-k}\\
            & \qquad \qquad \qquad \qquad \cdot \sum_{i=1}^{r-2k} \nord a_{p_1^c}^{(m_{p_1^c})}(y) \cdots \frac{\partial}{\partial y}a_{p_i^c}^{(m_i^c)}(y)\cdots a_{p_{r-2k}^c}^{(m_{p_{r-2k}^c}})(y)\nord \\
            = \ & \frac{\partial}{\partial y} Y_W(a_1(-m_1-1/2)\cdots a_r(-m_r-1/2)\one, y)
        \end{align*}
        \item Weak associativity: From Corollary \ref{Product-thm}, we see that for $r, s\in \N, a_1, ..., a_r, b_1, ..., b_s\in \h, m_1, ..., m_r, n_1, ..., n_s\in \N$, 
        \begin{align*}
            & Y_W(a_1(-m_1-1/2)\cdots a_r(-m_r-1/2)\one, x+y) Y_W(b_1(-n_1-1/2)\cdots b_s(-n_s-1/2)\one, y)\\
            = \ & \sum_{k=0}^\infty \sum_{1\leq p_1 < \cdots < p_{2k}\leq r}T_{a_1, ..., a_r}^{m_1, ..., m_r}(p_1, ..., p_{2k})\iota_{xy}(x+y)^{-m_{p_1}-\cdots -m_{p_2k}-k}\\
            & \cdot \sum_{l=0}^\infty \sum_{1\leq q_1 < \cdots < q_{2l}\leq s}T_{b_1, ..., b_s}^{n_1, ..., n_s}(q_1, ..., q_{2l}) y^{-n_{q_1}-\cdots -n_{q_{2l}}-l}\\
            & \cdot \bar{Y}_W(a_{p_1^c}(-m_{p_1^c}-1/2)\cdots a_{p_{r-2k}^c}(-m_{p_{r-2k}^c}-1/2)\one) \\
            & \cdot \bar{Y}_W(b_{q_1^c}(-n_{q_1^c}-1/2)\cdots b_{q_{s-2l}^c}(-n_{q_{s-2l}^c}-1/2)\one) \\
            = \ & \sum_{k=0}^\infty \sum_{1\leq p_1 < \cdots < p_{2k}\leq r}T_{a_1, ..., a_r}^{m_1, ..., m_r}(p_1, ..., p_{2k})\iota_{xy}(x+y)^{-m_{p_1}-\cdots -m_{p_2k}-k}\\
            & \cdot \sum_{l=0}^\infty \sum_{1\leq q_1 < \cdots < q_{2l}\leq s}T_{b_1, ..., b_s}^{n_1, ..., n_s}(q_1, ..., q_{2l}) y^{-n_{q_1}-\cdots -n_{q_{2l}}-l}\\
             & \cdot \sum_{\rho=0}^{\min(r,s)}\sum_{\substack{ 1 \leq i_1 < \cdots < i_\rho\leq r\\ 1 \leq j_1 < \cdots < j_\rho \leq s}}(-1)^{i_1+\cdots + i_\rho + j_1+\cdots + j_\rho }(-1)^{r\rho + \rho(\rho+1)/2}\nonumber \\
            & \hspace{9.7 em}\cdot \begin{vmatrix}
            (a_{p_{i_1}^c}, b_{q_{j_1}^c})\iota_{xy}g_{m_{p_{i_1}^c}n_{q_{j_1}^c}}(x+y, y) & \cdots & (a_{p_{i_1}^c}, b_{q_{j_\rho}^c})\iota_{xy}g_{m_{p_{i_1}^c}n_{q_{j_\rho}^c}}(x+y, y)\\
            \vdots & & \vdots\nonumber  \\
            (a_{p_{i_\rho}^c}, b_{q_{j_1}^c})\iota_{xy}g_{m_{p_{i_\rho}^c}n_{q_{j_1}^c}}(x+y, y) & \cdots & (a_{p_{i_\rho}^c}, b_{q_{j_\rho}^c})\iota_{xy}g_{m_{p_{i_\rho}^c}n_{q_{j_\rho}^c}}(x+y, y)
            \end{vmatrix}\nonumber \\
        & \hspace{10 em}\cdot \nord a_{p_1^c}^{(m_{p_1^c})}(x+y) \cdots \reallywidehat{a_{p_1^c}^{(m_{p_1^c})}(x+y)} \cdots \reallywidehat{a_{p_{i_\rho}^c}^{(m_{p_{i_\rho}^c})}(x+y)} \cdots a_{p_{r-2k}^c}^{(m_{p_{r-2k}^c})}(x+y)\nonumber  \\
        & \hspace{11 em}\cdot b_{q_1^c}^{(n_{q_1^c})}(y) \cdots \reallywidehat{b_{q_{j_1}^c}^{(n_{q_{j_1}^c})}(y)} \cdots \reallywidehat{b_{q_{j_\rho}^c}^{(n_{q_{j_\rho}^c})}(y)} \cdots b_{q_{s-2l}^c}^{(n_{q_{s-2l}^c})}(y)\nord 
        \end{align*}
        which differs from the iterate formula in Corollary \ref{Iterate-thm} by an expansion of zero. Moreover, if we act the product and the iterate on an element $h_1(-i_1)\cdots h_j(-i_j)u$ in the spanning set of $W$, then the expansion of zero can be removed by the multiplication of $(x+y)^P$, where $P$ may be chosen as any integer that is greater than 
        $$m_1 + \cdots + m_r + \frac r 2 + i_1 + \cdots + i_j. $$
        Thus the weak associativity holds. 
    \end{enumerate}
\end{proof}

\begin{rema}\label{D-comm-fail}
    We should show that a $D_W$ satisfying the $D$-commutator formula does not exist on $W$. Suppose there exists such an operator, then necessarily, for $a\in \h$, 
    \begin{align*}
        [D_W, a(x)] = \frac{d}{dx} a(x) \Rightarrow [D_W, a(n)] = (-n+1/2) a(n-1). 
    \end{align*}
    Now consider 
    \begin{align*}
        Y_W(a_1(-1/2)a_2(-1/2)\one, x) = \bar{Y}_W(a_1(-1/2)a_2(-1/2)\one, x) + (-1)^{1+2}(a_1, a_2)C_{00}x^{-1} \bar{Y}_W(\one, x)
    \end{align*}
    The commutator of $D_W$ with the second term is zero. Thus
    \begin{align}
        [D_W,  Y_W(a_1(-1/2)a_2(-1/2)\one, x)] & = [D_W, \nord a_1(x)a_2(x)\nord] \label{D-comm-fail-1}
    \end{align}
    If $D$-commutator formula holds, then it must be equal to 
    $$\frac d {dx} \nord a_1(x)a_2(x)\nord = \sum_{n_1, n_2\in \Z} \nord a_1(n_1)a_2(n_2)\nord (-n_1-n_2-2)x^{-n_1-n_2-2}$$
    In particular, only $\nord a_1(0)a_2(0)\nord$ may appear in the series. Clearly this is not the case in (\ref{D-comm-fail-1}), since the coefficient of $x^{-2}$ contains terms like 
    \begin{align*}
        & [D_W, \nord a_1(1)a_2(0) + a_1(0)a_2(1)\nord] \\
        = \ & [D_W, -a_2(0)a_1(1) + a_1(0)a_2(1)] \\
        = \ & \frac 1 2 a_2(-1)a_1(1) - \frac 1 2 a_2(0)a_1(0) + \frac 1 2 a_1(-1)a_2(1) - \frac 1 2 a_1(0)a_2(0)\\
        = \ & -\frac 1 2 \nord a_1(1)a_2(-1)\nord  + \frac 1 2 \nord a_1(-1)a_2(1)\nord - \frac 1 2 \nord a_2(0)a_1(0)\nord - \frac 1 2 \nord a_1(0)a_2(0)\nord 
    \end{align*}
    There is no way to get rid of $\nord a_2(0) a_1(0)\nord$ when $a_1(0)$ and $a_2(0)$ has no relations. 
\end{rema}

\noindent {\small \sc Department of Mathematics, University of Denver, Denver, CO  80210, USA}

\noindent {\em E-mail address}: fei.qi@du.edu | fei.qi.math.phys@gmail.com


\begin{thebibliography}{KWAK2}

\bibitem[FFR]{FFR} Alex J. Feingold, Igor B. Frenkel, John F. X. Ries, \textit{Spinor Construction of Vertex Operator Algebras, Triality, and $E_8^{(1)}$}, Contemporary Mathematics, Vol. 121, American Mathematical Society, Providence, RI, 1991. 












\bibitem[FLM]{FLM} I. Frenkel, J. Lepowsky and A. Meurman, {\it Vertex operator algebra and the monster},  Pure and Appl. Math., 134, Academic Press, New York, 1988.

\bibitem[FQ]{FQ} Francesco Fiordalisi, Fei Qi, Fermionic construction of the $\frac \Z 2$-graded meromorphic open-string vertex algebra and its $\Z_2$-twisted module, I, to appear. 

\bibitem[G]{G-PBW} Paul Garrett, Poincare-Birkhoff-Witten Theorem, https://www-users.cse.umn.edu/~garrett/m/algebra/pbw.pdf




\bibitem[H1]{H-MOSVA}
Yi-Zhi Huang, Meromorphic open string vertex algebras, \textit{J. Math. Phys.} \textbf{54} (2013), 051702. 

\bibitem[H2]{H-MOSVA-Riemann}
Yi-Zhi Huang, Meromorphic open-string vertex
algebras and Riemannian manifolds, arXiv:1205.2977.

\bibitem[H3]{H-Two-Constructions} Yi-Zhi Huang, Two constructions of grading-restricted vertex (super)algebras, \textit{J. Pure Appl. Alg.} \textbf{220} (2016), 3628-3649.

\bibitem[H4]{H-Twist} Yi-Zhi Huang, Twist vertex operators for twisted modules, \textit{J. Alg.} \textbf{539} (2019), 53--83.









\bibitem[LL]{LL} 
James Leposwky, Haisheng Li, \textit{Introduction to vertex operator algebras
and their representations}, Progress in Mathematics, 227, Birkh\"auser, Boston, 2004. 








\bibitem[Q1]{Q-Mod} Fei Qi, On modules for meromorphic open-string vertex algebras, \textit{Journal of Mathematical Physics} \textbf{60}, 031701 (2019)

\bibitem[Q2]{Q-2d-space-form} Fei Qi, Meromorphic open-string vertex algebras and modules over two-dimensional orientable space forms, \textit{Lett. Math. Phys. }\textbf{111} Article 27 (2021), 1--54. 

\bibitem[Q3]{Q-Cov} Fei Qi, Covariant derivatives of eigenfunctions along parallel tensors over space forms and a conjecture motivated by the vertex algebraic structure, 
\textit{J. Noncomm. Geom.} \textbf{16} no.2 (2022), 717--759. 


\bibitem[T]{T} Haruo Tsukada, Vertex operator superalgebras, \textit{Comm. Alg.}, \textbf{18}:7 (1990), 2249--2274. 



\end{thebibliography}
\end{document}